\documentclass[reqno]{amsart}

\title[]{Modelling Anomalous Diffusion: The Role of CTRWs and Non-Local Dynamics}

\author{Ivan Bio\v{c}i\'{c}}
\author{Bruno Toaldo}
\date{\today}

\thanks{The authors acknowledge financial support under the National Recovery and Resilience Plan (NRRP), Mission 4, Component 2, Investment 1.1, Call for tender No. 104 published on 2.2.2022 by the Italian Ministry of University and Research (MUR), funded by the European Union – NextGenerationEU– Project Title “Non–Markovian Dynamics and Non-local Equations” – 202277N5H9 - CUP: D53D23005670006 - Grant Assignment Decree No. 973 adopted on June 30, 2023, by the Italian Ministry of University and Research (MUR).}

\thanks{The author B. Toaldo would like to thank the Isaac Newton Institute for Mathematical Sciences, Cambridge, for support and hospitality during the programme Stochastic systems for anomalous diffusion, where work on this paper was undertaken. This work was supported by EPSRC grant EP/Z000580/1.}

\thanks{The author I. Biočić acknowledges financial support by the European Union – NextGenerationEU through the National Recovery and Resilience Plan 2021-2026 Institutional grant of University of Zagreb Faculty of Science (IK IA 1.1.3. Impact4Math), as well as the support by Croatian Science Foundation through the project IP-2025-02-8793.}

\thanks{The authors are very grateful to Giacomo Ascione, Enrico Scalas, and Francesco Virgili for their valuable comments, which have considerably improved a previous version of the manuscript.
}

\usepackage{amssymb, amsmath, amsfonts}
\usepackage{natbib}
\usepackage{lscape, upgreek}
\usepackage{aligned-overset}
\usepackage{color}
\usepackage{comment}
\usepackage{mathtools}
\usepackage{dsfont}
\usepackage{mathrsfs}
\usepackage{bm}
\usepackage{graphicx}
\usepackage{wrapfig}
\usepackage{caption}
\usepackage{tikz}
\usepackage{float}
\usetikzlibrary{arrows,decorations.markings,arrows.meta}
\usetikzlibrary{calc}

\usepackage{cancel}

\usepackage{soul}

\DeclareMathAlphabet{\mathpzc}{OT1}{pzc}{m}{it}
\DeclareMathOperator\supp{supp}

\bibpunct{[}{]}{;}{n}{,}{,}
\usepackage{algorithm2e}
\usepackage{algpseudocode}
\RestyleAlgo{ruled}

\usepackage{bbm} %Indicator
\newtheorem{theorem}{Theorem}[section] %Theorems
\newtheorem{lemma}[theorem]{Lemma}
\newtheorem{definition}[theorem]{Definition}
\newtheorem{proposition}[theorem]{Proposition}
\newtheorem{corollary}[theorem]{Corollary}
\theoremstyle{remark}
\newtheorem{remark}[theorem]{Remark}
\newtheorem{example}[theorem]{Example}
\numberwithin{equation}{section} %# by Section
\usepackage{authblk}
\usepackage{natbib}
\usepackage{hyperref}
\hypersetup{
	colorlinks=true,
}
\usepackage{caption}
\usepackage{subcaption}
\usepackage{graphicx}
\usepackage[pagewise]{lineno}
\usepackage{float} %for fixing positions of algotithms and figures

\newcommand{\R}{\mathbb{R}}
\newcommand{\N}{\mathbb{N}}
\newcommand{\p}{\mathds{P}}
\newcommand{\ex}{\mathds{E}}
\newcommand{\Lpsi}{\psi(-\Delta)}
\newcommand{\LpsiD}{\Lpsi_{|D}}
\newcommand{\RR}{\mathbb{R}}
\newcommand{\EE}{\mathds{E}}
\newcommand{\PP}{\mathds{P}}
\newcommand{\NN}{\mathbb{N}}

\newcommand{\LL}{\mathcal{L}}

\newcommand{\BB}{\mathcal{B}}
\newcommand{\DD}{\mathcal{D}}
\newcommand{\dist}{\text{dist}}
\newcommand{\de}{\delta_D}
\newcommand{\wh}{\widehat}
\newcommand{\wt}{\widetilde}
\newcommand{\1}{\mathds{1}}

\DeclareMathAlphabet{\mathpzc}{OT1}{pzc}{m}{it}

\allowdisplaybreaks

\newtheorem*{assumption*}{\assumptionnumber}
\providecommand{\assumptionnumber}{}
\makeatletter
\newenvironment{assumption}[2]
{%
	\renewcommand{\assumptionnumber}{(\textbf{#1#2})}%
	\begin{assumption*}%
		\protected@edef\@currentlabel{(\textbf{#1#2})}%
	}
	{%
	\end{assumption*}
}
\makeatother

\begin{document}
	
	\maketitle
	
	\begin{abstract}
		These notes provide a self-consistent summary of the stochastic approach to anomalous diffusion based on Continuous Time Random Walks (CTRWs) and their scaling limits. An introduction to CTRWs and their relationship to the theory of semi-Markov processes is provided. A general technique to study scaling limits of CTRWs is then described, and the semi-Markov property of the limiting processes is discussed. With this at hand, the connection of limit processes with non-local (fractional-type) equations is introduced, and the most recent (and general) contributions, going far beyond fractional equations, are also described. Indeed, the theory presented here includes very general non-local evolution equations as abstract Cauchy problems as well as pointwise non-local fractional diffusion equations in bounded domains.
	\end{abstract}
	
	{\hypersetup{linkcolor=black}
		% or \hypersetup{linkcolor=black}, if the colorlinks=true option of hyperref is used
		\tableofcontents
	}

	\section{Anomalous diffusion and fractional kinetic}

	The study of complex systems has gained a prominent place in physics, due to their characteristic structural and dynamical richness. Such systems are typically composed of diverse interacting elements and often exhibit unpredictable or anomalous time evolution \cite{mandelbrot1982, havlin2002}. These features are common across numerous physical, chemical, biological, and even ecological systems, including glasses, polymers, proteins, and living organisms \cite{bouchaud1990, metzler2000, klages2008}.
	
	Diffusion models can be viewed as tools for modelling and generating complex systems, because they capture how simple local interactions accumulate over time to produce rich, emergent global behavior. Specifically, diffusion models describe the motion of particles, or larger objects, subject to random effects. Anomalous diffusion refers to a class of processes whose behavior significantly departs from that predicted by the simplest diffusion models. A typical characteristic of anomalous diffusion is the non-linear time dependence of the mean squared displacement (MSD): a classical diffusion models the motion of an object with position $X(t)$, $t \geq 0$, such that the MSD, i.e., the quantity $\Delta^2X(t) \coloneqq \mathds{E}^x\left\| X(t)-x \right\|^2$, increases linearly in time (here $\mathds{E}^x$ denotes the expectation under a measure $\PP^x$ such that $\PP^x(X(0)=x)=1$, and $\left\| \cdot \right\|$ is the Euclidean norm). Deviations from this behavior are typical of anomalous diffusion. The prototype of non-linear behavior is the power-law $Ct^\alpha$, for $\alpha \neq 1$, which is typical of fractional kinetic processes \cite{metzler2000}. In general, when the MSD grows slower (faster) than linearly, the process is said to be subdiffusive (superdiffusive).

	The first indications of anomalous diffusion arose from Richardson’s 1926 work on turbulence \cite{richardson1926}. More structured theoretical treatments emerged in the 1970s with Scher and Montroll’s study of charge transport in amorphous semiconductors, employing the continuous time random walk (CTRW) framework \cite{scher1975}.
	
	Since then, anomalous diffusion has been experimentally verified in systems ranging from porous media, polymer chains, and cellular environments to turbulent flows and biological systems \cite{barkai0, barkai2000, golding2006, sokolov2005}. Subdiffusion, in particular, appears in polymer dynamics, protein transport, and {nuclear magnetic resonance } studies in disordered materials, whereas superdiffusion is often linked to Lévy statistics observed in turbulent and biological systems \cite{zumofen1993, bouchaud1990, weiss1994}.
	
	Single-particle mathematical models for anomalous diffusion often include memory effects, environmental inhomogeneities, or constraints on particle movement. Typical examples include boundary reflections \cite{burdzy2015reflections, mijatovicreflection}, interactions with the past trajectory \cite{amit1983asymptotic, toth1995true, toth2012superdiffusive}, and trapping effects, i.e., when heterogeneous media slow down diffusion due to random traps or spatial inhomogeneities. see \cite{arous2006dynamics, benarous}. Chaotic systems (Hamiltonian chaos) can exhibit anomalous transport behavior; this case is typically connected to fractional kinetic equations (FKEs), the prototype being the time-fractional diffusion equation 
	\begin{equation}
		\partial_t^\alpha q(t,x)=\frac{1}{2}\partial_x^2q(t,x), \qquad x \in \mathbb{R},\, t>0,
		\label{fde}
	\end{equation}
	and fractional kinetic processes (FKPs) connected to FKEs. In \eqref{fde}, the operator on the left-hand side is defined by
	\begin{align}
		\partial_t^\alpha f(t) \, = \, \frac{1}{\Gamma(1-\alpha)} \partial_t \int_0^t (f(s)-f(0)) (t-s)^{-\alpha} \, ds,
	\end{align}
	for any function $f$ such that integration and differentiation make sense.
	
	In \cite{zaslavsky3}, the author finally established the link between fractional kinetic equations (FKEs) and models of non-moderate Hamiltonian chaos (see also \cite{zaslavsky4, zaslavsky5} for further developments). It is interesting to observe that in this setting FKEs emerge from a scaling limit procedure. This is also the case, for instance, for the Hamiltonian systems describing cellular flows. Hairer et al. \cite{hairer1} provided a stochastic representation result for the limit of an averaging-homogenization problem; the link between such a problem and the FKE \eqref{fde} has been better highlighted in Hairer et al. \cite{hairer2}. Fractional models of anomalous diffusion arise also in continuous time finance (e.g., \cite{scalas2000fractional, mainardi2000fractional, gorenflo2001fractional}); also, many other modelling problems can be dealt with using fractional models (see, e.g. \cite{baleanu2012fractional, meerschaert2019stochastic} for several stochastic models related to fractional calculus).
	
	\subsection{Structure of the notes and leading idea}
	In these notes, we look at anomalous diffusion arising as scaling limits of CTRWs. Among the existing approaches, this is the most natural one from the purely probabilistic point of view.
	One of the most classical theories leading to diffusion is based on Donsker's well-known approximation of a Brownian motion. If one takes a random walk in $\mathbb{R}^d$, call it $S_n$, $n \in \mathbb{N}$, with $S_0=0$, a.s., and with i.i.d. jumps having unit variance and zero mean, then, as $c \to +\infty$,
	\begin{align}
		\left(  u \mapsto \frac{S_{[cu]}}{\sqrt{c}}, u \geq 0 \right) \Longrightarrow \left( u \mapsto B(u), u \geq 0 \right),
	\end{align}
	weakly in the space of càdlàg functions $D[0, +\infty)$ endowed with the Skorokhod $J$ topology\footnote{This topology is also quite often called $J_1$ topology.} \cite{billingsley2013convergence, jacod2013limit}. Then, the Brownian motion is Markovian and represents the stochastic ingredient for the heat equation. In these notes, we explain how to extend the limit theory considering CTRWs whose jumps and waiting times may be dependent but jointly form a Markov chain, converging to processes that are not Markovian and that exhibit anomalous diffusive behavior. Then we will discuss how these processes are naturally connected to non-local evolution equations, involving non-locality in both time and space.
	
	The notes are organized as follows.
	In the next Section \ref{ss:CTRW-intro} we introduce CTRWs and show their semi-Markov property. In Section \ref{sec_limit} we  explain how to study scaling limits of CTRWs and then, in Section \ref{semi_limit} we explain the semi-Markov property for the limit processes. In Section \ref{sec3_non-local} we establish the theory of (non-local) evolution equations for the limit processes. Finally, in Section \ref{section_killed}, we study more specific non-local equations, in bounded domains, which are related to killed processes (upon exiting the domain) arising from CTRWs limits.
	
	\section{Introduction to CTRWs and the semi-Markov property}\label{ss:CTRW-intro}
	We formalize here the construction and the assumptions of CTRWs.
	Let $(\Omega, \mathcal{F}, \PP^{(x,s)})$, $(x,s) \in \mathbb{R}^d \times \mathbb{R}$, be a family of probability spaces, and, on these spaces, let $(J_n, W_n)$, $n \in \mathbb{N}$, be a sequence of random vectors with values in $\mathbb{R}^d \times (0,+\infty)$. The random variables $J_n$, $n \in \mathbb{N}$, represent the (numbered) jumps of the CTRW, while $W_n$, $n \in \mathbb{N}$, represent the waiting times between jumps. To proceed with the construction, define $T_0$ such that $\PP^{(x,s)} (T_0=s)=1$, $T_n \coloneqq T_0+\sum_{i=1}^n W_i$, $n \in \mathbb{N}$, and
	\begin{align}\label{1047-inv}
		N(t) \coloneqq \max \{ n \in \mathbb{N}: T_n \leq t \}, \quad t\in\RR,
	\end{align}
	with $N(t)=0$ if $t<T_1$. Also, define $S_0$ such that $\PP^{(x,s)} (S_0=x)=1$, and $S_n \coloneqq S_0+\sum_{i=1}^nJ_i$, $n \in \mathbb{N}$. Note that $N(t)$ is the inverse of $T_n$, $n\in\N$. The CTRW is defined as the continuous-time process $Y(t),\, t \in \RR,$ where
	\begin{align}
		Y(t) \coloneqq S_{N(t)}.
		\label{defctrw}
	\end{align}
	It follows that
	\begin{align}
		Y(t) = S_n, \qquad T_n \leq t < T_{n+1}.
		\label{stepreprctrw}
	\end{align}
	
	Another relevant process for the discussion below is the process that tracks the time of the last jump before time $t\in\RR$, i.e.
	\begin{align}\label{last-jump}
		T(t) \coloneqq  T_{N(t)},\qquad t\in\RR.
	\end{align}
	
	These processes differ from continuous-time Markov chains since the waiting times between the jumps are not necessarily exponentially distributed. Here, an accumulation of infinitely many jumps in finite time is excluded by assumption, i.e., we always assume that
	\begin{align}
		\sum_n W_n = +\infty, \quad \PP^{(x,s)}\text{--a.s.},
		\label{nonexplosion}
	\end{align}
	for all $(x,s)\in \R^d\times \R$.
	
	The typical assumptions on $(J_n, W_n)$ are such that the process $(S_n, T_n)$, $n \in \mathbb{N}\cup\{0\}$, is a time-homogeneous Markov chain on $\mathbb{R}^d \times \mathbb{R}$, tracking the position of a particle and the time of the $n$-th jump. This means that the information on $(J_n, W_n)$, i.e., the $n$-th space-time jump, passes through the current position $(S_{n-1},T_{n-1})$. In other words, to construct a CTRW of this kind, it is enough to specify, for any $n\in \NN$, the joint distribution of $(J_n,W_n)=(S_n-S_{n-1},T_n-T_{n-1})$ conditionally on the random variables $S_{n-1}$ and $T_{n-1}$. This perspective aligns with the approach taken by Gikhman and Skorokhod \cite[Chapter 3, Section 3]{Gihman} in their introduction to semi-Markov processes.
	
	Indeed, in \cite[Chapter 3, Section 3]{Gihman} the authors considered a sequence of uniform random variables on $(0,1)$, say $\eta_n$, $n \in \mathbb{N}$, jointly independent of some Markov chain $S_n$, $n \in \mathbb{N}$, and then they defined a function $\varphi_{x,y}(t)$, $x,y \in \mathbb{R}^d$, such that $\varphi_{x,y}(\eta)$, where $\eta$ has the uniform distribution on $(0,1)$, has the cumulative distribution function $F_{x,y}(t)$. With this at hand, one can set $W_n \coloneqq \varphi_{S_{n-1}, S_n}(\eta_n)$, $n \in \mathbb{N}$. Obviously, the random variables $W_n$, $S_{n-1}$, and $S_n-S_{n-1}$ are not necessarily independent. To obtain a CTRW (in the spirit described above) from this construction, it is more suitable to consider functions $\varphi_{(x,y)}(\eta_n) \equiv \widetilde\varphi_{(x,y-x)}(\eta_n)$, where $\widetilde\varphi(x,h)=\varphi(x,x+h)$, i.e., the dependence of $W_n$ on the next position is due to the dependence on the forthcoming jump. From this, we easily see that $(S_n, W_n)$, $n \in \mathbb{N}$, is a discrete-time Markov chain, and so is $(S_n,T_n)_n$, $n\in \NN$. With this construction, the transition probabilities of this chain depend only on the current position $S_n$, since the waiting time $W_{n+1}$ is independent of $T_n$. The corresponding CTRW is called homogeneous in this case. This approach inspired by \cite{Gihman} is useful as it allows us to apply the theory of semi-Markov processes developed therein to CTRWs. In particular, we can draw upon \cite[Lemma 2, Section 3, Chapter 3]{Gihman}, and other results in the same reference, to prove a Markov embedding for $Y(t)$ and renewal equations for the corresponding Markov process: this property, stated in the next theorem, is called the semi-Markov property in the sense of Gikhman and Skorokhod. Renewal equations for some more specific CTRWs $Y(t)$ (the position coordinate alone) were already obtained in the literature (e.g., \cite{germano2009stochastic}).

	\begin{theorem}
		\label{thmGS}
		Let $Y(t)$, $t\in\RR$, be a CTRW as in \eqref{defctrw} and suppose that $(S_n, T_n)$, $n\in \NN \cup \{0\}$, is a Markov chain such that \eqref{nonexplosion} holds. Assume further that, for all $(z,r) \in \mathbb{R}^d \times \mathbb{R}$, $(z^\prime,r^\prime) \in \mathbb{R}^d \times \mathbb{R}$ and $(x,s) \in \mathbb{R}^d \times \mathbb{R}$, it holds that
		\begin{align}\label{semi-jump}
			\mathds{P}^{(z,r)} \left( S_{n+1} \in dy, W_{n+1} \in dw \mid S_n = x, T_n = s \right) \, = \, &\mathds{P}^{(z^\prime,r^\prime)} \left( S_{n+1} \in dy, W_{n+1} \in dw \mid S_n = x \right) \notag \\
			= \, & \mathds{P}^{(x,0)} \left( S_1 \in dy, W_1 \in dw \right) \notag  \\
			\eqqcolon \, & \mu_x(dy,dw).
		\end{align}
		Define, for $t \geq 0$,
		\begin{align}
			\xi(t) \, = \, t-T_n, \qquad T_n\leq t < T_{n+1},
		\end{align}
		i.e. $\xi(t) = t- T(t)$.
		Then the process $(Y(t), \xi (t))$, $t \geq 0$, is a time-homogeneous Markov process, under the measures $\PP^{(x,s)}$, $(x,s)\in \RR^d\times (-\infty,0]$, and the filtration $\mathcal{F}_t$, $t \geq 0$, generated by $N(s), S_0, \dots, S_{N(s)}, T_0, \dots, T_{N(s)}$, $s \le t$. Moreover, its semigroup $P_t$, $t \geq 0$, satisfies for $(x,s)\in \RR^d\times [0,+\infty)$
		\begin{align}
			P_tf(x,s) \, = \, f(x,s+t) \, \frac{\mu_x(\mathbb{R}^d, (t+s, +\infty))}{\mu_x(\mathbb{R}^d, (s, +\infty))} + \int_{(s,s+t]} \int_{\mathbb{R}^d} P_{t+s-u}f(y,0)   \frac{\mu_x(dy, du)}{\mu_x(\mathbb{R}^d, (s, +\infty))},
			\label{transition_step}
		\end{align}
		and 
		\begin{align}
			P_tf(x,0) = \mathds{E}^{(x,0)} f(Y(t), \xi(t)), \quad x\in \RR^d,
			\label{semizero}
		\end{align}
		for any function $f$ which is bounded and Borel measurable ($f \in \mathcal{B}_b (\mathbb{R}^d \times [0, +\infty))$).
	\end{theorem}
	\begin{proof}
		Note that the property \eqref{semi-jump} implies that we can use the construction of jumps $W_n=\varphi_{S_n,S_{n-1}}(\xi_n)$ as in Gikhman and Skorokhod \cite{Gihman}, described in the paragraph above, using the classical construction as in, e.g., \cite[Theorem 10.7.2]{Bogachev-II}. By using this, the Markov property follows from \cite[Lemma 2, Section 3, Chapter 3]{Gihman}, while the transition probabilities \eqref{transition_step} come from \cite[Eq. (33), Section 3, Chapter 3]{Gihman}. 
		For the reader's convenience, we provide here the details of the proof, using a slightly different technique that is easier to read.
		
		Take $t,T>0$, $f \in \mathcal{B}_b (\mathbb{R}^d \times [0, +\infty))$, and observe that by the tower property
		\begin{align}
			& \mathds{E}^{(x,s)} [ f(Y(T+t), \xi(T+t)) \mid \mathcal{F}_T ] \notag \\
			= \, & \mathds{E}^{(x,s)}\left[\mathds{E}^{(x,s)} \left[ f(Y(T+t), \xi(T+t)) \mid \mathcal{F}_T, S_{N(T)+1}, T_{N(T)+1} \right] \mid \mathcal{F}_T \right] \notag \\
			\begin{split}\label{2terms}
				= \, & \mathds{E}^{(x,s)}\left[\mathds{E}^{(x,s)} \left[ \1_{[N(T+t) = N(T)]}f(Y(T+t), \xi(T+t)) \mid \mathcal{F}_T, S_{N(T)+1}, T_{N(T)+1} \right] \mid \mathcal{F}_T \right]\\
				+ \, & \mathds{E}^{(x,s)}\left[\mathds{E}^{(x,s)} \left[ \1_{[N(T+t) > N(T)]}f(Y(T+t), \xi(T+t))\mid \mathcal{F}_T, S_{N(T)+1}, T_{N(T)+1} \right] \mid \mathcal{F}_T \right].
			\end{split}
		\end{align}
		The first term in \eqref{2terms} is equal to
		\begin{align}
			& \mathds{E}^{(x,s)}\left[\mathds{E}^{(x,s)} \left[ \1_{[N(T+t) = N(T)]}f(Y(T+t), \xi(T+t)) \mid \mathcal{F}_T, S_{N(T)+1}, T_{N(T)+1} \right] \mid \mathcal{F}_T \right]\notag\\
			= \, &f(S_{N(T)}, \xi (T)+t) \mathds{P}^{(x,s)}\left(N(T+t) = N(T) \mid \mathcal{F}_T \right)\notag\\
			= \, &f(S_{N(T)}, \xi (T)+t) \mathds{P}^{(x,s)}\left( W_{N(T)+1} > t+T-T_{N(T)} \mid \mathcal{F}_T\right)\notag\\
			= \, & f(S_{N(T)}, \xi (T)+t) \frac{\mathds{P}^{(S_{N(T)},0)} (W_{1} > t+T-T_{N(T)} )}{\mathds{P}^{(S_{N(T)},0)} (W_{1} >  T-T_{N(T)}  )}\label{1105midstep}\\
			= \, &f(Y(T), \xi (T)+t) \frac{\mathds{P}^{(Y(T),0)} (W_{1} > t+ \xi(T)  )}{\mathds{P}^{(Y(T),0)} (W_{1} >  \xi(T)  )}\label{1823}
		\end{align}
		where  \eqref{1105midstep} follows from \eqref{semi-jump}, and the terms in it should be read as $\mathds{P}^{(S_{N(T)},0)} (W_{1} > t+T-T_{N(T)})=h(S_{N(T)},t+T-T_{N(T)})$ with $h(x,w)=\mathds{P}^{(x,0)} (W_{1} > w )$.
		
		For the second term in \eqref{2terms} we have
		\begin{align}
			&  \mathds{E}^{(x,s)}\left[\mathds{E}^{(x,s)} \left[ \1_{[N(T+t) > N(T)]}f(Y(T+t), \xi(T+t) \mid \mathcal{F}_T, S_{N(T)+1}, T_{N(T)+1} \right] \mid \mathcal{F}_T \right] \notag \\
			= \, & \mathds{E}^{(x,s)} \left[ \mathds{E}^{(S_{N(T)+1},0)} \big[f(Y( t+T-T_{N(T)+1}), \xi(t+T-T_{N(T)+1}))\big] \mathds{1}_{[T_{N(T)+1}\leq t+T]} \mid \mathcal{F}_T \right]\notag \\
			= \, & \mathds{E}^{(x,s)} \left[ \mathds{E}^{(S_{N(T)+1},0)} \big[f(Y( t+T-T_{N(T)+1}), \xi(t+T-T_{N(T)+1}))\big] \mathds{1}_{[T_{N(T)+1}\leq t+T]} \left|S_{N(T)}, T_{N(T)} , N(T)\right. \right].
			\label{1721}
		\end{align}
		Here, in the second line, we have used that, conditionally on $(S_0, T_0)$, $(S_1,T_1)$, $\dots$, $(S_{N(T)+1}, T_{N(T)+1})$, $N(T)$, the process $t\mapsto \big(Y(t+T), \xi (t+T)\big)$ is a function of the random variables $S_{N(T)+2}$, $S_{N(T)+3},\dots$, and $T_{N(T)+2}$, $T_{N(T)+3}, \dots$ in the same way the process $t\mapsto \big(Y(t+T-T_{N(T)+1}),\xi(t+T-T_{N(T)+1})\big)$ is a function of the random variables $S_1$, $S_2,\dots$ and $T_1$, $T_2,\dots$ conditionally on $S_0=S_{N(T)+1}$, $T_0=0$. The third line follows from \eqref{semi-jump}. Further, \eqref{1721} is equal to
		\begin{align}
			& \int_{T}^{t+T}  \int_{\mathbb{R}^d} \mathds{E}^{(y,0)} \big[f(Y(t+T-w), \xi(t+T-w))\big]\notag \\ & \qquad \qquad \times \mathds{P}^{{(x,s)}} \left( S_{N(T)+1} \in dy, T_{N(T)+1} \in dw \mid S_{N(T)}, T_{N(T)} , N(T) \right) \notag \\
			\begin{split}
				\label{906}
				&\quad = \, \int_{\xi(T)}^{t+\xi(T)}  \int_{\mathbb{R}^d} \mathds{E}^{(y,0)} f(Y(t+\xi(T)-w), \xi(t+\xi(T)-w))\\ &\quad \qquad \qquad \times \mathds{P}^{(x,s)} \left( S_{N(T)+1} \in dy, W_{N(T)+1} \in dw \mid S_{N(T)}, \xi(T),  N(T) \right).
			\end{split}
		\end{align}
		Note that the information given by conditioning on $\xi(T), N(T)$ is tantamount to saying that $W_{N(T)+1}>\xi(T)$, so the probability measure in \eqref{906} with respect to which we integrate becomes
		\begin{align}
			&\mathds{P}^{(x,s)} \left( S_{N(T)+1} \in dy, W_{N(T)+1} \in dw \mid S_{N(T)}, \xi(T),  N(T) \right)\\
			&\qquad=\frac{\mathds{P}^{(x,s)} \left( S_{N(T)+1} \in dy, W_{N(T)+1} \in dw \mid S_{N(T)},  N(T) \right)}{\mathds{P}^{(x,s)}\left( W_{N(T)+1} > \xi(T) \mid  S_{N(T)},N(T) \right)} \notag \\
			&\qquad=\frac{\mathds{P}^{(S_{N(T)},0)} \left( S_{1} \in dy, W_{1} \in dw  \right)}{\mathds{P}^{(S_{N(T)},0)}\left( W_{1} > \xi(T) \right)},
			\label{1822}
		\end{align}
		where again $\mathds{P}^{(S_{N(T)},0)}\left( W_{1} > \xi(T) \right)$ should be read as $h(S_{N(T)},\xi(T))$ for $h(x,w)=\mathds{P}^{(x,0)}\left( W_{1} > w \right)$. By putting pieces together, we obtain the time-homogeneous Markov property, since \eqref{1823} and \eqref{906} (together with \eqref{1822}) only depend on $S_{N(T)}=Y(T)$, $\xi(T)$ and $t$. By choosing $T=0$ in \eqref{1823} and \eqref{906} (together with \eqref{1822}), we get the renewal-type equation \eqref{transition_step}. 
		The equality \eqref{semizero} follows since, under $\mathds{P}^{(x,0)}$, one has that $Y(0)=x$ and $\xi(0)=0$.
	\end{proof}
	\begin{remark}[On the notion of semi-Markov processes]
		We note that the semi-Markov property in the sense of Gikhman and Skorokhod as discussed so far is a Markov embedding. This is, however, not the only possible notion of semi-Markovianity. A more general notion of semi-Markov processes is discussed in \cite[Chapter 3]{harlamov}. In full generality, take a filtered probability space $(\Omega, \mathcal{F}_\infty, \mathcal{F}_t, \mathds{P})$ and consider a filtration $\mathcal{G}_t$, $t \geq 0$, and a progressively measurable càdlàg process $\mathfrak{X}(t), t \ge 0$, adapted to $\mathcal{G}_t$. We say that a $\mathcal{G}$-stopping time $\tau$ is a $\mathcal{G}$-regenerative time for $\mathfrak{X}$ if, by denoting $\mathds{P}^y(B)=\mathds{P}(B \mid \mathfrak{X}(0)=y)$, for all $B \in \mathcal{F}_\infty$ it holds that
		\begin{equation*}
			\mathds{P}(\mathfrak{X}(\cdot+\tau) \in B \mid \mathcal{G}_\tau)=\mathds{P}^{\mathfrak{X}(\tau)}(\mathfrak{X}(\cdot)\in B), \quad \forall B \in \mathcal{F}_\infty,
		\end{equation*}
		on $\{\tau<\infty\}$.
		Take an open set $\mathcal{O}$ in the state space and denote by $\tau_\mathcal{O}$ the exit time from this set.
		The process $\mathfrak{X}$ is said to be semi-Markov if $\tau_\mathcal{O}$ is $\mathcal{G}$-regenerative for any choice of the open set $\mathcal{O}$.
		
		It turns out, however, that the notion of semi-Markov processes in the sense of Gikhman and Skorokhod is equivalent to this general one whenever the process $\mathfrak{X}$ is a step process (see \cite[Proposition 3.4]{harlamov}). It follows that CTRWs as in Theorem \ref{thmGS} are also semi-Markov in this sense.
	\end{remark}
	\begin{remark}
		We also observe that if we replace \eqref{semi-jump} with the assumption
		\begin{align}
			\mathds{P}^{(z,r)} \left( S_{n+1} \in dy, W_{n+1} \in dw \mid S_n=x, T_n=s  \right) \, = \, \mu_x(dy) \theta(x)e^{-\theta(x) w} dw,
		\end{align}
		for probability measures $\mu_x$, $x\in \R^d$, on $\R^d$, and a function $\theta:\R^d\to [0,+\infty)$, then the process $t\mapsto Y(t)$ is a (continuous-time) Markov chain on $\mathbb{R}^d$. Indeed, equations \eqref{1823} and \eqref{906} (together with \eqref{1822}) would depend only on $Y(t)$ by using a change of variables and the memorylessness of the exponential distribution.
	\end{remark}

	Limit processes of CTRWs are particularly relevant for modelling anomalous diffusion related to fractional kinetics. We will describe the theory in the next sections.
	
	It turns out, however, that one can modify definition \eqref{defctrw} to obtain a different process, called the Overshooting Continuous Time Random Walk (OCTRW), which is given by
	\begin{align}
		Y^+(t) = S_{N(t)+1},\quad t\in\RR.
		\label{def_octrw}
	\end{align}
	This process arises naturally as a model for anomalous diffusion (see, e.g., \cite{weron2010overshooting}) and its scaling limit will be different from that of the corresponding CTRW (which we make precise in the next section). For this process, the representation \eqref{stepreprctrw} reads
	\begin{align}
		Y^+(t) \, = \, S_{n+1}, \qquad T_n \leq t < T_{n+1}.
	\end{align}
	In this context, in contrast to \eqref{last-jump}, the process
	\begin{align}
		T^+(t) \coloneqq T_{N(t)+1},\quad t\in\RR,
		\label{tnt1}
	\end{align}
	which tracks the time of the forthcoming jump will be relevant. Note that 
	\begin{align}
		N(t)+1=\min\{n\in\N: T_n>t\},
	\end{align}
	i.e., $N(t)+1$ can be considered another notion of the inverse of the coordinate $T_n$, $n\in \N$, of our Markov chain, in contrast to $N(t)$, see \eqref{1047-inv}.
	
	The process $Y^+(t)$, $t\in\RR$, has a different dependence between jumps and waiting times: indeed, for $Y^+$, the waiting times between the jumps depend on the \emph{previous} jumps. Hence, the semi-Markov property in the sense of Gikhman and Skorokhod can easily fail in this case. The following heuristic arguments clarify this. Observe that the conditioning on the path $Y(u)$, $u \leq s$, also contains the information $\xi(s)$, and as such, must be taken into account when computing the probability of events in the future of time $s$. This information is also sufficient, together with the current position $Y(s)$, to obtain the Markov embedding of $Y$ as we have seen in Theorem \ref{thmGS}. Indeed, the proof exploited exactly what we have described: the vectors $(S_{n+k}, W_{n+k})$, $k=0,1,2,\dots$, conditionally on $S_{n}$, are independent of the past, and the joint distribution of $S_{n+1}$, $W_{n+1}$ does not depend on $n$. Thus, in order to compute the probability of events in the future, one takes into account the extra information $\xi(s)$ by just conditioning on $W_{N(s)+1}>\xi(s)$ and by remembering the current position $S_n$. This shows that $\xi(s)$ should be added as an additional coordinate to obtain a Markov embedding. For the process $Y^+$, however, this argument fails. Indeed, conditioning on $Y^+(u)$, $u \leq s$, gives the information on the current position $S_{N(s)+1}$ and also $W_{N(s)+1}>\xi(s)$. Since the vectors $(S_n, W_n)$, $n =0,1,2,\dots$, form a Markov chain (and thus $(S_{n+1}, W_{n+1})$ is independent of the past conditionally on $S_n$), to compute the probability of the residual lifetime at time $s$ by using the Markov property of $(S_n,W_n)$, $n=0,1,2,\dots$, would require knowing $S_{N(s)}$, but this is not the current position $Y^+(s)$. Furthermore, adding the extra coordinate $\xi(s)$ to the process also does not solve the issue, since conditioning on the current position $S_{N(s)+1}$ and on $W_{N(s)+1}>\xi(s)$ does not allow us to use the Markov property of $(S_n, W_n)$, $n=0,1,2,\dots$.
	
	On the other hand, if the residual lifetime, i.e. $\Xi(s)\coloneqq T^+(s)-s$, is added as an extra coordinate, this problem is overcome since the next random quantities appearing after time $s$ are $(S_{N(s)+2}, W_{N(s)+2}), \dots$, and are independent of the past, conditionally on the current position $Y^+(s)=S_{N(s)+1}$.
	
	Moreover, it turns out that we can make the Markov embedding of $Y^+(t)$, $t\in\RR$, rigorous in two ways, which we present in the following two theorems.
	
	The first theorem embeds the left limits of $Y^+$, i.e., the process $Y^+(t-)=\lim_{\delta\searrow0}Y^+(t-\delta)$, $t\in\RR$.
	
	\begin{theorem}
		\label{thmGS+}
		Suppose that the assumptions of Theorem \ref{thmGS} are satisfied. Define
		\begin{align}\label{next-jump}
			\Xi(t) \, = \, T_{n+1}-t, \qquad T_n \le  t < T_{n+1},
		\end{align}
		i.e. $\Xi (t) = T^+(t)-t$, $t \geq 0$. Denote by $N^-(s):= N(s-)$, $s \geq 0$.
		Then the process $(Y^+(t-), \Xi(t-))$, $t \geq 0$, is a time-homogeneous Markov process, under all $\PP^{(x,s)}$, $(x,s)\in \RR^d\times (-\infty,0]$, and the filtration $\mathcal{G}^-_t$, $t \geq 0$, generated by $N^-(s)$, $S_0, \dots, S_{N^-(s)+1}$, $T_0, \dots, T_{N^-(s)+1}$, $s \leq t$. Its semigroup $P_t^+$, $t \geq 0$, satisfies for $t\ge 0$ and $(x,s)\in\RR^d\times [0,+\infty)$
		\begin{align}\label{710-a1}
			P_t^+f(x,s) \, = \, \mathds{1}_{[0 \leq t \le  s]} f(x,s-t)+ \mathds{1}_{[0 \leq s <t]} P_{t-s}^+f(x,0),
		\end{align}
		and
		\begin{align}\label{710-b1}
			P_t^+f(x,0) \, = \, \mathds{E}^{(x,0)}f(Y(t-), \Xi(t-)),\quad t>0,\, x\in \RR^d.
		\end{align}
		for any function $f \in \mathcal{B}_b (\mathbb{R}^d \times [0, +\infty))$.
	\end{theorem}
	\begin{proof}
		The proof can be conducted similarly to the proof of Theorem \ref{thmGS}, i.e., \cite[Lemma 2, Section 3, Chapter 3]{Gihman}. However, since the claim is new, to the best of our knowledge, we give a complete proof.
		
		To shorten the notation in the proof, denote by $\wt Y(t)= Y^+(t-)$ and $\wt \Xi(t)=\Xi(t-)$.
		
		Let us first prove the Markov property. Take $(x,s)\in \RR^d\times (-\infty,0]$, and let $T,t>0$, and $f \in \mathcal{B}_b (\mathbb{R}^d \times [0, +\infty))$.
		
		As in the previous theorem, we have that
		\begin{align}
			& \mathds{E}^{(x,s)} \left[  f (\widetilde{Y}(T+t), \widetilde{\Xi}(T+t)) \mid \mathcal{G}^-_T \right] \notag \\
			\begin{split}\label{1134}
				= \, &\mathds{E}^{(x,s)} \left[ \mathds{1}_{[{T+t\leq T_{N^-(T)+1}]}} f (\widetilde{Y}(T+t), \widetilde{\Xi}(T+t)) \mid \mathcal{G}^-_T \right]\\ &+\mathds{E}^{(x,s)} \left[ \mathds{1}_{{[T+t> T_{N^-(T)+1}]}} f (\widetilde{Y}(T+t), \widetilde{\Xi}(T+t)) \mid \mathcal{G}^-_T \right] .
			\end{split}
		\end{align}
		
		The first term in \eqref{1134} is
		\begin{align}
			\mathds{E}^{(x,s)} &\left[ \mathds{1}_{[{T+t\leq T_{N^-(T)+1}]}} f (\widetilde{Y}(T+t), \widetilde{\Xi}(T+t)) \mid \mathcal{G}^-_T \right]\\
			= \, &f (\widetilde{Y}(T), \widetilde{\Xi}(T)-t)\mathds{1}_{[{T+t\leq T_{N^-(T)+1}]}}\, = \, f (\widetilde{Y}(T), \widetilde{\Xi}(T)-t)\mathds{1}_{[t\leq \widetilde{\Xi}(T)]},
		\end{align}
		since no jumps have occurred on the corresponding event.
		
		For the second term in \eqref{1134}, instead, there must have been at least one jump. Furthermore, on $T_{N^-(T)+1}<T+t$, conditionally on $N^-(T)$, $S_0$, $T_0$, $\dots$, $S_{N^-(T)+1}$, $T_{N^-(T)+1}$, the process $\widetilde{Y}(T+t)$, $t \geq 0$, is constructed as a function of the r.v.'s $S_{N^-(T)+2}$, $T_{N^-(T)+2}$, $S_{N^-(T)+3}$, $T_{N^-(T)+3},\dots$, in the same way the process $\widetilde{Y}(T+t-T_{N^-(T)+1})$, $t \geq 0$, is a function of $S_0$, $T_0$, $S_1$, $T_1,\dots$ conditionally on $T_0=0$ and $S_0 = S_{N^-(T)+1}$. Hence, 
		\begin{align}
			& \mathds{E}^{(x,s)} \left[ \mathds{1}_{[T+t> T_{N^-(T)+1}]} f (\widetilde{Y}(T+t), \widetilde{\Xi}(T+t)) \mid \mathcal{G}^-_T \right] \notag \\
			= \, & \mathds{1}_{[T+t>T_{N^-(T)+1}]}\mathds{E}^{(S_{N^-(T)+1},0)}f (\widetilde{Y}(T+t-T_{N^-(T)+1}), \widetilde{\Xi}(T+t-T_{N^-(T)+1})) \notag \\
			= \, & \mathds{1}_{[\widetilde{\Xi}(T) < t]} P^+_{t-\widetilde{\Xi}(T)}f(\widetilde{Y}(T),0).
		\end{align}
		By putting pieces together, we obtain the time-homogeneous Markov property since we showed that the conditional expectation in the first display of \eqref{1134} only depends on $\widetilde{\Xi}(T), \widetilde{Y}(T)$ and $t$. By putting $T=0$ we get  \eqref{710-a1}, while the relation \eqref{710-b1} follows since, under $\PP^{(x,0)}$, one has that $T_0=0$ a.s. and $\Xi(0-)=0$.
	\end{proof}
	
	The next theorem, in contrast to the previous one, embeds the original right-continuous process $Y^+$.
	\begin{theorem}
		\label{thmGS+1}
		Suppose that the assumptions of Theorem \ref{thmGS} are satisfied, and let $\Xi(t)$, $t\in \RR$, be defined as in \eqref{next-jump}. Then the process $(Y^+(t), \Xi(t))$, $t \geq 0$, is a time-homogeneous Markov process, under all $\PP^{(x,s)}$, $(x,s)\in \RR^d\times (-\infty,0]$, and the filtration $\mathcal{G}_t$, $t \geq0$, generated by $N(s), S_0, \dots, {S_{N(s)+1}}, T_0, \dots, {T_{N(s)+1}}$, $s \leq t$. Its semigroup $\wt P_t^+$, $t \geq 0$, satisfies for $(x,s)\in \RR^d\times (0,+\infty)$
		\begin{align}\label{710-a}
			\wt P^+_tf(x,s)&=\mathds{1}_{[0\le t<s]}f(x,s-t)+\mathds{1}_{[0<s\le t ]}\int_{\RR^d}\int_{(0,+\infty)}\wt P^+_{t-s}f(y,w)\mu_{x}(dy,dw),
		\end{align}
		for any function $f \in \mathcal{B}_b (\mathbb{R}^d \times (0, +\infty))$.
	\end{theorem}
	\begin{proof}
		The proof is again very similar to the proofs of Theorem \ref{thmGS} and Theorem \ref{thmGS+}. However, since the claim also appears to be new, we give a detailed sketch.
		
		We have that
		\begin{align}
			&\mathds{E}^{(x,s)} \left[ f(Y^+(T+t), \Xi(T+t)) \mid \mathcal{G}_T \right] \notag \\
			\begin{split}\label{1612}
				\, = \, & \mathds{E}^{(x,s)} \left[ f(Y^+(T+t), \Xi(T+t)) \mathds{1}_{[T+t < T_{N(T)+1}]} \mid \mathcal{G}_T \right] \notag \\ & + \mathds{E}^{(x,s)} \left[ f(Y^+(T+t), \Xi(T+t)) \mathds{1}_{[T+t \geq T_{N(T)+1}]} \mid \mathcal{G}_T \right].
			\end{split}
		\end{align}
		The first term in \eqref{1612} is just
		\begin{align}
			\mathds{E}^{(x,s)} \left[ f(Y^+(T+t), \Xi(T+t)) \mathds{1}_{[T+t < T_{N(T)+1}]} \mid \mathcal{G}_T \right] \, = \, f(Y^+(T), \Xi(T)-t)\mathds{1}_{[t < \Xi(T)]}.
		\end{align}
		For the second term, instead, we obtain
		\begin{align}
			&\mathds{E}^{(x,s)} \left[ f(Y^+(T+t), \Xi(T+t)) \mathds{1}_{[T+t \geq T_{N(T)+1}]} \mid \mathcal{G}_T \right] \notag \\
			= \, &\mathds{E}^{(x,s)} \left[ \mathds{1}_{[T+t \geq T_{N(T)+1}]} \mathds{E}^{(x,s)} \left[  f(Y^+(T+t), \Xi(T+t)) \mid \mathcal{G}_T, T_{N(T)+2}, S_{N(T)+2} \right] \mid \mathcal{G}_T  \right].\label{1142}
		\end{align}
		Now observe that the process $(Y^+(T+t), \Xi(T+t))$, on $T+t \geq T_{N(T)+1}$, conditionally on $\mathcal{G}_T$, $S_{N(T)+2}$, $T_{N(T)+2}$ is constructed with respect to the r.v.'s $S_{N(T)+3}$, $T_{N(T)+3},\dots$, in the same way the process $(Y^+(T+t-T_{N(T)+1}), \Xi(T+t-T_{N(T)+1}))$ is constructed with respect to the r.v.'s $S_2$, $T_2$, $\dots$, conditionally on $T_0=0$, $S_1 = S_{N(T)+2}$, and $T_1 = T_{N(T)+2}-T_{N(T)+1}$. Therefore \eqref{1142} becomes
		\begin{align}
			&\mathds{E}^{(x,s)} \left[ \mathds{1}_{[T+t \geq T_{N(T)+1}]} \mathds{E}^{(x,0)} \left[  f(Y^+(T+t-T_{N(T)+1}), \Xi(T+t-T_{N(T)+1})) \right. \right. \notag  \\ & \left. \left. \qquad \mid S_1=S_{N(T)+2}, T_1= T_{N(T)+2}-T_{N(T)+1} \right] \mid \mathcal{G}_T  \right] \notag \\
			= \, &\mathds{1}_{[T+t \geq T_{N(T)+1}]} \int_0^{+\infty} \int_{\mathbb{R}^d} \wt P^+_{T+t-T_{N(T)+1}}f(y, w) \mathds{P}^{(x,s)} (S_{N(T)+2} \in dy, W_{N(T)+2} \in dw \mid \mathcal{G}_T) \notag \\
			= \, &\mathds{1}_{[t \geq \Xi(T)]} \int_0^{+\infty} \int_{\mathbb{R}^d} {\wt P}^+_{t-\Xi(T)}f(y, w) \mu_{{Y^+(T)}} (dy, dw).
		\end{align} 
		This concludes the proof.
	\end{proof}
	
	\section{Scaling limits of CTRWs}
	\label{sec_limit}
	Scaling limits of CTRWs are time-changed Markov processes. We summarize here the main facts of the theory. In this section, we mainly refer to the papers \cite{straka2011lagging, Meerschaert2014}, as they contain the most recent (and general) formulations of the results that we will mention here. However, the literature on this topic is extensive; foundational contributions include \cite{becker2004limit, Meerschaert2004, meerschaert2008triangular, kolokoltsov2009generalized}.
	
	In order to give a general statement, we first define some objects and fix notation.
	We consider a canonical Feller process $(A,S) = \big((A(u), S(u)),\, u \geq 0\big)$ in $\mathbb{R}^{d+1}$. This means that $(A, S)$ is {a Markov process} constructed on a probability space $(\Omega, \mathcal{F}_\infty, \mathcal{F}_u, \PP^{(x,s)})$ as follows. The set $\Omega$ is $D([0, +\infty), \mathbb{R}^{d+1})$ where $D([0, +\infty), \mathbb{R}^{d+1})$ denotes the set of càdlàg functions from $[0, +\infty)$ to $\mathbb{R}^{d+1}$, and $\big(A(u), S(u)\big)(\omega) = \omega(u)$ for any $u \geq 0$. We assume that $D([0, +\infty), \mathbb{R}^{d+1})$ is endowed with the Skorokhod $J$ topology (see \cite[Chapter 6]{jacod2013limit}). By $\mathcal{F}_t^0$ we denote the $\sigma$-algebras generated by $(A(u), S(u))$, $0 \leq u \leq t$, and denote $\mathcal{F}^0_\infty\coloneqq \bigvee_t \mathcal{F}^0_t $. We recall that, under the Skorokhod $J$ topology, $\mathcal{F}_\infty^0$ is the same as the Borel $\sigma$-algebra on $D$ (see \cite[Chapter 6, Theorem 1.14]{jacod2013limit}). The probability measures $\PP^{(x,s)}$ are the unique measures on $\mathcal{F}^0$ such that $\mathds{P}^{(x,s)}\big((A(0), S(0))=(x,s)\big)=1$. We denote by $\mathcal{F}_\infty$, and $\mathcal{F}_u$, $u \geq 0$, the completion of $\mathcal{F}_\infty^0$ and $\mathcal{F}^0_u$, $u \geq 0$, respectively, with $\mathds{P}^{(x,s)}$-null sets, $(x,s)\in \RR^{d+1}$. With the symbol $P_u$, $u \geq 0$, we will denote the corresponding (Feller) semigroup of operators (see more on Feller processes and semigroups in Appendix \ref{appendix_semigroups}).
	
	In our setting, we will always assume that the process $S(u)$, $u\geq 0$, is strictly increasing and unbounded. On this probability space, we will consider the (generalized) inverse of the strictly increasing process $S$, i.e., the process $L= (L(t),\, t \in \mathbb{R})$, where
	\begin{align}\label{defL}
		L(t) \coloneqq {\sup\{ u \geq 0: S(u)\leq t  \}  \, = \, } \inf \{ u>0 : S(u) > t \},\quad t\in\RR.
	\end{align}

	Let us introduce the scale parameter $c>0$ in the construction of the (O)CTRWs of the previous section.\footnote{Here, in this generality, the scaling rates of both the jumps and the waiting times will not be given explicitly, but will instead be left free, and in applications the rates should be chosen such that they imply the existence of the scaling limit (see Theorem \ref{maintheoremctrwlim}, as well as Examples \ref{example_uncoupled} and \ref{example_limit_coupled}, which use explicit rates).} In other words, for each $c>0$, we have the sequences of random variables $W_n^c, J_n^c$, $n \in \mathbb{N}$, and the processes $S_n^c = S_0 + \sum_{i=1}^n J_i^c$, and $T_n^c = T_0+ \sum_{i=1}^n W_i^c$. Therefore, the (O)CTRWs are given by
	\begin{align}
		Y_c(t) = \sum_{i=1}^{N^c(t)} J_i^c,  \qquad Y_c^+(t) = \sum_{i=1}^{N^c(t)+1} J_i^c,\qquad t\in \RR,
	\end{align}
	where
	\begin{align}
		N^c(t) = \max \{ n \in \mathbb{N} \cup 0 : T_n^c \leq t \},\qquad t\in \RR,
	\end{align}
	with $N^c(t)=0$ for $t<T_0^c$. The times of the last and the forthcoming jumps are now
	\begin{align}
		T_c(t) = T_{N^c(t)}^c,  \qquad T_c^+(t) = T_{N^c(t)+1}^c,\qquad t\in \RR.
	\end{align}

	The processes $(Y_c,T_c) = \big ((Y_c(t), T_c(t)), \, t \geq 0\big)$ and $(Y^+_c, T^+_c)=\big( (Y^+(t), T^+(t)), \, t \geq 0 \big)$ are given on some probability space $(\wt{\Omega}, \wt{\mathcal{G}}, \wt{\PP}_c^{(x,s)} )$, but can also be seen as mappings $( \wt{\Omega}, \wt{\mathcal{G}}) \mapsto (D([0, +\infty), \mathbb{R}^{d+1}), {\mathcal{F}_\infty})$, so, by taking the preimage, we can consider that the measures $\wt{\PP}_c^{(x,s)}$ are defined on $D([0, +\infty), \mathbb{R}^{d+1})$. Hence, in the following theorem we may study the weak convergence of $\wt{\PP}_c^{(x,s)}$ to the limit law $\PP^{(x,s)}$ (the law of the Feller process $(A,S)$) on $(D([0, +\infty), \mathbb{R}^{d+1}), \mathcal{F}_0 )$ equipped with the Skorokhod topology $J$ as $c\to +\infty$.

	The next theorem comes from \cite[Theorem 3.6]{straka2011lagging}, but see also \cite[Section 2]{Meerschaert2014} for a discussion on weakening the assumptions of \cite{straka2011lagging}.
	\begin{theorem}
		\label{maintheoremctrwlim}
		Let $c>0$ be a scale parameter and denote by $Y_c(t)$, $t \geq 0$, and $Y_c^+(t)$, $t \geq 0$, respectively, the CTRW and the OCTRW defined as above, depending on $c>0$. Assume that $(S^c_{[cu]},T^c_{[cu]})$, $u\ge0$, converges to $(A(u), S(u))$, $u\ge0$, weakly in $D([0, +\infty), \mathbb{R}^{d+1})$ {as $c\to +\infty$}\footnote{Here, the notation $[cu]$ denotes the integer part of the real number $cu$.}, with $S$ being strictly increasing and unbounded. Then, as $c \to +\infty$,
		\begin{align}
			&( Y_c(t), T_c(t) ) \to (A(L(t)-), S(L(t)-))^+\\
			& ( Y_c^+(t), T_c^+(t) ) \to (A(L(t)), S(L(t)))
		\end{align}
		weakly in $D([0, +\infty), \mathbb{R}^{d+1})$ (endowed with the Skorokhod $J$ topology).
	\end{theorem}
	\begin{proof}[Sketch of proof]
		We summarize here the strategy of the proof that can be found in \cite{straka2011lagging}. The proof is made explicit when $(A, S)$ is a L\'evy process, but no substantial change is needed in our more general setting (as already pointed out in \cite{Meerschaert2014}).
		
		For $\mathpzc{g} \in D([0, +\infty), \mathbb{R}^{d+1})$, write $\mathpzc{g}=(\mathpzc{a}, \mathpzc{s})$, $\mathpzc{a} \in D([0, +\infty), \mathbb{R}^d)$, $\mathpzc{s} \in D([0, +\infty), \mathbb{R})$. Denote by $D_{\uparrow,u}$ ($D_{\uparrow \uparrow,u}$) the subset of $D([0, +\infty), \mathbb{R}^{d+1})$ such that the coordinate $\mathpzc{s}$ is non-decreasing (strictly increasing) and unbounded, and denote by $l_{\mathpzc{s}} (t) \coloneqq \inf \{ h \geq 0 : \mathpzc{s}(h) > t \}$ its (generalized) inverse. Note that if $\mathpzc{s}$ is strictly increasing and unbounded, then $l_{\mathpzc{s}}$ is continuous. For an element $\mathpzc{g}$ in $D([0, +\infty), E)$, where $E$ is any codomain, we denote by $\mathpzc{g}^-$ the function $t \mapsto \mathpzc{g}(t-)$, and set $\mathpzc{g}^-(0)\coloneqq \mathpzc{g}(0)$. The equivalent notation is used for a function $\mathpzc{h}$ that is left-continuous with right limits, i.e., $\mathpzc{h}^+$ denotes $t \mapsto \mathpzc{h}(t+)$.
		
		The steps of the proof are:
		\begin{enumerate}
			\item For $\mathpzc{g} \in D_{\uparrow, u}$, define
			\begin{align*}
				&  \Phi (\mathpzc{g}) \coloneqq (\mathpzc{g}^- ( l_{\mathpzc{s}}^-))^+, \\
				&  \Psi (\mathpzc{g}) \, \coloneqq \, \mathpzc{g} ({l}_{\mathpzc{s}}).
			\end{align*}
			\item\label{label2} Prove that $\Phi$ and $\Psi$ are continuous mappings on $D_{\uparrow \uparrow, u}$ (with the subspace topology induced by $J$).
			\item Observe that
			\begin{align*}
				& (Y_c, T_c) = \Phi ( S^c, T^c ), \\
				& (Y_c^+,T_c^+) = \Psi ( S^c, T^c ).
			\end{align*}
			\item Note that since $S(u)$, $u\ge0$, is assumed to be strictly increasing, its inverse $L(t)$, $t\ge0$, has continuous trajectories, hence $ \Phi (A, S) = (A^-(L), S^-(L))^+$. Further,  it also holds that $\Psi(A,S) = (A(L), S(L))$.
			\item Use item \eqref{label2} to obtain the claim by the continuous mapping theorem.
		\end{enumerate}
	\end{proof}
	
	\subsection{Examples}
	We show here some applications of Theorem \ref{maintheoremctrwlim}, i.e., we construct some (O)CTRWs to which we apply the result.
	
	In this context, it is useful to introduce some objects from the theory of operator stable laws, for which we refer the reader to \cite{meerschaert2001limit} for a more detailed discussion. We say that a random vector $X$ in $\mathbb{R}^d$ (i.e., its law $\nu$) is full if the scalar product $\langle t, X \rangle$ is not a degenerate random variable (i.e., its law is not a point mass), for any $t \neq 0$. Further, consider a random vector $Y$ with a full distribution $\nu$ and a sequence of i.i.d. r.v.'s $X, X_1, X_2, \dots$ Then, $\nu$ is said to be \emph{operator stable} if there exist linear operators $A_n : \mathbb{R}^d \mapsto \mathbb{R}^d$ and vectors $b_n \in \mathbb{R}^d$ such that
	\begin{align}
		A_n (X_1 + \dots + X_n ) + b_n \to Y \quad \text{ in distribution as $n\to+\infty$.}
		\label{defopstable}
	\end{align}
	The r.v. $X$ is said to be in the domain of attraction of $Y$ if \eqref{defopstable} holds true. It is noteworthy that in \eqref{defopstable} the r.v. $Y$ must be infinitely divisible and one can always choose $A_n$ to be regularly varying. Here, a function $B:(0, +\infty) \mapsto \text{GL}(\mathbb{R}^d)$, where $\text{GL}(\mathbb{R}^d)$ is the Lie group of invertible linear operators on $\mathbb{R}^d$, is said to be regularly varying (with operator $E$) if $\lim_{t \to +\infty}B( \lambda t) B(t)^{-1}= \lambda^E$ for all $\lambda >0$, where $\lambda^E  \coloneqq \sum_{k=0}^{+\infty} \frac{(E \log \lambda)^k}{k!}$.
	
	In the following examples, we will use the following theorem on convergence in distribution and in the $J$ topology to verify the hypotheses of Theorem \ref{maintheoremctrwlim}. The theorem is a combination of results in \cite{meerschaert2001limit, Meerschaert2004}, and we also note that it is instructive to read \cite[Example 11.2.18]{meerschaert2001limit}.
	\begin{theorem}
		\label{theoremforexamples}
		Let $\mathcal{X}(t)$, $t \geq0$, be a stochastic process in $\RR^d$ with $\mathcal{X}(0) =0$, with stationary and independent increments. Assume that for some linear operator $E$ one has $\mathcal{X}(ct)\stackrel{\text{d}}{=} c^E\mathcal{X}(t)$, for all $t>0$, and that the law of $\mathcal{X}(1)$ is full. Let $X_n$, $n\in\NN$, be a sequence of r.v.'s in $\mathbb{R}^d$, in the domain of attraction of some full operator stable law as in \eqref{defopstable}, with $b_n=0$, $n\in \NN$, and with $Y \stackrel{\text{d}}{=}\mathcal{X}(1)$. Then there exists a regularly varying function $B$ with operator $-E$ such that
		\begin{align}\label{1351claim}
			B(n) \sum_{i=1}^n X_i \to \mathcal{X}(1),\quad \text{as $n\to+\infty,$}  
		\end{align}
		in distribution.
		Furthermore, as $c \to +\infty$,
		\begin{align}
			B([c]) \sum_{i=1}^{[ct]} X_i \to  \mathcal{X}(t)
		\end{align}
		and the convergence holds both in the sense of all finite-dimensional distributions and weakly on $D([0, +\infty), \mathbb{R}^d)$ with the $J$ topology.
	\end{theorem}
	\begin{proof}[Sketch of proof]
		The existence of $B$ as claimed in \eqref{1351claim} comes directly from \cite[Theorem 8.1.5]{meerschaert2001limit}. By Theorem \cite[Theorem 4.2.9]{meerschaert2001limit} we know that the function $c \mapsto B([c])$ varies regularly with the same exponent $-E$. Now it easily follows that
		\begin{align}
			B([c])\sum_{i=1}^{[c]} X_i  \to \mathcal{X}(1),\quad \text{as $c\to+\infty$}
		\end{align}
		in distribution. Further, for all $s <t$, observe that
		\begin{align}
			& \sum_{i=1}^{[t]} X_i - \sum_{i=1}^{[s]} X_i \stackrel{\text{d}}{=} \sum_{i=1}^{[t]-[s]} X_i,\quad \text{and}   \\
			& \mathcal{X}(t)-\mathcal{X}(s) \stackrel{\text{d}}{=} (t-s)^E \mathcal{X}(1).
		\end{align}
		From these, we conclude that, as $c\to+\infty$,
		\begin{align}
			B([c])\left(\sum_{i=1}^{[ct]} X_i - \sum_{i=1}^{[cs]} X_i\right) \to (t-s)^E \mathcal{X}(1) \stackrel{\text{d}}{=} \mathcal{X}(t)-\mathcal{X}(s),
		\end{align}
		where the convergence is in distribution.
		
		To obtain convergence of all finite-dimensional distributions, we use what was already obtained together with the independence of the increments $X_{[t_i]}-X_{[t_{i-1}]}$ and $\mathcal{X}(t_i)-\mathcal{X}(t_{i-1})$, $i=1, \dots, n$. For the convergence in the $J$ topology, we apply \cite[Theorem 4.1]{Meerschaert2004}.
	\end{proof}

	\begin{example}[Uncoupled CTRWs]\label{example_uncoupled}
		CTRWs are said to be uncoupled if the time jumps $W_n$ and the space jumps $J_n$ are independent. In this simpler case, Theorem \ref{maintheoremctrwlim} can be applied as follows.
		Suppose that $W_n$, $n\in\NN$, are i.i.d. Mittag-Leffler distributed, i.e., they have c.d.f. 
		\begin{align}
			F(t) = 1-E_{\alpha}(-t^\alpha) , \qquad t\geq 0,
			\label{mlcdf}
		\end{align}
		for some fixed $\alpha \in (0,1)$, where the Mittag-Leffler function $E_\alpha$ is given by
		\begin{align}
			E_\alpha(z) \coloneqq \sum_{k=0}^{+\infty} \frac{z^k}{\Gamma (1+\alpha k)}.
		\end{align}
		Then, in distribution, as $n \to +\infty$
		\begin{align}
			n^{-1/\alpha}   (W_1 + \dots + W_n) \to \mathcal{S},
		\end{align}
		where $\mathcal{S}$ is a stable law with the Laplace transform
		\begin{align}
			\mathds{E}e^{- \lambda \mathcal{S}} \, = \, e^{-\lambda^\alpha}.
		\end{align}
		This can be seen from the convergence of the Laplace transforms, as follows
		\begin{align}
			\mathds{E}e^{-\lambda n^{-1/\alpha}   (W_1 + \dots + W_n) } \, = \, \left( \frac{1}{1+\lambda^\alpha/n} \right)^n  \to e^{-\lambda^\alpha}, \quad \text{as $n\to\infty$},
		\end{align}
		where in the first step we used that (see, e.g., \cite{scalas2006five} and references therein)
		\begin{align}
			\int_0^{+\infty} e^{-\lambda t} E_\alpha (-t^\alpha ) dt \, = \, \frac{\lambda^{\alpha-1}}{1+\lambda^\alpha}
		\end{align}
		to get
		\begin{align}
			\mathds{E}e^{-\lambda n^{-1/\alpha} W_1} \, = \, \frac{1 }{1+\lambda^\alpha /n}.
			\label{1649}
		\end{align}
		Suppose that $J_n$, $n\in \NN$, are i.i.d. r.v.'s having unit variance and zero mean, independent of the sequence $W_n$, $n\in\NN$. Then, in distribution, $n^{-1/2} (J_1 + \dots + J_n) \to \mathcal{Z}$, where $\mathcal{Z}$ is a standard normal r.v. 
		
		The previous calculations and Theorem \ref{theoremforexamples} suggest what the scaling of the space-time jumps $(J_n,W_n)$ should be. Indeed, take the process $\mathcal{X}(t)$, $t\ge0$, with
		\begin{align}
			\mathcal{X}(1) = (\mathcal{Z}, \mathcal{S}),
		\end{align}
		and $\mathcal{X}(t) = (A(t), S(t))$, $t \geq 0$, where $A(t)$, $t\ge0$, is a standard $d$-dimensional Brownian motion independent of a stable subordinator $S(t)$, $t\ge0$, of order $\alpha \in (0,1)$\footnote{See more on subordinators in {Appendix} \ref{appendix_bernstein}.}. Further, in the notation of Theorem \ref{theoremforexamples}, take $X_n \coloneqq  (J_n, W_n)$, $n\in\NN$. It follows that the scaling function is $B(c) = \text{diag}(c^{-1/2}, c^{-1/\alpha})$, $c>0$.
		
		By using the notation of the scaled CTRWs, for $(W^c_k,J_k^c)\coloneqq B(c)(W_k,J_k)$, Theorem \ref{theoremforexamples} yields
		\begin{align}
			(S^c_{[cu]},T^c_{[cu]})=B(c)\sum_{i=1}^{[cu]}X_i\to \mathcal{X}(t)=(A(t),S(t)),\quad\text{as $c\to+\infty$,}
		\end{align}
		weakly in $D([0,\infty),{\mathbb{R}^{2}})$ with the $J$ topology.

		Thus, Theorem \ref{maintheoremctrwlim} implies that
		\begin{align}
			&  (Y_c(t),T_c(t))=\Phi ( S^c_{[ct]}, T^c_{[ct]} ) \to (A(L(t)), S(L(t)-))^+ \label{1654},\\
			& (Y^+_c(t),T^+_c(t))=\Psi ( S^c_{[ct]}, T^c_{[ct]} ) \to (A(L(t)), S(L(t)))  \label{1655}
		\end{align}
		in the Skorokhod $J$ topology. In \eqref{1654}, we also used the fact that $A(t)$, $t\ge0$, is a Brownian motion, and hence continuous. In order to make everything very explicit, one can rewrite \eqref{1654} and \eqref{1655} as follows:
		\begin{align}
			& \frac{1}{\sqrt{c}} \sum_{i=1}^{N(c^{1/\alpha} t) } J_i \to A(L(t)), \qquad \frac{1}{c^{1/\alpha}} \sum_{i=1}^{N(c^{1/\alpha}t)} W_i \to S(L(t)-)^+, \\
			&  \frac{1}{\sqrt{c}} \sum_{i=1}^{N(c^{1/\alpha} t)+1 } J_i \to A(L(t)), \qquad \frac{1}{c^{1/\alpha}} \sum_{i=1}^{N(c^{1/\alpha}t)+1} W_i  \to S(L(t)),
		\end{align}
		where $N(t)=\max\{n\in \NN\cup\{0\}: \sum_{i=1}^nW_i\le t\}$.
		\qed
	\end{example}
	
	The previous example is a special case of a fairly general family of limit procedures involving stable distributions that can be made quite explicit by relying on the scaling property of stable laws. Also, the previous example suggests that the CTRW and OCTRW limits ($A(L(t)-)^+$ and $A(L(t))$, respectively) should coincide in the uncoupled case. This is, indeed, true: if the limit processes $A$ and $S$ (of Theorem \ref{maintheoremctrwlim}) are independent, then they do not have simultaneous jumps almost surely. In particular, the following lemma holds.
	\begin{lemma}[{\cite[Lemma 3.9]{straka2011lagging}}]
		\label{lemma_ctrw=octrw}
		If $\textrm{Disc}(A) \cap \textrm{Disc} (S) = \emptyset$\footnote{Here, $\textrm{Disc}(A)$ denotes the (random) set of discontinuities of the process $A$.}, a.s., then for all $x\in\RR^{d}$
		\begin{align}
			\mathds{P}^{(x,0)} (A(L(t))=A(L(t)-)^+ \text{ for all } t \geq 0) = 1.
		\end{align}
	\end{lemma}
	\begin{proof}
		Since both $A(L(t))$ and $A(L(t)-)^+$ start from the same position $A(0)=x$, we only have to check whether the left limits of the processes coincide. To this end, the left limits of $A(L(t))$ are
		\begin{align}\label{iv1430}
			A(L(t-))=\begin{cases}
				A(L(t)-),&\text{if $L(s)<L(t)$ for $s<t$},\\
				A(L(t)),&\text{otherwise},
			\end{cases}
		\end{align}
		while the left limits of $A(L(t)-)^+$ are
		$A(L(t)-)$ (since $S$ is strictly increasing, so $L$ is continuous). The left limits coincide in the first case of \eqref{iv1430} while, in the second case, we have $L(t)\in \text{Disc}(S)$, hence $L(t)\notin \text{Disc}(A)$, and consequently $A(L(t))=A(L(t)-)$.
	\end{proof}
	
	\begin{example}[Coupled CTRWs]
		\label{example_limit_coupled}
		CTRWs are said to be coupled when the space-time jumps $J_n, W_n$ are not independent. We provide here a very explicit example that can be made more general as in the previous (uncoupled) case. Take a CTRW with time jumps $W_n$, $n\in\NN$, that are (again) Mittag-Leffler distributed with parameter $\alpha\in (0,1)$. Suppose, however, that the space-time jumps are not independent and admit the joint density
		\begin{align}
			f_{J,W}(x,s) \, = \, \frac{1}{\sqrt{2\pi s}} e^{-\frac{x^2}{2s}} e_\alpha (s) \mathds{1}_{[x \in \mathbb{R}, s > 0]}
		\end{align}
		where $e_\alpha (\cdot)$ is the density induced by the Mittag-Leffler c.d.f. \eqref{mlcdf}. 
		Let $B(t)$, $t \geq 0$, be a standard Brownian motion independent of an $\alpha$-stable subordinator $S(t)$, $t \geq 0$. Define $S_0(t) = S(t) - S(0)$, $A(t) = B(S_0(t))$ and $\mathcal{X}(t) = (A(t), S(t))$, $t \geq 0$. Note that, by using independence of $B$ and $S$, we have
		\begin{align}
			\mathds{E}^{(x,s)}\left[e^{i\xi \cdot B(S_0(t))-\lambda S(t) } \right]\, = \, &\mathds{E}^{(x,s)}  \mathds{E}^{(x,s)} \left[ e^{i\xi \cdot B(S_0(t))-\lambda S(t) } \mid S(t) \right] \notag \\ 
			= \, &e^{i\xi \cdot x} \mathds{E}^{(x,s)} e^{-\lambda S(t)}  e^{-S_0(t) \frac{|\xi|^2}{2}} \notag \\
			= \, & e^{i\xi \cdot x+\lambda s} \mathds{E}^{(x,s)} e^{-\lambda S_0(t)}  e^{-S_0(t) \frac{|\xi|^2}{2}} \notag \\
			= \, & e^{i\xi \cdot x+\lambda s} e^{-t(\lambda + |\xi|^2/2)^\alpha}.
		\end{align}
		Hence the process $(A(u), S(u))$, $u\ge 0$, is a L\'evy process with the Fourier-Laplace symbol
		\begin{align}
			\varphi (\xi, \lambda)=  \left( \lambda + \frac{|\xi|^2}{2} \right)^\alpha.
		\end{align}
		Observe that, in distribution, 
		\begin{align}
			\text{diag}\left( (\sqrt{n})^{-1/\alpha}, n^{-1/\alpha} \right) \sum_{i=1}^n (J_i, W_i) \to \mathcal{X}(1),
		\end{align}
		as $n \to +\infty$. Indeed,
		\begin{align}
			& \mathds{E}\left(e^{-\lambda n^{-1/\alpha}\sum_{k=1}^{n} W_k + i\xi \cdot (\sqrt{n})^{-1/\alpha}\sum_{k=1}^{n} J_k}\right) \notag \\ = \, &\mathds{E} \left(\mathds{E} \left[ e^{-\lambda n^{-1/\alpha}\sum_{k=1}^{n} W_k + i\xi \cdot (\sqrt{n})^{-1/\alpha}\sum_{k=1}^{n} J_k} \mid W_1, \dots, W_{n} \right]\right) \notag \\
			= \, & \mathds{E}\prod_{k=1}^{n} \left(e^{-n^{-1/\alpha}\lambda W_k} \mathds{E} \left[ e^{i (\sqrt{n})^{-1/\alpha}\xi \cdot J_k} \mid W_k \right]\right) \notag \\
			= \, & \left( \mathds{E}\left[e^{-\lambda n^{-\frac{1}{\alpha}} W_1} e^{-W_1 \frac{\xi^2}{2n^{1/\alpha}}} \right]\right)^{n} \notag \\
			= \, & \left( \frac{1}{1+ \frac{1}{n}(\lambda + \xi^2/2)^\alpha} \right)^{n} \to e^{-(\lambda + |\xi|^2/2)^\alpha},\quad \text{as $n\to+\infty$.}
			\label{forosmall}
		\end{align}
		We can now use Theorem \ref{theoremforexamples} to get that
		\begin{align}
			\left(   \frac{1}{c^{\frac{1}{2\alpha}}}  \sum_{i=1}^{[cu]} J_i, \frac{1}{c^{\frac{1}{\alpha}}} \sum_{i=1}^{[cu]} W_i  \right) \eqqcolon (S^c_{[cu]},  T^c_{[cu]}) \to (A(u), S(u)),
		\end{align}
		weakly in $D([0, +\infty), {\mathbb{R}^{2}})$ endowed with the $J$ topology.
		
		Therefore, it follows from Theorem \ref{maintheoremctrwlim} that
		\begin{align}
			&    \frac{1}{c^{1/2\alpha}}\sum_{i=1}^{N(c^{1/\alpha}t)}J_i \to A(L(t)-)^+, &  \frac{1}{c^{1/\alpha}}\sum_{i=1}^{N(c^{1/\alpha}t)}W_i \to S(L(t)-)^+, \\
			&    \frac{1}{c^{1/2\alpha}}\sum_{i=1}^{N(c^{1/\alpha}t)+1}J_i \to A(L(t)), &  \frac{1}{c^{1/\alpha}}\sum_{i=1}^{N(c^{1/\alpha}t)+1}W_i \to S(L(t)),
		\end{align}
		weakly in $D([0, +\infty), {\mathbb{R}^{2}})$ endowed with the $J$ topology, where $N(t)=\max\{n\in \NN\cup\{0\}: \sum_{i=1}^nW_i\le t\}$.
		\qed
	\end{example}
	
	In the previous example the process $A(L(t)-)^+, t \geq 0$, can be seen as the Brownian motion $B$ time-changed by the undershooting process $S(L(t)-)$, $t\ge0$, of the (independent) subordinator $S$. On the other hand, the process $A(L(t))$, $t \geq 0$ can be seen as the Brownian motion $B$, time-changed by the overshooting process $S(L(t))$, $t\ge0$, of the (independent) subordinator $S$.
	
	In general, for a subordinator $S(t)$, $t\ge0$,  and fixed time $t\in \RR$, the triplet of its inverse, overshooting, and undershooting, i.e., $\big(L(t),S(L(t)), S(L(t)-)\big)$, is called the first passage event of the subordinator over level $t$. Indeed, the inverse $L(t)$ is the time at which $S$ crosses the level $t$, the overshooting $S(L(t))$ is the position of $S$ after it crosses the level $t$, while the undershooting $S(L(t)-)$ is the position of $S$ just before crossing the level $t$, see Figure \ref{figure}.
	
	The previous example can be generalized to more general Markov processes $B$ and more general subordinators $S$. This motivates the study of processes time-changed by overshootings and undershootings of subordinators.
	
	\begin{figure}
		\centering
		\begin{subfigure}{\textwidth}
			\centering
			\includegraphics[width=0.875\linewidth]{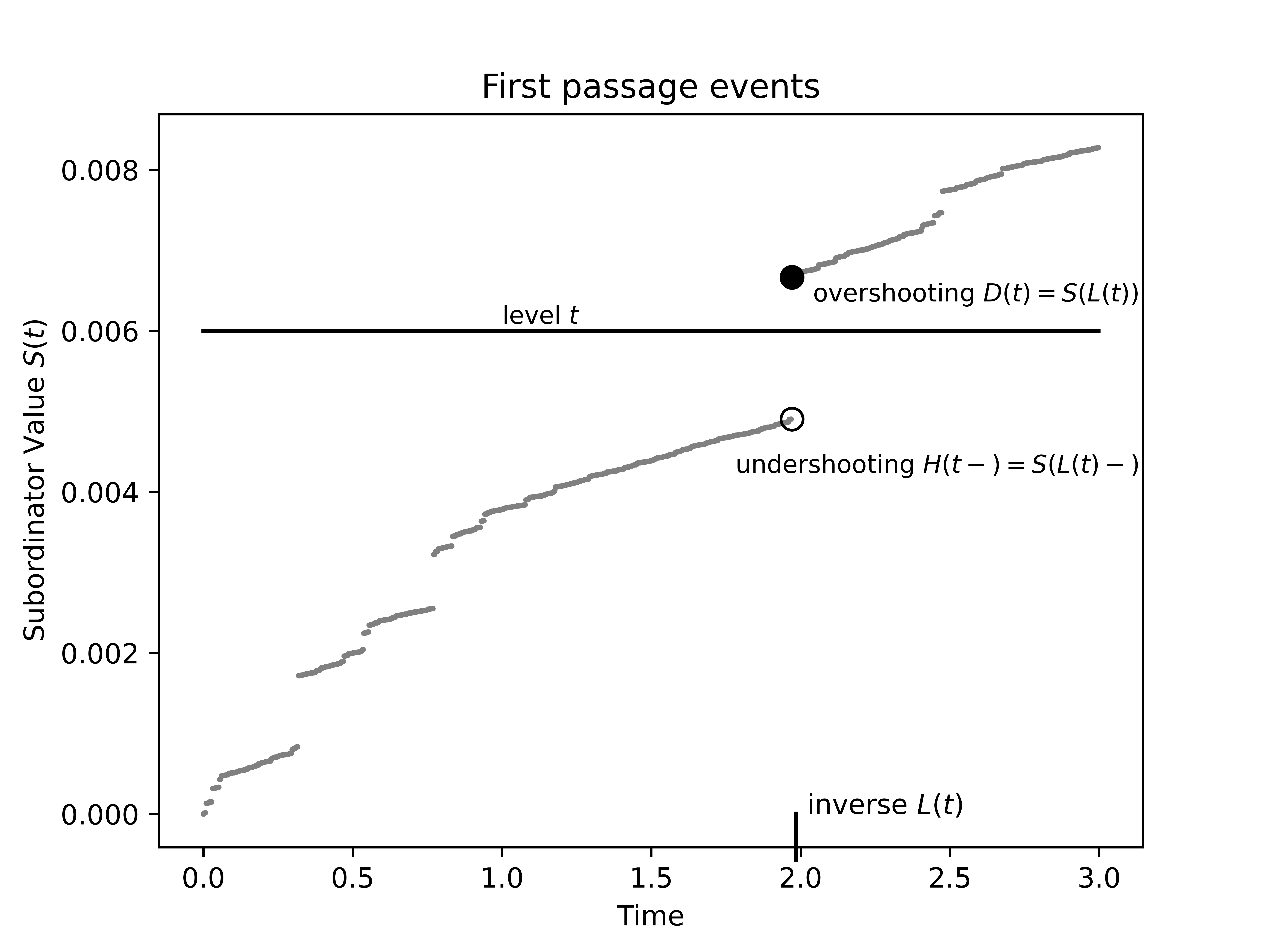}
			\caption{An exact sample of the trajectory of a stable (with $\alpha=0.8$) subordinator, together with its first passage events for level $t=0.006$.}
			\label{Top}
		\end{subfigure}
		\\
		\begin{subfigure}{\textwidth}
			\centering
			\includegraphics[width=0.875\linewidth]{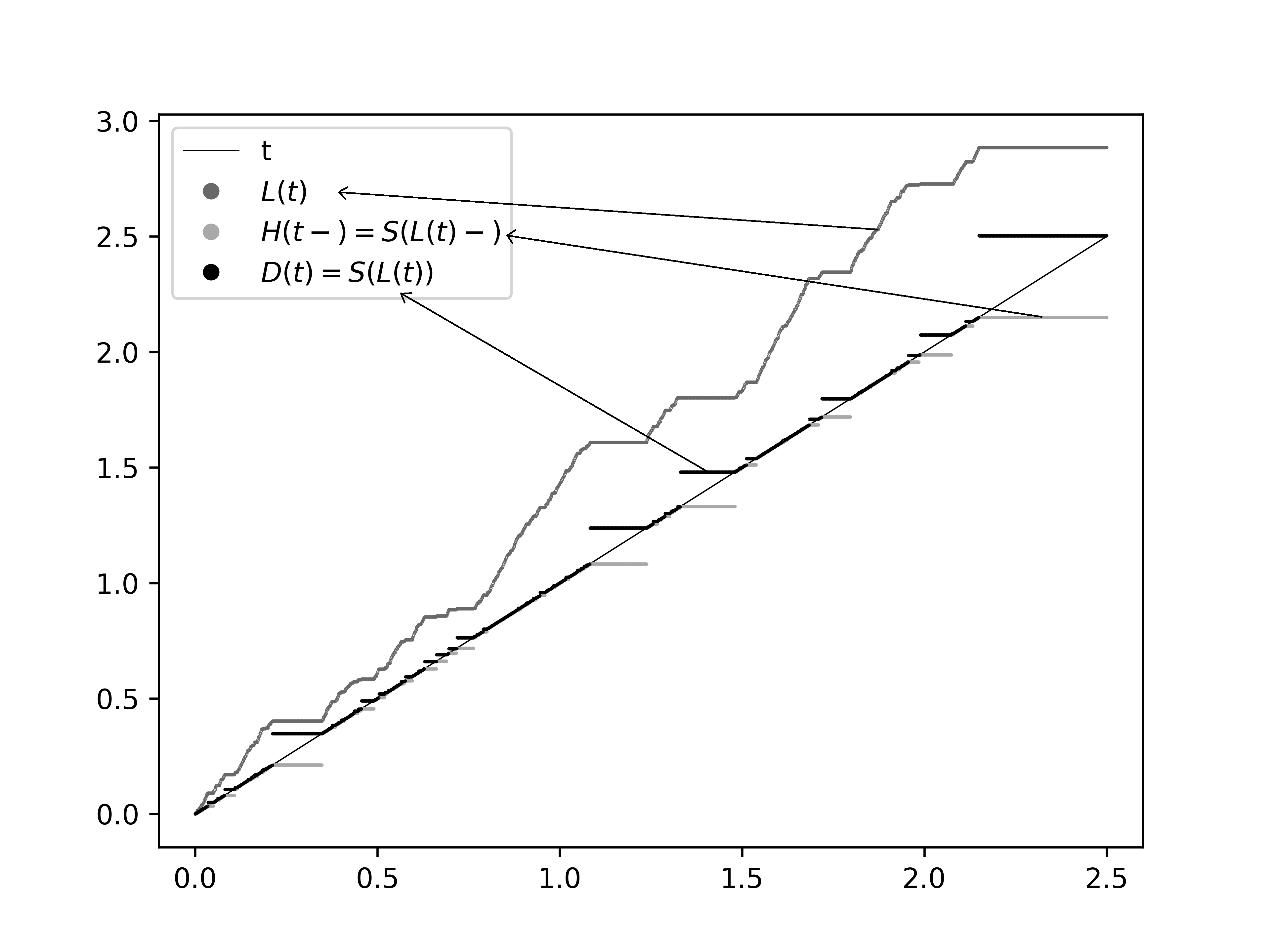}
			\caption{An exact sample of the joint trajectories of an inverse (stable with $\alpha=0.75$) subordinator, its overshooting, and  its undershooting processes; exactly sampled using the algorithms from \cite{daniel24}.}
			\label{Bottom}
		\end{subfigure}
		\caption{}
		\label{figure}
	\end{figure}
	
	\subsection{The semi-Markov property of the limit process}
	\label{semi_limit}
	It is relevant that the scaling limits of (O)CTRWs preserve the same type of semi-Markov property as the corresponding (O)CTRWs. In this section, we are going to explain the connection of the processes
	\begin{align}
		&  X^+ (t) \coloneqq A(L(t)), \quad t \geq 0, \label{defn:X^+}\\
		&   X (t) \coloneqq A(L(t)-)^+, \quad t \geq 0,\label{defn:X}
	\end{align}
	with the theory of semi-Markov processes in the sense of Gikhman and Skorokhod. In other words, we are going to show the Markov embedding property of $X^+$ and $X$ in Theorem \ref{theorem_semim} that is the counterpart for the limit processes of the embeddings in Theorems \ref{thmGS} and \ref{thmGS+}. The presented theory is due to \cite{Meerschaert2014}. We first introduce some notation and make assumptions.
	
	In accordance with Theorem \ref{maintheoremctrwlim}, we are assuming that the limit process $(A,S)=\big((A(u), S(u)),\, u \geq 0\big)$, is a Feller process in $\RR^{d+1}$, given on the canonical probability space, and that the second coordinate $S$ is strictly increasing and unbounded.
	
	Recall that the semigroup of $(A,S)$ is denoted by $P_u$, $u \geq 0$, and acts on $C_0 (\mathbb{R}^{d+1})$. Let us assume that its infinitesimal generator admits the following Courr\'ege--von Waldenfels type form on a suitable core:
	\begin{align}
		\mathcal{A}f(x,t) \, = \, &\sum_{i=1}^d b_i(x,t)\partial_{x_i} f(x,t) + \gamma (x,t) \partial_tf(x,t) + \frac{1}{2} \sum_{1 \leq i, j \leq d} a_{ij}(x,t)\partial_{x_i x_j}^2 f(x,t) \notag \\
		& + \int \left[ f(x+y, t+w)-f(x,t) -\sum_{i=1}^d h_i(y,w) \partial_{x_i} f(x,t)  \right] K(x,t;dy, dw).
		\label{generatorlimit}
	\end{align}
	Here $b_i$ and $\gamma$ are real-valued functions, $a_{ij}(x,t)$ are the entries of a non-negative definite $d\times d$ matrices, $h_i(x,t) = x_i \mathds{1}_{\big[(x,t) \in [-1,1]^{d+1}\big]}$. The kernel $K(x,t;dy, dw)$ is a jumping kernel from $\mathbb{R}^{d+1}$ to itself, i.e., the function $(x,t) \mapsto K(x,t;C)$ is Borel measurable in $\RR^{d+1}$ for any Borel set $C$ in $\RR^{d+1}$, $C \mapsto K(x,t;C)$ is a measure for all $(x,t)\in \RR^{d+1}$, and the integrability condition holds
	\begin{align}
		\int \left( 1 \wedge (\left\| y \right\|^2 + |w|) \right)  K(x,t; dy, dw) < +\infty.
	\end{align}
	
	We will further assume that the process $(A, S)$ is a solution to an SDE of the form \cite[Eq. (6.12), Section 6.2]{applebaum} and also that it is a semi-martingale (see, e.g., \cite{jacod2013limit}). These properties are, e.g., implied by the Lipschitz and linear growth conditions of the coefficients $b, \gamma, a, K$ (see \cite[Chapter 6]{applebaum}).

	Since the processes $X^+$ and $X$ are not stepped, we need to define the counterparts of the processes \eqref{last-jump} and \eqref{tnt1}, given that in this case the number of renewal points (i.e. jumps of $S$) in a time interval does not have to be finite. 
	Define the regenerative set
	\begin{align}
		& \mathcal{R} \coloneqq \{  (t, \omega) \in \mathbb{R} \times \Omega : \, t = S(u)\big(\omega\big) \text{ for some } u\geq 0 \},
	\end{align}
	i.e. the random set of image points of $S(t)$, $t\ge0$, and denote the $\omega$-slice of $\mathcal{R}$ by
	\begin{align}
		& \mathcal{R}(\omega) \coloneqq \{ t \in \mathbb{R} : (t, \omega) \in \mathcal{R} \},\quad \omega\in\Omega.
	\end{align}
	
	The set $\mathcal{R}$ turns out to be the set of renewal points\footnote{In other words, on the complement of $\mathcal{R}(\omega)$ the processes $X^+(t)=A(L(t))$ and $X(t)=A(L(t)-)^+$ do not move, i.e. they are stuck in the interval of constancy.} of the inverse process $L(t)$, $t\ge0$, of $S$. Note also that since $S$ is càdlàg and has increasing sample paths,  the complement of the slice $\mathcal{R}(\omega)$ in $\RR$ is a countable union of intervals of the form
	$\big[S(u-)(\omega), S(u) (\omega)\big)$, where $u \ge 0$ ranges over the jump epochs of the process $S$.
	
	Let, for $t \geq 0$,
	\begin{align}
		&H(t)\big(\omega\big) \coloneqq \sup \{ s \leq t : s \in \mathcal{R}(\omega) \}, \label{def_under}\\
		& D(t)\big(\omega\big) \coloneqq \inf \{ s >t : s \in \mathcal{R}(\omega) \}\label{def_over},
	\end{align}
	where we set $H(t)\big(\omega\big) = s$ on $\{ \omega : \{ s \leq t : s \in \mathcal{R}(\omega) \} = \emptyset \}$, since $\mathds{P}^{(x,s)} (\inf \mathcal{R}(\omega) = s)=1$. 
	
	The processes $H(t)$, $t \in \mathbb{R}$, and $D(t)$, $t \in \mathbb{R}$, are càdlàg. Further, denote by $x_i$, $i\in\NN$, the countable set of discontinuities of the (càdlàg) function $S(t)$, $t\ge 0$,  and set $y_i^+ = S(x_i)$, $y_i^-=S(x_i-)$. Then we have that
	\begin{align}
		S(L(t)) = \begin{cases}
			S(x_i), \qquad & t \in [y_{i}^-, y_i^+), \\
			t, &  \text{else},
		\end{cases}
		\label{traj_over}
	\end{align}
	as well as
	\begin{align}
		S(L(t)-) = \begin{cases}
			S(x_i-), \qquad & t \in (y_{i}^-, y_i^+], \\
			t, &  \text{else.}
		\end{cases}
		\label{traj_under}
	\end{align}
	It is now easy to see that
	\begin{align}
		D(t) = S(L(t)) , \qquad H(t-) = S(L(t)-).
	\end{align}
	With this at hand, we can define the age process $\gamma(t)$, $t \in \mathbb{R}$, and the remaining lifetime process $\Gamma(t)$, $t \in \mathbb{R}$, as
	\begin{align}
		\Gamma (t) = D(t)-t, \qquad \gamma(t) = t-H(t).
	\end{align}
	Define the filtrations $\mathcal{H}_t= \mathcal{F}_{L(t)}$, $t \in \mathbb{R}$, and $\mathcal{G}_t = \mathcal{F}_{L(t)-}$, $t  \in \mathbb{R}$.
	Define the family of operators $Q_{s,t}$, $0 \leq s < t$, acting on the space of bounded Borel measurable functions on $\mathbb{R}^d \times [0, +\infty)$, as
	\begin{align}
		&  Q_{s,t}^+f (y,0)  = \mathds{E}^{(y,s)} f(X^+(t), \Gamma (t)), \\
		& Q_{s,t}^+ f(y,r) = \mathds{1}_{[r>t-s]} f(y,r-(t-s))+ \mathds{1}_{[0 \leq r \leq t-s]} Q_{s+r,t}^+f(y,0),
	\end{align}
	and also the family $Q_{s,t}$, $0 \leq s < t$, as
	\begin{align}
		&Q_{s,t}f(x,0) = \mathds{E}^{(x,s)} f(X(t-), \gamma(t-)), \\
		&Q_{s,t}f(x,v) = f(x,v+t-s) \frac{K(x,s-v; \mathbb{R}^d \times [v+t-s,+\infty))}{K(x,s-v; \mathbb{R}^d \times [v,+\infty))}  \\ & \qquad +  \int_{\mathbb{R}^d} \int_{[v,v+t-s)} Q_{s+w-v,t} f(x+y,0)  \frac{K(x,s-v;dy,dw)}{K(x,s-v; \mathbb{R}^d \times [v,+\infty))}.
	\end{align}
	It turns out that the families of operators $Q_{s,t}$ and $Q^+_{s,t}$ are two-parameter semigroups of operators. In other words, they describe the evolutions of the processes $(X^+(t), \Gamma(t))$ and $(X(t-), \gamma(t-))$, respectively, which are therefore Markovian (with respect to $\mathcal{H}_t$ and $\mathcal{G}_t$, respectively). Here is the precise statement, which is due to \cite{Meerschaert2014}.
	\begin{theorem}
		\label{theorem_semim}
		The process $(X^+(t), \Gamma(t))$, $t \geq 0$, is a Hunt process, {i.e. a strong Markov process that is quasi-left.continuous}, with respect to the filtration $\mathcal{H}_t$ and with the (two-parameter) semigroup $Q^+_{s,t}$. The (left-continuous) process $(X(t-), \gamma(t-))$, $t \geq 0$, is a simple Markov process with respect to the filtration $\mathcal{G}_t$, $t \geq 0$, and with the (two-parameter) semigroup $Q_{s,t}$.
	\end{theorem}

	The case in which the limit process $(A,S)$, is Markov additive means that the kernel and the coefficients in \eqref{generatorlimit} do not depend on the second variable $t$. Therefore, the jumps of the process $S(t)$, $t\ge0$, do not depend on its position, so the intervals of constancy of $L(t), t \geq 0$, do not depend on time. It can  then be concluded that the processes $(X(t-), \gamma(t-))$, $t\ge0$, and $(X^+(t), \Gamma(t))$, $t\ge0$, are time-homogeneous. In this setting the semigroups appearing in Theorem \ref{theorem_semim} simplify, respectively, to
	\begin{align}
		&P^+_t f(y,0) = \mathds{E}^{(y,0)} f (X^+(t), \Gamma(t)), \\
		& P^+_tf(y,r) = \mathds{1}_{[r>t]} f(y,r-t) + \mathds{1}_{[0\leq r \leq t]} P^+_{t-r}f(y,0),
	\end{align}
	and
	\begin{align}
		&P_t f(x,0) = \mathds{E}^{(x,0)} f (X(t-), \gamma(t-)), \\
		& P_tf(x,v) =f(x,v+t) \frac{K(x; \mathbb{R}^d \times [v+t,+\infty))}{K(x; \mathbb{R}^d \times [v,+\infty))}  \\ & \qquad +  \int_{\mathbb{R}^d} \int_{[v,v+t)} P_{t+v-w} f(x+y,0)  \frac{K(x;dy,dw)}{K(x; \mathbb{R}^d \times [v,+\infty))},
	\end{align}
	as shown in \cite[Theorem 4.1]{Meerschaert2014}.
	
	It is interesting to note that the process $(X^+(t), \Gamma(t))$, $t \geq 0$, in spirit resembles (in a pathwise sense) the OCTRW process $(Y^+(t),\Xi(t))$, $t\ge0$, from Subsection \ref{ss:CTRW-intro}, but its semigroup resembles the renewal equation of $(Y^+(t-),\Xi(t-))$, $t\ge0$, cf. Theorem \ref{thmGS+} (and \eqref{710-a1} and \eqref{710-b1}). This `merging' comes from the fact that the process $\Xi$ does not assume value 0 (and there are finitely many jumps in a given open interval), but after taking the scaling limit the corresponding process $\Gamma$ does assume value 0. This can be immediately seen by looking at \eqref{traj_over}. It is also interesting to observe that the left-continuous process $\gamma(t-)$ can touch zero, as we can see from \eqref{traj_under}; this is, again, different from what happens in the discrete case for the process $\xi(t-)$ (i.e., the left-continuous version of the process $\xi(t)$ appearing in Theorem \ref{thmGS}), which never touches zero.
	
	Here we note that the filtration $\mathcal{G}_t$, $t\ge0$, driving the process $(X,\gamma)$, is the true `continuous' counterpart to the `discrete' filtration  \[\sigma(N(s),S_0,S_1,\dots,S_{N(s)},T_0,T_1,\dots,T_{N(s)}:s\le t),\]$t\ge0$, driving the CTRW $(Y,\xi)$, as it provides information on the trajectory of the process $(A,S)$ strictly before the jump $L(t)$. The same can be said for the filtration $\mathcal{H}_t$, $t\ge0$, which is a continuous counterpart of \[\sigma(N(s),S_0,S_1,\dots,S_{N(s)+1},T_0,T_1,\dots,T_{N(s)+1}:s\le t),\] $t\ge0$, that contains the information driving the OCTRW.

	\subsection{The limit processes as an anomalous diffusion}
	
	We have seen that the limit processes (time-changed Markov processes) may lose the Markov property because of time-change. The diffusive behavior of the original process $A(t)$, $t\ge0$, may be perturbed. Heuristically, the explanation is as follows. Suppose that the process $A(t)$, $t\ge0$, is diffusive in the sense that the mean squared displacement\footnote{Recall that $A(t)$, $t\ge0$, is considered to be in $\RR^d$ for $d\ge1$, and $\|\cdot\|$ is the Euclidean norm.} $(\Delta A )^2 (t) \coloneqq \mathds{E}^{(x,0)} \left\| A(t)-x \right\|^2$ increases linearly with time, i.e., $(\Delta A)^2(t) \sim Ct$ for some $C>0$, as $t \to +\infty$. If we consider now the process $X^+(t) = A(L(t))$, then one could expect that $(\Delta X^+)^2 (t) \sim C\,\mathds{E}^{(x,0)} L(t)$, and the very same behavior is expected  $X(t)$, $t\ge0$, i.e. $(\Delta X)^2 (t)\sim C\,\mathds{E}^{(x,0)} L(t)$. To the best of our knowledge, a thorough study, in full generality, of the mean squared displacements of CTRWs limit processes has not been done yet. However, it is useful to see what happens in the case dealt with in Examples \ref{example_uncoupled} and \ref{example_limit_coupled} where the computations can be done explicitly.
	
	We start with the uncoupled case of Example \ref{example_uncoupled}, and thus take the limit process $X^+(t) = X(t) = B(L(t))$, $t \geq 0$. Here $B(t)$, $t\ge0$, is a standard Brownian motion in $\RR^d$, independent of $S(t)$, $t\ge0$, which is a stable subordinator (with parameter $\alpha\in(0,1)$), and $L(t)$, $t\ge0$, is its inverse. In this case, by using independence and a conditioning argument we have
	\begin{align}
		\mathds{E}^{(x,0)} \left\| B(L(t)) - x\right\|^2 \, = \, & \int_0^{+\infty}\mathds{E}^{(x,0)} \left\| B(s) - x \right\|^2 \mathds{P}^{(x,0)} (L(t) \in ds) \notag \\
		= \, & d\int_0^{+\infty}  \, s \, \mathds{P}^{(x,0)} (L(t) \in ds) \notag \\
		= \, & d\,\mathds{E}^{(x,0)}L(t) \notag \\
		= \, & d \frac{t^{\alpha}}{\Gamma (\alpha+1)}.
		\label{msd_stable}
	\end{align}
	In the last step we used that
	\begin{align}
		\mathds{E}^{(x,0)} L(t)  = & \int_0^{+\infty} \mathds{E}^{(x,0)} \mathds{1}_{[L(t)>s]} ds 
		=  \mathds{E}^{(x,0)}\int_0^{+\infty}  \mathds{1}_{[S(s)\leq t]} ds = U((-\infty,t]), \notag
	\end{align}
	where $U$ denotes the potential measure of the stable subordinator $S$, whose explicit form is known (see, e.g., \cite[Example 5.8]{bogdan2009potential}) and thus \eqref{msd_stable} is obtained. Thus \eqref{msd_stable} is obtained. We have therefore shown that the process $X^+(t)=X(t)=B(L(t))$, $t\ge0$, is a subdiffusion with the mean squared displacement growing as $\frac{d}{\Gamma(\alpha+1)}\,t^\alpha$, where $\alpha \in (0,1)$.
	
	The trajectory of a Brownian motion time-changed by an independent inverse subordinator can be seen in Figure \ref{fig:subdiffusion}. Note that such a time-changed Brownian motion visits the same positions and in the same order as the original Brownian motion, but the speed of visiting the positions is lower (for large times) due to sticking effects produced by the intervals in which the inverse subordinator is constant. Some authors call such a process a slowed-down (or delayed) diffusion (see, e.g., \cite{CapitanelliDOvidio2020, kochubei2008distributed, magdziarz2015asymptotic, toaldo2015}).
	
	Let us now compute the mean squared displacements for the limit processes appearing in Example \ref{example_limit_coupled}. Note that here the limit processes are
	\begin{align}
		X^+(t) = B(S(L(t))), \qquad X(t) = (B(S(L(t)-)))^+,
	\end{align}
	where we recall that $S$ and $L$ are the $\alpha$-stable subordinator and its inverse, respectively, independent of the standard Brownian motion $B$. Also, note that $(B(S(L(t)-)))^+=(B(S(L(t)-)))$ a.s. since at fixed time the undershooting of a driftless subordinator does not have a jump, see \cite[Proposition 2 \& Theorem 4, in Chapter 3]{bertoin1996}. 
	
	Thus, a conditioning argument similar to that in \eqref{msd_stable} yields
	\begin{align}
		\mathds{E}^{(x,0)} \left\| B(S(L(t)))-x \right\|^2  &=  d\, \mathds{E}^{(x,0)} S(L(t)), \label{261}\\
		\mathds{E}^{(x,0)} \left\| B(S(L(t)-))^+-x \right\|^2  &    = d\, \mathds{E}^{(x,0)} S(L(t)-).
		\label{262}
	\end{align}
	The two expectations appearing in \eqref{261} and \eqref{262} can be computed explicitly with the help of the joint distribution of $S(L(t))$ and $S(L(t)-)$, see \cite[page 96]{bogdan2009potential} and \cite[Example 5.8]{bogdan2009potential}. In other words, it holds that
	\begin{align*}
		\mathds{P}^{(x,0)} (S(L(t)-) \in ds, S(L(t)) \in dx) \, = \, \frac{s^{\alpha -1}}{\Gamma(\alpha)} \frac{\alpha (x-s)^{-\alpha-1}}{\Gamma (1-\alpha)} \mathds{1}_{[0\le s\le t<x]} \, ds \, dx,
	\end{align*}
	so we get
	\begin{align}
		&  \mathds{E}^{(x,0)} S(L(t))  = +\infty, \\
		& \mathds{E}^{(x,0)} S(L(t)-) = \alpha t,
	\end{align}
	proving that the process $X^+$ is superdiffusive with infinite mean squared displacement while the process $X$ is diffusive with the mean squared displacement $d\alpha\,t$. Trajectories of overshooted and undershooted Brownian motion can be seen in Figure \ref{fig:subdiffusion}.
	Heuristically, this can be explained as follows. The time-change with an $\alpha$-stable inverse subordinator induces a subdiffusion with mean squared displacement of order $t^\alpha$, as the time-changed process is obtained by `running' on the same Brownian path with the `speed' given by the inverse subordinator. In the time-change with the undershooting, the trajectory of the process leaves the original Brownian path with an interval in which the undershooting is constant and then jumps back onto the Brownian path: hence the resulting process returns to a Brownian path and thus has the same asymptotic behavior. In the time-change with the overshooting, the process jumps away from the Brownian path and reaches it with an interval of constancy: since the jumps are heavy-tailed, an infinite mean squared displacement is obtained.
	
	\begin{figure}
		\centering
		\begin{subfigure}{\textwidth}
			\centering
			\includegraphics[width=.8\linewidth]{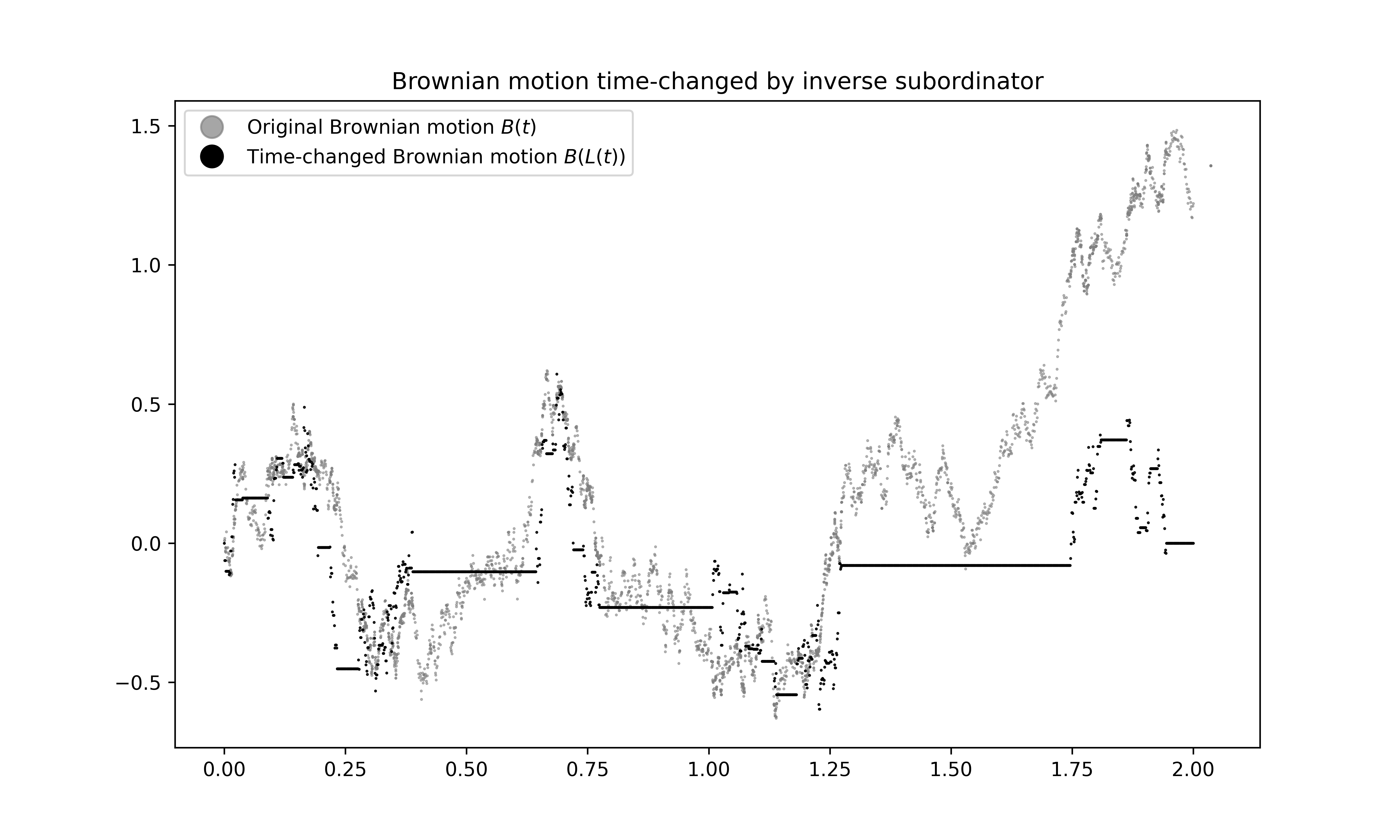}
			\caption{Exactly sampled trajectories (in discrete time points) of a standard Brownian motion (in $d=1$ and $t\in[0,2])$ and the same Brownian motion time-changed by an inverse (stable $\alpha=0.75$) subordinator; using algorithms from  \cite{daniel24} providing exact samples from the finite-dimensional distributions of inverse subordinators.}
			\label{Top-inverse}
		\end{subfigure}
		\\
		\begin{subfigure}{\textwidth}
			\centering
			\includegraphics[width=\linewidth]{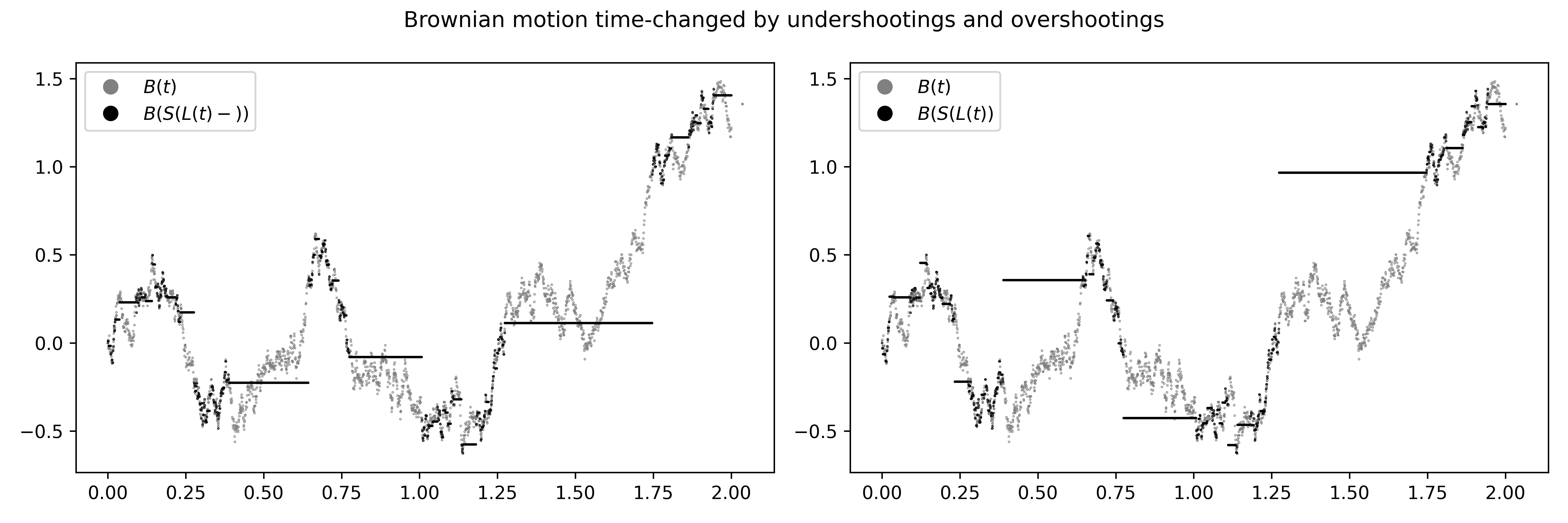}
			\caption{Two exact samples of trajectories of a Brownian motion (in $d=1$ and $t\in [0,2]$) and the same Brownian motion time-changed by the overshootings of a stable ($\alpha=0.75$) subordinator (left figure), and a Brownian motion together with the same Brownian motion time-changed by the undershootings of a stable ($\alpha=0.75$) subordinator (right figure). The undershootings and overshootings come from the same exact samples (from the finite-dimensional distribution) obtained by algorithms from  \cite{daniel24}.}
			\label{Bottom-UO}
		\end{subfigure}
		\caption{}
		\label{fig:subdiffusion}
	\end{figure}

	\section{Non-local equations for CTRWs' limit processes}
	\label{sec3_non-local}
	One of the main tools for the analysis of the distribution of (O)CTRWs limit processes, are their governing equations. Given a process $X^+(t)=A(L(t))$, $t\ge0$, or $X(t)=A(L(t)-)^+$, $t\ge0$, arising as the limit of a suitable OCTRW (CTRW respectively), one can define the functions on $\RR^{d}\times [0,+\infty)$
	\begin{align}
		q^+(x,t) \, &= \, \mathds{E}^{(x,0)} u(X^+(t)),\\
		q(x,t) \, &= \, \mathds{E}^{(x,0)}u(X(t)),
	\end{align}
	and can ask the question: what is the integro-differential equation, if any, satisfied by these functions? Whenever the limit processes $X^+$ or $X$ are Markovian, the governing equation is provided by classical semigroup theory. This can happen, for example, in the uncoupled example of Example \ref{example_uncoupled} assuming that the waiting times $W_n$, $n\in\NN$, are i.i.d. and exponential with parameter $\theta>0$. Indeed, under this assumption 
	\begin{align}
		\mathds{E} e^{-\lambda n^{-1} (W_1 + \dots + W_n) } = \left( \frac{1}{1+\lambda \theta / n} \right)^n \to e^{-\theta \lambda}.
	\end{align}
	Hence, by the same arguments as in Example \ref{example_uncoupled}, the limit process is now $\mathcal{X}(t) = (A(t), S(t))$, $t \geq 0$, where $A(t)$ is a standard Brownian motion while $S(t)$ is the deterministic process $S(t) = \theta t$. Therefore, the CTRW limit (which is also equal to the OCTRW limit) is $X(t) = A(\theta^{-1}t)$, i.e., a Brownian motion with zero mean and variance $t/\theta$. More generally, whenever the process $X(t)$, $t\ge0$, is an arbitrary Feller process, associated with the semigroup $P_t$, $t \geq 0$, on $C_0(\mathbb{R}^d)$, with the generator $(G, \mathcal{D}(G) )$, then the function $q(t) = P_tu$, is the unique classical solution to the abstract Cauchy problem
	\begin{align}
		\frac{d}{dt} q(t) = Gq(t), \qquad q(0) = u \in \mathcal{D}(G),
		\label{class_cauchy}
	\end{align}
	see e.g. \cite[Proposition 3.1.9]{arendt2001cauchy}. This means that the mapping $t \mapsto q(t)$ is continuous as a function $q:[0, +\infty) \mapsto C_0 (\mathbb{R}^d)$, and the derivative on the left-hand side of \eqref{class_cauchy} exists in the strong sense in $C_0 (\mathbb{R}^d)$.
	
	In general, the processes $X$ and $X^+$ are not Markovian and semigroup theory cannot be applied directly. Several different approaches have been adopted in the literature to overcome this obstacle under different assumptions (see, e.g., \cite{ascione2021time, Baeumer2001, mirko, chen, hernandez2017, kochubei, kolokoltsov2015, mainardi2000fractional,  meerschaert2002governing, savtoa, ScalasGorenfloMainardi2004, STRAKA2018451, toaldo_pota, toaldo2015, Ton19jmaa}). A classical example of governing equations in the non-Markovian context is the fractional kinetic equation (FKE), which governs fractional-time processes. Such processes arise as in Example \ref{example_uncoupled} when the limiting (Feller) process $(A(t), S(t))$, $t \geq 0$, is such that $A(t)$ is not merely a Brownian motion, but can rather be replaced by an arbitrary Feller process $A(t), t \geq 0$. In this context, the CTRW limit process (and also the OCTRW limit, as they coincide by Lemma \ref{lemma_ctrw=octrw}), is the process $X(t) = A(L(t))$, where $L(t)$ is the inverse of a stable subordinator (for more on inverse subordinators, see \cite{bertoin1999, tlms}). This case has been covered by the very first results on the fractional-time equations related to subordination techniques, in \cite{bazlekova, Baeumer2001}, where the fractional derivative of order $\alpha \in (0,1)$, is the integro-differential operator defined by
	\begin{align}
		\partial_t^\alpha q(t) \coloneqq \partial_t \int_0^t q(s) \frac{(t-s)^{-\alpha}}{\Gamma (1-\alpha)} ds.
		\label{frac_der}
	\end{align}
	An alternative definition to \eqref{frac_der} is the Caputo-type fractional derivative
	\begin{align}
		^C\partial_t^\alpha q(t) \, = \, \int_0^t \partial_sq(s) \, \frac{(t-s)^{-\alpha}}{\Gamma(1-\alpha)} ds,
		\label{frac_der_caputo}
	\end{align}
	where the well-known relationship between   \eqref{frac_der} and \eqref{frac_der_caputo} holds:
	\begin{align}
		^C\partial_t^\alpha q(t) = \partial_t^\alpha (q(t) - q(0)),
	\end{align}
	on suitable function spaces, see Lemma \ref{equiv_def}.
	In this context, in particular, \cite[Theorem 3.1]{Baeumer2001} implies the following claim.
	\begin{theorem}
		Let $l_t^\alpha (\cdot)$ be the distribution of the inverse $\alpha$-stable subordinator.  Let $P_t$, $t \geq 0$, be a uniformly bounded strongly continuous semigroup on a Banach space $(\mathfrak{B}, \left\| \cdot \right\|)$, associated with the generator $(G, \mathcal{D}(G))$. Then the family of operators $\mathcal{P}_t$, $t \geq 0$, given by
		\begin{align}
			\mathcal{P}_t u = \int_0^{+\infty} P_s u \, l_t^\alpha(ds),
		\end{align}
		is a uniformly bounded family of linear operators from $\mathfrak{B}$ to $\mathfrak{B}$ and the mapping $t \mapsto \mathcal{P}_tu$ solves, whenever $u \in \mathcal{D}(G)$, the problem
		\begin{align}
			\partial_t^\alpha (q(t)-q(0))\, = \, G q(t), \qquad q(0)  = u.
		\end{align}
	\end{theorem}
	\begin{proof}
		This theorem is a corollary of the more general Theorem \ref{theorem_convolution}, which will be stated and proved later.
	\end{proof}
	
	\subsection{Uncoupled CTRWs and generalized FKEs}
	\label{sec_generalized_fractional}
	In the uncoupled case of Example \ref{example_uncoupled}, where $A$ and $S$ are independent, it is possible to generalize the approach therein and obtain that the limit process $S(t)$, $t\ge0$, is a (general) strictly increasing subordinator. A subordinator is a non-decreasing L\'evy process characterized by its Laplace exponent $\lambda\mapsto \phi (\lambda)$  which is a Bernstein function (see {Appendix} \ref{appendix_bernstein} for basic facts on subordinators and Bernstein functions). In other words, we have
	\begin{align}
		\mathds{E}^{(x,0)}e^{-\lambda S(t)} = e^{-t\phi(\lambda)},\quad t,\,\lambda\ge0,
	\end{align}
	where the exponent $\phi$ can be represented as
	\begin{align}
		\phi (\lambda) = \,& a + b \lambda + \int_0^{+\infty} \left( 1-e^{-\lambda s} \right) \nu(ds) \label{bernstein-eq0}\\
		= \, & b\lambda + \lambda \int_0^{+\infty} e^{-\lambda t} \bar{\nu}(t) \, dt .
		\label{bernstein-eq}
	\end{align}
	Here,  $a,b\ge0$, while $\nu (\cdot)$ is a measure supported on $(0, +\infty)$ such that
	\begin{align}
		\int_0^{+\infty} \left( s \wedge 1 \right) \nu (ds) < +\infty,
	\end{align}
	and $\bar{\nu}(t)\coloneqq a+\nu(t, +\infty)$. 
	
	Actually, the setting in this section allows $a>0$, which means that a (classical) non-decreasing L\'evy process can be killed (i.e. sent to the state $\{+\infty\}$) at an independent exponentially distributed random time with parameter $a$. Recall that in the previous sections we always considered the second coordinate process $S$ to be strictly increasing (see, e.g.,  Theorem \ref{maintheoremctrwlim}), and hence $L$ was continuous. In the setting of subordinators, this assumption is equivalent to assuming $b>0$ or $\nu(0, +\infty) = +\infty$, where the latter is called the infinite activity of the subordinator. This will remain our assumption throughout the section. We note that the case $b=0$ and $\nu(0, +\infty) < +\infty$ corresponds to the so-called compound Poisson process, in which $L$ is not continuous and has exponentially distributed jumps. 
	
	In this general setting, the inverse process $L(t)=\inf\{s\ge0:S(s)>t\}$, $t \geq 0$, of the subordinator $S$ is usually called inverse subordinator. Note that the fractional case (i.e. the $\alpha$-stable subordinator) can be recovered by taking $\phi(\lambda) = \lambda^\alpha$ (where $a=b=0$ and $\nu(ds)=\frac{\alpha}{\Gamma (1-\alpha)}s^{-1-\alpha}$), while some other interesting examples can be found in \cite{bernstein}, particularly in Chapter 20 therein, where one can find a list of many Bernstein functions.
	
	We return to the main question where we ask ourselves what is the governing equation of $X^+(t)=A(L(t))$, $t \geq 0$ (and recall that $X=X^+$ a.s., by Lemma \ref{lemma_ctrw=octrw} since $A$ and $S$ are independent). To answer this question, one should understand which equation is satisfied by the function $q(x,t) = \mathds{E}^{(x,0)}u(A(L(t)))$, $(x,t)\in\RR^{d}\times [0,+\infty)$. Since $A$ and $S$ are independent, by a simple conditioning argument we get 
	\begin{align}
		q(x,t) = \int_0^{+\infty} \mathds{E}^{(x,0)}u(A(s)) \, \mathds{P}^{(x,0)} (L(t) \in ds).
	\end{align}
	This suggests looking at the problem in a more general setting: we look at the evolution equation on some Banach space $\mathfrak{B}$, for the mapping
	\begin{align}
		[0, +\infty) \ni  t \mapsto \int_0^{+\infty} P_s u \, l_t(ds) \in \mathfrak{B},
	\end{align}
	where $P_s$, $s \geq 0$, is a suitable semigroup of operators, $l_t(ds)$ denotes the law of $L(t)$ under $\mathds{P}^{(x,0)}$ which is independent of $x$, and the integral is meant in the Bochner sense.\footnote{For a detailed approach to Bochner integrals, we refer the reader to \cite{Bochner}.} Whenever $\mathfrak{B} = C_0(E)$, for some $E$ locally compact and separable space, there exists a Markov process $A(t)$, $t\ge0$, on $E$ such that $P_tu(x) = \mathds{E}^{(x,0)}u(A(t))$, for all $u\in C_0(E)$ and $x \in E$ (see basic facts on semigroups in {Appendix} \ref{appendix_semigroups}). We also remark that the one-dimensional law of $L(t)$, i.e., $l_t(\cdot)$ as denoted above, is nowadays well studied (see, in particular, \cite{ascione2024regularity, rivero_doney}).

	It turns out that this equation has the form of an FKE with a time operator having the form \eqref{frac_der} but with a different convolution kernel, namely the one associated with a Bernstein function $\phi$. We define this general operator as in \cite[Definition 2.1]{toaldo_pota}
	\begin{align}
		\phi(\partial_t) f(t) \, \coloneqq \, \partial_t \left( b f(t) + \int_0^t f(s) \bar{\nu}(t-s) ds \right),
		\label{conv_der}
	\end{align}
	where $\bar{\nu}(t) \coloneqq a+\nu(t, +\infty)$, for $f$ such that the operations of derivation and integration make sense, see Lemma \ref{equiv_def}. On suitable function spaces (with regularity conditions and $f(0)=0$) this definition coincides with
	\begin{align}
		\,^C\phi(\partial_t)f(t) \, = \,\partial_t b f(t) + \int_0^t \partial_s f(s) \bar{\nu}(t-s) ds 
		\label{conv_der_toa}
	\end{align}
	in the same spirit in which \eqref{frac_der} and \eqref{frac_der_caputo} coincide. 
	Now it is possible to make this statement precise. The following lemma is an extension of \cite[Lemma 2.4]{moving} for Banach-valued functions.
	Here we emphasize that in calculus with Banach-valued functions one has to be careful: e.g. not all Lipschitz functions are differentiable almost everywhere, which follows from the fact that not all spaces have the Radon-Nikodym property. In other words, the absolute continuity assumption of \cite[Lemma 2.4]{moving} is replaced here by the condition \eqref{abs_integ}.
	\begin{lemma}
		\label{equiv_def}
		Fix $T>0$, and let $(\mathfrak{B}, \left\| \cdot \right\|)$ be a Banach space. Suppose that $f: [0, +\infty) \mapsto \mathfrak{B}$,  is such that there exists a Bochner integrable function $g: [0, +\infty) \mapsto \mathfrak{B}$ such that
		\begin{align}\label{abs_integ}
			f(t) - f(0) = \int_0^t g(s) \, ds,\quad t \in [0,T].
		\end{align}
		Then $\partial_tf(t)$ and $\phi(\partial_t) f(t)$ exists a.e. in $[0,T]$ and are elements of $L^1([0,T]; \mathfrak{B})$. Moreover, for almost all $t\in [0,T]$, it holds that $\partial_tf(t)=g(t)$ and
		\begin{align}
			\phi(\partial_t)\big(f(t)-f(0)\big) = b \partial_tf(t) + \int_0^t \partial_sf(s) \, \bar{\nu}(t-s) \, ds.
			\label{der_dentro}
		\end{align}
		If, furthermore, $g\in C([0,T]; \mathfrak{B})$, then $t \mapsto \phi(\partial_t ) f(t) \in C([0,T]; \mathfrak{B})$ and \eqref{der_dentro} holds for all $t\in [0,T]$.
		
	\end{lemma}
	\begin{proof}
		Note that for all $t\in[0,T]$ we have
		\begin{align}
			\int_0^t\int_0^s\bar{\nu}(s-h)\|g(h)\|dh\,ds&=\int_0^t \bar\nu(s)\int_0^{t-s}\|g(h)\|dh\, ds\\
			&\le \left(\int_0^t \bar\nu(s)ds\right)\left(\int_0^{t}\|g(h)\|dh\right)<+\infty.
		\end{align}
		Fubini's theorem (for Bochner integrals \cite[Theorem 3.7.13]{hille1996functional}) now implies
		\begin{align}
			\begin{split}\label{1003}
				\int_0^t\left(b\,g(s)+\int_0^s\bar{\nu}(s-h)g(h)dh\right)\,ds&=b\big(f(t)-f(0)\big)+\int_0^t\bar\nu(s)\int_0^{t-s}g(h)dh\,ds\\&\hspace{-6em}=b\big(f(t)-f(0)\big)+\int_0^t \bar\nu(t-s)\big(f(s)-f(0)\big)ds.
			\end{split}
		\end{align}
		
		Note that since $g$ is integrable, $f$ is a local primitive of $g$, see \cite[Theorem 3.3 in Chapter XIII]{Bochner}, so by \cite[Theorem 3.4 in Chapter XIII]{Bochner} the derivative of both sides in \eqref{1003} exists a.e. In other words, $\phi(\partial_t)f$ exists a.e., is integrable, and satisfies \eqref{1003}.
		
		If $g$ is continuous, $\partial_t f\equiv g$ by \cite[Theorem 2.2 in Chapter XIII]{Bochner}, and the second claim easily follows.
	\end{proof}

	The following theorem relates the operator $\phi(\partial_t)$ to $q(x,t)=\mathds{E}^{(x,0)}u(A(t))$ in the same way the fractional derivative $\partial^\alpha_t$ is related to $q(x,t)$ in the case when $L(t)$, $t\ge0$, is the inverse stable subordinator. The most general previous proof of the claim was due to \cite{chen}; the relation, however, was also known due to \cite[Theorem 5.1]{toaldo_pota} where the time operator was $\,^C\phi(\partial_t)$, but the proof contains a mistake\footnote{In particular, the steps of the proof of \cite[Theorem 5.1]{toaldo_pota} showing the existence of the strong derivative $^C\phi(\partial_t)$ are not correct, i.e., the exchange of limit and integral at the top of page 134 is not justified, in full generality. See also Remark \ref{r:914} for details.}. Here, we provide a statement that is slightly more general than \cite{chen}, but we adopt the proof of \cite{toaldo_pota}: in particular, here we allow the subordinator to be killed and we show that the Laplace transform method used in the proof of \cite[Theorem 5.1]{toaldo_pota} can be fixed and adapted to this context, and works well.
	\begin{theorem}
		\label{theorem_convolution}
		Assume that the subordinator $S(t)$, $t\ge0$, is given by a general exponent $\phi$ as in \eqref{bernstein-eq} with $b>0$ or $\nu(0,+\infty)=+\infty$, and let $L(t)$, $t \geq 0$ be its inverse subordinator. Let $P_s$, $s \geq 0$, be a strongly continuous semigroup of operators on some Banach space $(\mathfrak{B}, \left\| \cdot \right\|)$, with the generator $(G,\DD(G))$, {such that}
		\begin{align}
			\left\| P_t u \right\| \, \leq \, M e^{\lambda_0 t} \left\| u \right\|, \quad {u \in \mathfrak{B},}
		\end{align}
		{where $M>0$ and $\lambda_0\ge0$}\footnote{{All strongly continuous semigroups are exponentially bounded in this way, see Appendix \ref{appendix_semigroups}.}}. Define
		\begin{align}
			\mathcal{P}_tu \, = \, \int_0^{+\infty} P_su \, l_t(ds),
		\end{align}
		as a Bochner integral on $\mathfrak{B}$, where $l_t(ds)$  denotes the distribution of $L$ under $\PP^{(x,0)}$ (which is independent of $x$).
		
		Then, whenever $u \in \mathfrak{B}$, the mapping $[0,+\infty)\ni t \mapsto q(t) \coloneqq \mathcal{P}_tu$ is strongly continuous on $\mathfrak{B}$ and {exponentially} bounded ({bounded by a constant if $\lambda_0=0$}). If, in addition, $u \in \mathcal{D}(G)$, then
		\begin{align}
			\phi(\partial_t) (q(t)-q(0)) \, = \, G q(t), \qquad q(0) = u,
			\label{eq_conv_theorem}
		\end{align}
		in the sense that $q(t) \in \mathcal{D}(G)$ for all $t \geq 0$, $t \mapsto q(t)$ is strongly continuous (on $\mathfrak{B}$) and {exponentially} bounded {(bounded by a constant if $\lambda_0=0$)}, the generalized derivative on the left hand side of \eqref{eq_conv_theorem} exists for all $t > 0$ in the strong sense, and both $t \mapsto Gq(t)$ and $t \mapsto \phi(\partial_t) (q(t)-q(0))$ are strongly continuous and {exponentially} bounded {(bounded if $\lambda_0=0$)}.
		
		Finally, if $\mathpzc{q}(t)$ is another solution with the same properties as above, then $\mathpzc{q}(t) = q(t)$ for all $t \geq 0$.
	\end{theorem}
	\begin{proof}
		Since $\phi$ is an unbounded function, there is $\lambda_1>0$ such that $\phi(\lambda_1) > \lambda_0$, and we can increase this $\lambda_1$ so that, by Lemma \ref{daniel}, it also holds that
		\begin{align}\label{1236-0}
			\int_0^{+\infty} e^{\lambda_0 s} \, l_t(ds) \, \leq \, \mathfrak{C}e^{\lambda_1 t},\quad t\ge0,
		\end{align}
		where $\mathfrak{C}=\mathfrak{C}(M, \phi, \lambda_0)$.
		Let $u \in \mathfrak{B}$ and note that 
		\begin{align}
			\left\|  \mathcal{P}_t u \right\| \, \leq \, \int_0^{+\infty} \left\| P_s u \right\| l_t(ds)  \leq \, & {\mathfrak{C}\,Me^{\lambda_1 t}} \left\| u \right\|,\quad t\ge0,
			\label{1236}
		\end{align}
		which shows that $\mathcal{P}_t$ is well-defined as a Bochner integral and that the mapping $t \mapsto \mathcal{P}_tu$ is exponentially bounded. By using the probabilistic interpretation that $l_t(ds)$ is the distribution of $L$, we can also write (for any $x$, independently of $x$)
		\begin{align}
			q(t) \, = \,\int_\Omega P_{L(t, \omega)} u \, \mathds{P}^{(x,0)}(d\omega),
			\label{exp_bochner}
		\end{align}
		as a Bochner integral on $\mathfrak{B}$. Note that $t \mapsto L(t)$ has, a.s., continuous and non-decreasing trajectories and thus, since  $\left\| P_{L(t,\omega)}u \right\| \leq M e^{\lambda_0L(t, \omega)} \left\| u \right\|$, for $s,t \in [0, R]$ and $R>0$ arbitrary it holds that
		\begin{align}
			\left\| P_{L(t, \omega)} - P_{L(s, \omega)} \right\| \,  \leq \, 2M e^{\lambda_0 L(R, \omega)} \left\| u \right\|.
		\end{align}
		Further, by using \eqref{1236-0}, we have
		\begin{align}
			\int_\Omega 2M e^{\lambda_0 L(R, \omega)} \left\| u \right\| \mathds{P}^{(x,0)} (d\omega) \, \leq \, 2\mathfrak{C} \,Me^{\lambda_1 R} \left\| u \right\|.
		\end{align}
		Thus, the dominated convergence theorem for the Bochner integral \cite[Theorem E.6]{cohn2013measure} implies
		\begin{align}
			\left\| q(t) - q(s) \right\| \, \leq \, \int_\Omega \left\| P_{L(t, \omega)}u - P_{L(s, \omega)}u \right\| \mathds{P}^{(x,0)}(d\omega) \to 0,
			\label{1249}
		\end{align}
		as $s \to t$.
		
		In the rest of the proof, assume that $u \in \mathcal{D}(G)$. We have $GP_su=P_sGu$, so  $\left\|  GP_s u\right\| = \left\|  P_s Gu\right\| \leq { M e^{\lambda_0 s}} \left\| Gu \right\|$. Hence, {by \eqref{1236}, we have that $GP_su$ is Bochner integrable on $(0, +\infty)$ with respect to $l_t(\cdot)$} and thus, by \cite[Proposition 3.7.12]{hille1996functional}, one has that $\mathcal{P}_tu \in \mathcal{D}(G)$ for every $t \geq 0$, whenever $u \in \mathcal{D}(G)$. Furthermore
		\begin{align}
			G\mathcal{P}_t u = \int_0^{+\infty} GP_su \, l_t(ds) \, = \, \int_0^{+\infty} P_sGu \, l_t(ds),
			\label{318-0}
		\end{align}
		which implies
		\begin{align}
			\left\| G \mathcal{P}_tu  \right\| \leq {\mathfrak{C}\,Me^{\lambda_1 t}} \left\| Gu \right\|, \quad u\in \DD(G).
			\label{319-0}
		\end{align}
		By using the representations \eqref{exp_bochner} and \eqref{318-0}, we get
		\begin{align}
			Gq(t) \, = \, \int_{\Omega} P_{L(t, \omega)} G u \, \mathds{P}^{(x,0)}(d\omega)
		\end{align}
		so the dominated convergence theorem, {applied in the same fashion as in \eqref{1249}}, yields the (strong) continuity of $[0, +\infty) \ni t \mapsto Gq(t) $, for every $u \in \mathcal{D}(G)$. 
		
		In order to show \eqref{eq_conv_theorem}, we first show that
		\begin{align}
			\left(  b(q(t)-q(0))+   \int_0^t (q(s)-q(0)) \bar{\nu}(t-s)ds \right) \, = \, \int_0^t Gq(s) ds
			\label{int_eq}
		\end{align}
		for every $t > 0$, and then, by observing that in \eqref{int_eq} we can differentiate (see \cite[Chapter XIII]{Bochner}), we obtain \eqref{eq_conv_theorem}.
		
		To obtain this, we use the uniqueness property of the Laplace transform. First, we show that the both sides of \eqref{int_eq} share the same Laplace transform. This implies that the two sides coincide for almost all $t \geq 0$ by \cite[Theorem 1.7.3]{arendt2001cauchy}. Then we show continuity of the left-hand side, and since the right-hand side is obviously continuous, we obtain equality for all $t \geq 0$. {To this end, observe that
			\begin{align*}
				& \int_0^{+\infty} e^{-\lambda t} \left\| q(t) \right\| \, dt \, \leq  \, \mathfrak{C}\,M(\lambda-\lambda_1)^{-1} \left\| u \right\| & \text{since } \left\|q(t) \right\| \leq \mathfrak{C}\,Me^{\lambda_1 t}\left\| u \right\|,\\
				& \int_0^{+\infty} e^{-\lambda t}  \left\| Gq(t) \right\| \, dt \, \leq  \, \mathfrak{C}\,M(\lambda-\lambda_1)^{-1} \left\| 
				Gu \right\|& \text{by \eqref{319-0}}, \\
				& \int_0^{+\infty} e^{-\lambda t} \bar{\nu}(t) dt\, = \, \frac{\phi(\lambda)}{\lambda}-b & \text{by \eqref{bernstein-eq},}
			\end{align*}
			for every $\lambda > \lambda_1$.}
		
		By \cite[Proposition 3.7.12]{hille1996functional}, since infinitesimal generators are closed operators, one has that
		\begin{align}\label{GenLap}
			\int_0^{+\infty} e^{-\lambda t}   Gq(t)  \, dt =G\left(\int_0^{+\infty} e^{-\lambda t}  q(t)  \, dt\right),
		\end{align}
		which, in our context, is just an application of the dominated convergence theorem and the usual generator bounds (see e.g. \cite[Eq. (15.3)]{bernstein}), since $u\in \DD(G)$.

		Denote by $\widetilde{q}(\lambda)$ the Laplace transform of $q$, so the Laplace transform of the left-hand side of \eqref{int_eq} becomes, {for $\lambda > \lambda_1$,}
		\begin{align}
			&   \int_0^{+\infty} e^{-\lambda t} \left(  b(q(t)-q(0))+   \int_0^t (q(s)-q(0)) \bar{\nu}(t-s)ds \right) \, dt \notag \\  = \,& b \widetilde{q}(\lambda) -b\lambda^{-1}q(0)+ \left( \frac{\phi(\lambda)}{\lambda}-b \right) (\widetilde{q}(\lambda) - \lambda^{-1}q(0)) \notag \\
			= \, & \frac{\phi(\lambda)}{\lambda}  \widetilde{q}(\lambda)- \frac{\phi(\lambda)}{\lambda^2} u, 
			\label{326}
		\end{align}
		where in the first equality we used \cite[Proposition 1.6.4]{arendt2001cauchy}.
		
		For the Laplace transform of the right-hand side of \eqref{int_eq}, it will be useful to first calculate the Laplace transform of $q$. To this end, note that 
		\begin{align}\label{eq:uncoupled-sol}
			q(t)=\int_0^{+\infty}P_s u \,\PP^{(x,0)}(L(t)\in ds)=u+\int_0^{+\infty}GP_su\,\PP^{(x,0)}(L(t)\ge s)ds, 
		\end{align}
		where in the last equality we used\footnote{The exact property is $\int_0^{+\infty}f(s) \,\PP^{(x,0)}(L(t)\in ds)=f(0)+\int_0^{+\infty}f'(s)\,\PP^{(x,0)}(L(t)\ge s)ds$, which follows from the classical proof using Fubini's theorem for Bochner integrals \cite[Theorem 3.7.13]{hille1996functional}.} that $s\mapsto f(s)\coloneqq P_su$ is differentiable (in the strong sense of the Banach space $(\mathfrak{B}, \left\| \cdot \right\|)$) with $f'(s)=GP_su$, since $u\in \DD(G)$, see \cite[Section 1.3]{apple-semigr}. Now we apply the Laplace transform to \eqref{eq:uncoupled-sol} to get, {for $\lambda > \lambda_1$,}
		\begin{align}\label{resolv}
			\wt q(\lambda)=\frac{u}{\lambda}+\frac{1}{\lambda}G\int_0^{+\infty}P_su\,e^{-s\phi(\lambda)}ds,
		\end{align}
		where we used Fubini's theorem \cite[Theorem 3.7.13]{hille1996functional} and the property
		\begin{align}
			\int_0^{+\infty} e^{-\lambda t} \,\PP^{(x,0)}(L(t)\ge s) \, dt & = \int_0^{+\infty} e^{-\lambda t} \PP^{(x,0)}{(S(s)\leq  t)}dt\\
			&=\EE^{(x,0)}\left[\int_0^{+\infty} e^{-\lambda t} {\mathds{1}_{S(s)\leq t}}dt\right]\\
			&=\frac{1}{\lambda}\EE^{(x,0)}\left[e^{-\lambda S(s)}\right]=\frac{1}{\lambda} e^{-s\phi(\lambda)},
		\end{align}
		where $\PP^{(x,0)}(L(t)\ge s)=\PP^{(x,0)}(S(s)\leq  t)$ holds by the definition of the inverse and since $S(s)$ jumps at a fixed time $s>0$ with probability zero.
		
		Note that the integral in \eqref{resolv} is the $\phi(\lambda)$-resolvent of the semigroup $P_t$, $t\ge0$\footnote{See more on resolvents in Appendix \ref{appendix_semigroups}.}, here denoted by $R_{\phi(\lambda)}$. Thus, by \cite[Lemma 2.1.1]{apple-semigr} we get\footnote{{Note that, by the definition of $\lambda_1$, we have $\phi(\lambda) > \lambda_0$ for all $\lambda> \lambda_1$, so the $\lambda$-potential (resolvent) is well defined for $\lambda>\lambda_1$.}}
		\begin{align}\label{1429}
			\wt q(\lambda)=\frac{u}{\lambda}+\frac{1}{\lambda}G R_{\phi(\lambda)}u=\frac{u}{\lambda}+\frac{1}{\lambda}\big(\phi(\lambda) R_{\phi(\lambda)}u-u\big)=\frac{\phi(\lambda)}{\lambda}\big(\phi(\lambda)-G\big)^{-1}u,
		\end{align}
		where in the last line we used $R_{\phi(\lambda)}=\big(\phi(\lambda)-G\big)^{-1}$. The relation \eqref{1429} further implies
		\begin{align}
			u=\frac{\lambda}{\phi(\lambda)} (\phi(\lambda)-G) \widetilde{q}(\lambda).
			\label{332}
		\end{align}
		
		From the commutativity of the Laplace transform and $G$, from \eqref{GenLap}, and by \eqref{1429}, we get
		\begin{align}
			\wt{Gq}(\lambda)=G\big(\wt q(\lambda)\big)= \, &\frac{\phi(\lambda)}{\lambda} G(\phi(\lambda) - G)^{-1}u.
			\label{329}
		\end{align}

		Thus, by \cite[Corollary 1.6.5]{arendt2001cauchy} and \eqref{329}, the right-hand side of \eqref{int_eq} has the Laplace transform, {for $\lambda > \lambda_1$,}
		\begin{align}
			\int_0^{+\infty} e^{-\lambda t} \int_0^t Gq(s) ds \, dt \, = \, & \frac{1}{\lambda} G \widetilde{q}(\lambda)  . 
			\label{laplrhs}
		\end{align}
		
		By substituting \eqref{332} in \eqref{326} we obtain that the Laplace transform of the left-hand side of \eqref{int_eq} becomes
		\begin{align}
			\frac{\phi(\lambda)}{\lambda} \widetilde{q}(\lambda) - \frac{\phi (\lambda)}{\lambda^2} \left[ \frac{\lambda}{\phi(\lambda)} (\phi(\lambda) -G) \widetilde{q}(\lambda) \right] = \frac{1}{\lambda} G \widetilde{q}(\lambda),
		\end{align}
		and since this coincides with \eqref{laplrhs}, we obtain that \eqref{int_eq} holds for almost all $t\geq 0$. 
		
		We proceed by proving continuity of the left-hand side (by using an easy argument suggested by \cite{ascione_private}) to extend the equality \eqref{int_eq} to all $t \geq 0$. To this end, observe that for any $w<t$ we have
		\begin{align}
			&  \left\| \int_0^t (q(s)-q(0)) \bar{\nu}(t-s) ds - \int_0^w (q(s)-q(0)) \bar{\nu}(w-s) ds\right\|   \notag \\
			\leq \, & \left\| \int_0^w (q(s)-q(0)) (\bar{\nu}(t-s) - \bar{\nu}(w-s)  ds \right\|  + \left\| \int_w^t (q(s)-q(0)) \bar{\nu}(t-s) ds \right\| \notag \\
			\leq \, &  2 \left\| u \right\|\mathfrak{C}\,Me^{\lambda_1 t} \int_0^w |\bar{\nu}(t-s) - \bar{\nu}(w-s) |  ds   + 2 \left\| u\right\|\mathfrak{C}\,Me^{\lambda_1 t} \int_w^t  \bar{\nu}(t-s) ds \notag \\
			= \, & 2\mathfrak{C}\,Me^{\lambda_1 t}\left\|u \right\| \left( \int_0^w \bar{\nu}(s) ds - \int_{0}^t \bar{\nu} (s) ds + 2 \int_0^{t-w} \bar{\nu}(s) ds  \right),
		\end{align}
		and this tends to zero, as $|t-w|\to0$, since $\bar{\nu}(s)$ is locally integrable on $[0, +\infty)$. This proves continuity of the integral part. Strong continuity of the whole left-hand side of \eqref{int_eq} follows from continuity of $t \mapsto q(t)$. Taking the strong derivative on both sides of \eqref{int_eq} then yields the statement \eqref{eq_conv_theorem}, in the claimed sense.
		
		Now we prove uniqueness. Let $\mathpzc{q}(t)$ be an another solution (with its implied properties). Then, $v(t) = q(t) -\mathpzc{q}(t)$ satisfies \eqref{eq_conv_theorem} with $u= 0$. Furthermore, it follows that $v(t)$ solves \eqref{int_eq}, for all $t \geq 0$. Hence, passing to the Laplace space, {for all $\lambda$ large enough,}
		\begin{align}
			\widetilde{v}(\lambda) \, = \, \frac{\phi(\lambda)}{\lambda} (\phi(\lambda)-G)^{-1} 0 \, = \, 0,
		\end{align}
		since the resolvent $(\phi(\lambda)-G)^{-1}$ is a bounded operator. Uniqueness of the Laplace transform does the rest.
	\end{proof}

	\begin{remark}\label{r:914}
		As we explained above, \cite[Theorem 5.1]{toaldo_pota} is the equivalent formulation of the previous theorem using the operator $^C\phi(\partial_t)$ instead of $\phi(\partial_t)$, but the proof has a mistake. However, the statement of \cite[Theorem 5.1]{toaldo_pota} is salvageable without adding assumptions to it, but with the clarification that the equation holds for almost all $t>0$. For completeness, here we provide the details. 
		
		Recall that \cite[Theorem 5.1]{toaldo_pota} is stated under the assumptions that $P_t$ is contractive, i.e., $\|P_t\|\le 1$, and the L\'evy measure satisfies $\nu(0, +\infty)= +\infty$ and $s \mapsto \bar{\nu}(s)$ is absolutely continuous.  In this context, it is well known that the random variable $S(t)$ admits a density (see, e.g., \cite[Theorem 27.7]{sato}), which we will denote here by $x \mapsto p_t^S(x)$.
		Since
		\begin{align}
			P_su-u=\int_0^s P_rAu\,dr,\quad u\in \DD(A),
		\end{align}
		Fubini's theorem for the Bochner integral and the identity
		\begin{align}
			\mathds {P}^{(x,0)}\bigl(L(t)>r\bigr)
			=
			\mathds {P}^{(x,0)}\bigl(S(r)<t\bigr)
		\end{align}
		give
		\begin{align}
			\begin{aligned}
				q(t)-u=\mathcal P_tu-u
				&=
				\int_{[0,\infty)}(P_su-u)\,\mathds{P}^{(x,0)}(L(t) \in ds)\\
				&=
				\int_{[0,\infty)}
				\int_0^s P_rAu\,dr\,\mathds{P}^{(x,0)}(L(t) \in ds)\\
				&=
				\int_0^\infty
				P_rAu\,\mathds{P}^{(x,0)}(L(t) >r)\,dr\\
				&=
				\int_0^\infty
				P_rAu\,\mathds{P}^{(x,0)}(S(r) <t)\,dr.
			\end{aligned}
			\label{attempt}
		\end{align}
		The use of Fubini's theorem is justified since, for every $\lambda>0$ and $T>0$,
		\begin{align}
			\begin{aligned}
				\int_0^\infty
				\mathds{P}^{(x,0)}\bigl(S(r)<T\bigr)\,dr
				&\leq
				e^{\lambda T}
				\int_0^\infty
				\mathds {E}^{(x,0)}\!\left[e^{-\lambda \, S(r)}\right]dr\\
				&=
				e^{\lambda T}
				\int_0^\infty e^{-r \phi(\lambda)}\,dr\\
				&=
				\frac{e^{\lambda T}}{\phi(\lambda)}
				<\infty.
			\end{aligned}
		\end{align}
		Using the fact that $S(r)$ is an absolutely continuous random variable, we then have
		\begin{align}
			\begin{aligned}
				q(t)-u
				&=
				\int_0^\infty
				P_rAu\left(\int_0^t p_r^S(\tau)\,d\tau\right)dr\\
				&=
				\int_0^t
				\left(\int_0^\infty P_rAu\,p_r^S(\tau)\,dr\right)d\tau.
			\end{aligned}
		\end{align}
		Indeed, for every $T>0$,
		\begin{align}
			\begin{aligned}
				\int_0^T
				\int_0^\infty
				\|P_rAu\|p_r^S(\tau)\,dr\,d\tau
				&\leq
				\|Au\|
				\int_0^\infty \mathds{P}^{(x,0)} (S(r) < T)\,dr\\
				&\leq
				\|Au\|
				\frac{e^{\lambda T}}{\phi(\lambda)}
				<\infty.
			\end{aligned}
		\end{align}
		Thus $q(\cdot)\in AC_{\mathrm{loc}}
		\bigl([0,\infty);\mathfrak B\bigr)$
		and
		\begin{equation}
			q'(t)
			=
			\int_0^\infty P_rAu\,p_r^S(t)\,dr
			\qquad\text{for a.e. }t>0,
		\end{equation}
		where the derivative is a strong derivative in $\mathfrak B$.
		Recall that in the proof of the previous Theorem \ref{theorem_convolution} we have proved that
		\begin{align}\label{457}
			\frac{d}{dt}
			\left[
			b\bigl(q(t)-u\bigr)
			+
			\int_0^t
			\bar{\nu}(t-s)\bigl(q(s)-u\bigr)\,ds
			\right]
			=
			Aq(t)
		\end{align}
		strongly for every $t>0$.
		Therefore, it remains to identify the  derivative on the left-hand side with the operator in the
		statement of \cite[Theorem 5.1]{toaldo_pota}, i.e., $\,^C\phi(\partial_t)q(t)$. Since $q\in AC_{\mathrm{loc}}$, we have
		\begin{align}
			q(t)-u=\int_0^t q'(r)\,dr.
		\end{align}
		By Young's inequality \cite[Proposition 1.3.2]{arendt2001cauchy}
		\begin{equation}
			t\longmapsto
			(\bar{\nu}*q')(t)
			:=
			\int_0^t q'(t-s)\bar{\nu}(s)\,ds
		\end{equation}
		belongs to
		$L^1_{\mathrm{loc}}\bigl((0,\infty);\mathfrak B\bigr)$;
		indeed, for every $T>0$,
		\begin{align}
			\int_0^T\|(\bar{\nu}*q')(t)\|\,dt \leq \left( \int_0^T \bar{\nu}(s) ds \right) \left(
			\int_0^T \left\| q'(s) \right\| ds \right) . 
		\end{align}
		Fubini's theorem for the Bochner integral now gives 
		\begin{align}
			b\bigl(q(t)-u\bigr) +  \int_0^t
			\bar{\nu}(t-s)\bigl(q(s)-u\bigr)\,ds = \int_0^t
			\left[ bq'(r)+(\bar{\nu}*q')(r)
			\right]dr.
			\label{dinnertime}
		\end{align}
		Consequently, the left-hand side of \eqref{457} is equal to
		\begin{align}
			bq'(t) + \int_0^t q'(t-s)\bar{\nu}(s)\,ds
			\qquad\text{for a.e. }t>0.
			\label{late_wifecalling}
		\end{align}
		From the proof of the previous Theorem \ref{theorem_convolution}, i.e., \eqref{457}, and from \eqref{late_wifecalling}, we conclude that
		\begin{equation}
			{}^C\phi(\partial_t)q(t)
			=
			bq'(t)
			+
			\int_0^t q'(t-s)\bar{\nu}(s)\,ds
			=
			Aq(t)
		\end{equation}
		strongly in $\mathfrak B$ for a.e. $t>0$.

	\end{remark}

	\subsection{Coupled CTRWs and general non-local equations}
	\label{sec_coupled_harmonic}
	When the CTRW is not uncoupled, the processes $A$ and $S$ are not independent, and the governing evolution equations of $X^+(t)=A(L(t))$, $t\ge0$, and $X(t)=A(L(t)-)^+$, $t\ge0$, take more complicated forms. Heuristically, because of the dependence between the jumps and the intervals of constancy, the governing evolution equation represented by a (linear) operator should be non-local in time and space and coupled, i.e., it cannot be expressed as the sum of two operators acting separately on time and space. In other words, one should expect that $q^+(x,t) = \mathds{E}^{(x,0)}u(X^+(t))$ and $q(x,t) = \mathds{E}^{(x,0)} u(X(t))$ solve
	\begin{align}\label{evolutions}
		&\begin{cases}
			\begin{array}{rcll}
				\mathfrak{A}^+q^+(x,t) &=& 0, \enskip &t>0,\\
				q^+(x,0) &=& u,\enskip&
			\end{array}
		\end{cases}\enskip\text{and}\enskip
		\begin{cases}
			\begin{array}{rcll}
				\mathfrak{A}q(x,t) &=& 0, \enskip &t>0,\\
				q(x,0) &=& u,\enskip&
			\end{array}
		\end{cases}
	\end{align}
	where $\mathfrak{A}$ and $\mathfrak{A}^+$ are non-local in the variables $(x,t)$ and coupled. We also recall that $X^+$ and $X$ are, in general, not a.s. identical, as seen in Example \ref{example_limit_coupled}, and also in Figure \ref{fig:subdiffusion}.
	
	Relevant contributions, in this context, are \cite{kolokoltsov2009generalized, hernandez2017, baeumer2017fokker, savtoa, ascione2024, biocic2025}.
	The most recent article is \cite{biocic2025} and it provides a `universal' method to identify the form of the operators $\mathfrak{A}$ and $\mathfrak{A}^+$; here we explain the approach.
	
	In this subsection, we take a general (possibly sub-Markovian) Feller process $(A,S)=\big((A(t),S(t)),\, t\ge0\big)$ on $E\times \RR$, for some locally compact and separable space $E$, where the second coordinate $S$ is strictly increasing and unbounded, and we study the processes $X^+$ and $X$ in this generalized setting. In other words, we allow arbitrary dependence between the coordinates $A$ and $S$, and $(A,S)$ is understood as a generalized version of the scaling limit process that appeared in Theorem \ref{maintheoremctrwlim}.
	
	Note that the inverse $L(t)=\inf\{s\ge0:S(s)>t\}$ is the exit time of the process $(A,S)$ from the closed set $E \times (-\infty, t]$. Hence, motivated by the classical theory, one could expect that the process $(A,S)$ stopped at $L(t)$ has a harmonic property: if $\mathcal{A}$ is the generator of $(A,S)$, then the equation $\mathcal{A}q(x,v)=0$, in $E \times (-\infty, t]$, subject to the initial condition $q(x,v)=u(x,v)$ in $E \times (t, +\infty)$, is solved by $q(x,v) =\mathds{E}^{(x,v)}u(A({L(t)}), S({L(t)}))$. However, even if this property is satisfied, this is not an evolution as desired in \eqref{evolutions}, since the initial condition is given in the variables $(x,v)$. We instead want a governing equation acting on the space-time variables $(x,t)\in E\times (0,\infty)$, with a suitable initial condition on $E\times \{0\}$. Nevertheless, in the next theorem, we use a modification of this classical argument to construct two new processes that will give us our desired operators $\mathfrak{A}^+$ and $\mathfrak{A}$. The theorem comes from \cite[Theorem 2.1]{biocic2025}.

	\begin{theorem}\label{theorem_harmoniv}
		Let $(A(u), S(u))$, $u \geq 0$, be a general Feller process on $E \times \RR$ with the second coordinate strictly increasing, under the corresponding family of probability measures $\mathds{P}^{(x,v)}$, $(x,v)\in E\times \RR$. Denote its transition probabilities by $p_{u}(x,v;dy,dw)\coloneqq \PP^{(x,v)}(A(u)\in dy, S(u) \in dw) $.
		
		Then, for all $t\ge0$, there exist two Markov processes $\,^tZ^+=(^tA^+(u), \,^tS^+(u))$, $u\geq 0$, and $\,^tZ^-=(^tA^-(u), \,^tS^-(u))$, $u \geq 0$, on the canonical probability space with measures $\,^t\mathds{P}_\pm^{(x,v)}$, $(x,v)\in E\times [0,+\infty)$, with the transition probabilities 
		\begin{align*}
			\, ^t{p}_{u}^+(x,v;dy,dw) &\coloneqq \mathds{1}_{[0<w\leq v]}  p_{u}(x,t-v;dy,t-dw)+
			\delta_0(dw)\,^t\upvarpi_u^+(x,t-v,dy), \\
			\,^tp_{u}^-(x,v;dy,dw) &\coloneqq \mathds{1}_{[0<w\leq v]}  p_{u}(x,t-v;dy,t-dw)+\delta_0(dw)\,^t\upvarpi_u^-(x,t-v,dy),
		\end{align*}
		respectively, where $\,^t\upvarpi_u^+(x,t-v,dy)\coloneqq \mathds{P}^{(x,t-v)}\left( A(L(t))\in dy, S(u)\in[t,+\infty)\right)$ and $\,^t\upvarpi_u^-(x,t-v,dy)\coloneqq \mathds{P}^{(x,t-v)}\left( A(L(t)-)\in dy, 
		S(u)\in[t,+\infty)\right)$, such that 
		\begin{align}
			& \mathds{P}^{(x,0)} \left( A(L(t)) \in dy \right) =  \,^t\mathds{P}_+^{(x,t)} \left( \,^tA^+(\tau_0^+) \in dy \right), \label{317}\\
			& \mathds{P}^{(x,0)} \left( A(L(t)-) \in dy \right) =  \,^t\mathds{P}_-^{(x,t)} \left( \,^tA^-(\tau_0^-) \in dy \right). \label{318}
		\end{align}
		where $\tau^{\pm}_0 \coloneqq \inf \{ u \geq 0 : \,^tZ^{\pm}(u) \in E \times \{ 0 \} \}$.
	\end{theorem}
	\begin{proof}
		It is a straightforward exercise to check that, for each $t\ge 0$, the kernels $\,^tp^{\pm}$ satisfy the Chapman-Kolmogorov property, thus forming (sub-)Markov transition functions. Hence, for each $t\ge0$, we can construct the processes $^tZ^+=(^tA(u)^+, \,^tS(u)^+)$,  $u\ge0$, and $^tZ^-=(^tA(u)^-, \,^tS(u)^-)$,  $u\ge0$, as canonical processes with the state space $E \times [0, +\infty)$, with measures $\,^t\mathds{P}_\pm^{(x,v)}$, $(x,v)\in E\times [0,+\infty)$, whose transition kernels are $^tp^+$ and $^tp^-$, respectively.
		
		Note that when $v=t$, that is, under $\,^t\mathds{P}_\pm^{(x,t)}$, we have
		\begin{align}
			^t p^+_u(x,t;dy,dw) & =  \mathds{1}_{[w>0]} \mathds{P}^{(x,0)} \left( A(u) \in dy, t-S(u) \in dw  \right) \notag \\ &\hspace{4em}+\mathds{P}^{(x,0)} \left( A(L(t)) \in dy, L(t) \leq u \right)\delta_0(dw), \\
			^tp^-_u(x,t;dy, dw)&= \mathds{1}_{[w>0]} \mathds{P}^{(x,0)} \left( A(u) \in dy, t-S(u) \in dw \right) \notag \\ &\hspace{4em}+ \mathds{P}^{(x,0)} \left( A(L(t)-)  \in dy, L(t) \leq u \right) \delta_0(dw).
		\end{align}
		In other words, when the processes $^tZ^+$ and $^tZ^-$ start from $(x,t)$, and when $(A,S)$ starts from $(x,0)$, then one has (in the sense of finite-dimensional distributions) that
		\begin{align}\label{eq:heuristics}
			^tZ^+(u)\, = \, \begin{cases}
				\big(A(u), t-S(u)\big), \enskip & u < L(t), \\ \big(A(L(t)),0\big), & u \geq  L(t),
			\end{cases} \quad ^tZ^-_u = \begin{cases}
				\big(A(u), t-S(u)\big), \enskip & u < L(t), \\ \big(A(L(t)-),0\big), & u \geq  L(t).
			\end{cases} 
		\end{align}
		
		To finish the proof, first note that $L(t)=\inf\{s\ge0:S(s)>t\}=\inf\{s\ge0:S(s)\ge t\}\eqqcolon \wt L(t)$ a.s., since $S$ is assumed to be strictly increasing. Thus, by \eqref{eq:heuristics}, $\tau^{\pm}_0 \coloneqq \inf \{ u \geq 0 : \,^tZ^{\pm}(u) \in E \times \{ 0 \} \}$ under $\,^t\mathds{P}_\pm^{(x,t)}$ corresponds to  $L(t)$ under $\PP^{(x,0)}$, so we get \eqref{317} and \eqref{318}.    
	\end{proof}

	The construction presented in Theorem \ref{theorem_harmoniv} suggest us, for each $t\ge0$, two operators $\,^t\mathfrak{A}^+$ and $\,^t\mathfrak{A}$ that are generators (in the sense discussed below) of the processes $\,^t Z^+$ and $\,^t Z^-$, respectively. The relations \eqref{317} and \eqref{318} now imply that the evolution equations for $X^+(t)=A(L(t))$, $t\ge0$, and $X(t)=A(L(t)-)^+$, $t\ge0$, could be obtained in the following way:
	The family of harmonic problems for $\,^t Z^\pm$, $t\ge0$, upon exiting from the open set $E \times (0, +\infty)$, is given by
	\begin{align}
		\begin{cases} 
			\begin{array}{rcll}
				^t\mathfrak{A}^{\pm} q_t(x,v)  &=&  0, \qquad &(x,v) \in E \times (0, +\infty) \\
				q_t(x,v) &=& u(x), & (x,v) \in E \times \{ 0 \}.
			\end{array}
		\end{cases}
		\label{1407}
	\end{align}
	The problems, heuristically, should be solved by
	\begin{align}
		q_t(x,v) \, = \, \mathds{E}_{\pm}^{(x,v)} u(\,^tA^{\pm}_{\tau_{0}^{\pm}}).
		\label{1408}
	\end{align}
	In view of \eqref{317} and \eqref{318}, by fixing $t=v$, we should obtain
	\begin{align}
		&     q_t(x,t) \, = \, \mathds{E}_{+}^{(x,t)} u(\, ^tA^{+}_{\tau_{0}^{+}}) \, = \, \mathds{E}^{(x,0)} u(A(L(t))) \label{320}\\ 
		&   q_t(x,t) \, = \, \mathds{E}_{-}^{(x,t)} u(\, ^tA^{-}_{\tau_{0}^{-}}) \, = \, \mathds{E}^{(x,0)} u(A({L(t)-)}).
	\end{align}
	
	Note that we will study the harmonic problem for the process $\,^tZ^-$, related to $A(L(t)-)$. However, this is equivalent to studying $A(L(t)-)^+$ in the setting of the following proposition. 
	\begin{proposition}\label{p:1333}
		Suppose that $(A,S)$ is generated by the operator as in \eqref{generatorlimit} with core $C_c^\infty(\R^{d+1})$.
		Then, for all $t>0$, it holds that $(A(L(t)-)^+=A(L(t)-)$, $\PP^{(x,s)}$-a.s. for all $x\in \R^d$ and $s\le t$, if and only if for all $t>0$ and for almost every $h>0$ it holds that  $K(A(h),S(h);\R^d\times \{t-S(h)\})=0$, $\PP^{(x,s)}$-a.s. for all $x\in \R^d$ and $s\le t$.
	\end{proposition}
	\begin{proof}
		See \cite{biocic2025}.
	\end{proof}
	\begin{remark}
		In general, it is possible that it does not hold $(A(L(t)-)^+=A(L(t)-)$ $\PP^{(x,s)}$-a.s. for all $x\in \R^d$ and $s\le t$, see \cite{biocic2025}. If, however, the jump kernel $K$ has no atoms, then clearly the necessary and sufficient condition of Proposition \ref{p:1333} is satisfied.
	\end{remark}
	
	We emphasize that in the construction of harmonic problems above, we were deliberately not rigorous in saying what kind of generators $\,^t\mathfrak{A}^+$ and $\,^t\mathfrak{A}$ are. The goal here is to get a feeling for them in some vague sense, or to heuristically describe them, and after we get a good candidate for the evolution operator, we can proceed by rigorously defining it (e.g. pointwise, weakly, distributionally, etc.) and try to obtain the harmonic property. Checking that \eqref{1408} indeed solves \eqref{1407} is a problem in itself, and we note that, to the best of our knowledge, this problem can be attacked only case by case, as there is no universal method for solving harmonic problems for general Markov processes. The main issue is usually the regularity of \eqref{1408}. Here we mention that harmonic problems were studied for unimodal L\'evy processes in \cite{bogdan_extension}, in the pointwise, distributional and Dirichlet forms settings. However, the process $(A,S)$ is neither symmetric, nor unimodal, nor L\'evy. Moreover, it is known that in some quite reasonable cases, see \cite[p. 150]{BS}, the harmonic function (i.e. the candidate for a solution) may lack sufficient regularity to apply the generator on it pointwise. For the Dirichlet forms point of view on harmonic problems, see  \cite{Chen09} and \cite{MZZ10}.

	When the process $(A, S)$ is Markov additive, we can give a more explicit version of the previous theorem, and derive the evolution equations in a more approachable and simplified form. We recall that the Markov additive case is when the future evolution of $(A(u),S(u))$ depends only on the current position of $A(u)$. Furthermore, the first coordinate of $(A(u),S(u))$, ${u\ge 0}$, is a Markov process by itself, while the second coordinate is homogeneous in space. For details, see, e.g., \cite{cinlarma}. Analytically, if we have the generator representation \eqref{generatorlimit}, then the coefficients $b_i$, $\gamma$, $a_{ij}$, and the jumping kernel $K$ do not depend on $t$. This allows us to simplify the previous considerations in which we can abandon the parameter $t$ in the transitions $\,^tp^{\pm}$, while at the same time we still cover a wide range of interesting examples.  The simplification of Theorem \ref{theorem_harmoniv} is given in the next proposition.

	\begin{proposition}\label{p:harmonic}
		Let $(A(u), S(u))$, $u \geq 0$, be a Markov additive process as above, with the transition probabilities $p_{u}(x,v;dy,dw)$. Then, there exist  Markov processes $Z^+=(A^+(u),S^+(u))$, $u \geq 0$, and $Z^-=(A^-(u),S^-(u))$, $u \geq 0$, on $E\times [0,+\infty)$ on the canonical probability space with the measures $\mathds{P}_{\pm}^{(x,t)}$, $(x,t)\in E\times [0,+\infty)$, and the transition probabilities 
		\begin{align}
			{p}_{u}^+(x,t;dy,dw)&=\mathds{P}^{(x,0)}(A(u)\in dy,(t-S(u))\vee 0\in dw),\label{transitionunder-NO4} \\
			\begin{split}\label{transitionunder-NO5}
				{p}_{u}^-(x,t;dy,dw)&=\mathds{1}_{[w>0]} \mathds{P}^{(x,0)} \left( A(u) \in dy, t-S(u) \in dw \right)\\
				&\hspace{4em}+ \mathds{P}^{(x,0)} \left( A({L(t)-}) \in dy, L(t) \leq u \right) \delta_0(dw),
			\end{split}
		\end{align}
		for which it holds
		\begin{align}
			& \mathds{P}^{(x,0)} \left( A({L(t)}) \in dy \right) =  \mathds{P}_+^{(x,t)} \left(A^+({\,\tau_0^+}) \in dy \right), \label{evol-over}\\
			& \mathds{P}^{(x,0)} \left( A({L(t)-}) \in dy \right) =  \mathds{P}_-^{(x,t)} \left(A^-({\,\tau_0^-}) \in dy \right),\label{evol-under}
		\end{align}
		where $\tau^{\pm}_0 \coloneqq \inf \{ u \geq 0 : Z^{\pm}(u) \in E \times \{ 0 \} \}$.
	\end{proposition}
	\begin{proof}
		Recall that since $(A,S)$ is Markov additive, the second coordinate $S$ is homogeneous in space so we have $p_u(x,t-v;dy,t-dw)=p_u(x,-v;dy,-dw)=p_u(x,0;dy,v-dw)$. Therefore, it is not hard to see that the kernels
		\begin{align}
			\begin{split}\label{transitionover-NO2}
				{p}_{u}^+(x,t;dy,dw) &\coloneqq \mathds{1}_{[0<w\leq t]} \, p_{u}(x,0;dy,t-dw)\\
				&\hspace{4em}+ \delta_0(dw)\, p_{u} (x,0; dy, [t, +\infty)).
			\end{split}\\
			\begin{split}\label{transitionover-NO3}
				{p}_{u}^-(x,t;dy,dw) &\coloneqq \mathds{1}_{[0<w\leq t]} \, p_{u}(x,0;dy,t-dw)\\
				&\hspace{4em}+ \delta_0(dw)\,\upvarpi_u^-(x,t,dy),
			\end{split}
		\end{align}
		where $\upvarpi_u^-(x,t,dy)\coloneqq \mathds{P}^{(x,0)}\left( A(L(t)-)\in dy, 
		S(u)\in[t,+\infty))\right)$,
		satisfy the Chapman-Kolmogorov property, thus forming (sub-)Markov transition kernels. By the construction, we get \eqref{transitionunder-NO4} and \eqref{transitionunder-NO5}, and the  rest of the proof is the same as the proof of Theorem \ref{theorem_harmoniv}.
	\end{proof}
	
	Recall that, since $(A,S)$ is Markov additive, the first coordinate $A$ is a Markov process by itself, so in the evolution of $Z^+$ in Proposition \ref{p:harmonic} we allowed the first coordinate to keep moving after stopping the second one, unlike in Theorem \ref{theorem_harmoniv} where we have stopped both coordinates.
	
	The derivation of the evolution equations for $X^+$ and $X$ in the Markov additive setting now goes like this. The processes $Z^+$ and $Z^-$ from Proposition \ref{p:harmonic} induce two operators $\mathfrak{A}^+$ and $\mathfrak{A}^-$ as their generators (again, taken in some appropriate sense), so the harmonic problems for the processes $Z^\pm$ upon exiting $E\times (0,\infty)$
	\begin{align}\label{evolutions-NO2}
		\mathfrak{A}^\pm q^\pm(x,t) &= 0, \quad x\in E,\,t>0,\\
		q^\pm(x,0) &= u(x),\quad x\in E.
	\end{align}
	should have solutions $q^\pm(x,t)=\EE^{(x,t)}_\pm\big(u(A^\pm(\tau^\pm_0))\big)$. By \eqref{evol-over} and \eqref{evol-under} these are exactly the evolutions of $X^+(t)$, $t\ge0$, and $X(t)$, $t\ge0$.

	\subsubsection*{Applying the `universal' method}
	Now we apply the `universal' method explained above to write the governing evolution equation for $X$ and $X^+$ in some special cases when the well-posedness has been made clear.
	
	In particular, in the undershooting and overshooting case explained in Example \ref{example_limit_coupled}, we can derive the operators explicitly and rigorously and establish well-posedness of the harmonic problem. To this end, let $M(u)$, $u\ge0$, be a Feller process in $E$ associated with the semigroup $P_t$, $t \geq 0$, and the generator $(G,\DD(G))$. Let also the process  $S(u)$, $u\ge 0$, be a strictly increasing subordinator, independent of $M$, characterized by its Laplace exponent $\phi$ as in \eqref{bernstein-eq0} (with parameters $a$, $b$ and $\nu$). We consider here the Feller process $(A(u), S(u))$, $u \geq 0$, where $A(u) \coloneqq M(S(u)-S(0))$, so that $X^+(t) = M(S(L(t)))$ and $X(t) = M(S(L(t)-))$. Note that $X^+$ is the process $M$ time-changed by the overshootings $D(t)$, $t\ge0$, of the subordinator $S$, while $X$ is $M$ time-changed by the undershootings $H(t-)$, $t\ge0$, of $S$.

	The expectations $\mathds{E}^{(x,v)}$, $(x,v)\in E\times \R$,  considered here correspond to the canonical probability measures for which $\mathds{P}^{(x,v)}\big(M(0)=x,S(0)=v\big)=1$. Let us denote by $S^0$ the process $S^0(u), u \geq 0$, where $S^0(u)=S(u)-S(0)$. Note that under $\mathds{P}^{(x,0)}$ we have $S^0(u)=S(u)$, a.s., for all $u\geq0$.
	
	Let us now derive rigorously the operator $\mathfrak{A}^+$ in the case $X^+(t)=A(L(t))=M(S(L(t))=M(D(t))$, $t\ge0$.

	\begin{lemma}\label{lem:operator}
		The process $(Z_u^+)_{u\ge0}$ is a Feller process on $E\times[0,\infty)$ with the generator $\mathfrak{A}^+$ such that
		\begin{align*}%\label{gen-final}
			\mathfrak{A}^+ h(x,t)&=-a\,h(x,t)+b(G-\partial_t^{(0,\infty)})h(x,t)\\
			&\hspace{4em}+\int_0^\infty \left(P_sh\big(\cdot,(t-s)\vee0\big)(x)-h(x,t)\right)\nu(ds),
		\end{align*}
		for all $h\in \mathrm{span} \{(x,t)\mapsto f(x)g(t): f\in \DD(G),g\in C_{00}^1([0,\infty))\}$.\footnote{Here $C_{00}^1([0,\infty))\coloneqq \{f\in C_0([0,\infty)):f\in C^1(0,\infty),\, \lim\limits_{t\searrow0}f'(t)=\lim\limits_{t\to\infty}f'(t)=0\}=\{f\in C_0([0,\infty)):f\in C^1(0,\infty),\, \lim\limits_{t\searrow0}(f(t)-f(0))/t=\lim\limits_{t\to\infty}f'(t)=0\}$, and $\partial_t^{(0,\infty)}f(t)=f'(t)\1_{t>0}$, $f\in C_{00}^1([0,\infty))$.}
	\end{lemma}
	\begin{proof}
		Note that the process $Z^+(u)$, ${u\ge0}$,  can be written as $Z^+(u)=\big(M(S^0(u)),(t-S^0(u))\vee 0\big)$ under $\PP_+^{(x,t)}$. We recognize that this form is actually the subordination with $S^0$ of the Feller process $(M(u),\Gamma^0(u))$, $u\ge0$, where $\Gamma^0$ is the negative translation (pure drift) stopped at zero. In other words, first take the process $(M(u),\Gamma^0(u))$, ${u\ge0}$, in $E\times[0,\infty)$ given by 
		\begin{align}\label{1747}
			\PP^{(x,t)}(M(u)\in dy,\Gamma^0(u)\in dw)=p_u(x,dy)\delta_{(t-u)\vee 0}(dw).
		\end{align}
		It is easy to show that $(M(u),\Gamma^0(u))$, $u\ge0$, is a Feller process on $E\times[0,\infty)$ and that for its generator, denoted by $(G-\partial_t^{(0,\infty)})$, it holds that
		\begin{align}
			(G-\partial_t^{(0,\infty)})h(x,t)=Gh(x,t)-\partial_t^{(0,\infty)}h(x,t),
		\end{align}
		for $h\in \mathrm{span} \{(x,t)\mapsto f(x)g(t): f\in \DD(G),g\in C_{00}^1([0,\infty))\}\subset\DD(G-\partial_t^{(0,\infty)})$.
		By subordinating the process $(M(u),\Gamma^0(u))$, $u\ge0$, by $S^0$, we get the process with transitions
		\begin{align}
			\PP^{(x,t)}(M({S^{0}(u)})\in dy,\Gamma^0({S^0(u)})\in dw)&=\PP^{(x,0)}(M({S^{0}(u)})\in dy,(t-S(u))\vee 0\in dw)\\
			&=\PP_+^{(x,t)}(Z^+(u)\in(dy,dw)).
		\end{align}
		Thus, by \cite[Proposition 15.1]{bernstein} we get that $(Z(u)^+)$, ${u\ge0}$, is a Feller process and by Phillips' theorem \cite[Theorem 15.6]{bernstein}, we get that for the generator $\mathfrak{A}^+$ of $Z^+$ it holds that
		\begin{align*}
			\mathfrak{A}^+ h(x,t)&=-a\,h(x,t)+b(G-\partial_t^{(0,\infty)})h(x,t)\\&\hspace{4em}+\int_0^\infty \left(\ex^{(x,0)}\left[h(M_s,(t-s)\vee0)\right]-h(x,t)\right)\nu(ds)\\
			&=-a\,h(x,t)+b(G-\partial_t^{(0,\infty)})h(x,t)\\&\hspace{4em}+\int_0^\infty \left(P_sh\big(\cdot,(t-s)\vee0\big)(x)-h(x,t)\right)\nu(ds),
		\end{align*}
		for all $h\in \DD(G-\partial_t^{(0,\infty)})\subset \DD(\mathfrak{A}^+)$, and in particular for all $h\in \mathrm{span} \{(x,t)\mapsto f(x)g(t): f\in \DD(G),g\in C_{00}^1([0,\infty))\}$.
	\end{proof}
	
	Although we obtained the operator $\mathfrak{A}^+$ quite explicitly, it still remains to prove that $q^+(x,t)=\EE^{(x,0)}(u(X^+(t)))$ solves the harmonic problem. To be able to conduct the calculations, we need assumptions on the subordinator $S$, expressed in terms of its Laplace exponent. We will assume:
	
	\begin{assumption}{S}{}\label{assphi}
		The function $\phi$ is a special Bernstein function such that $a=b=0$, $\nu(0,\infty)=\infty$, and its potential density $u^\phi$ satisfies $\int_0^1|\log(t)|u^\phi(t)dt<\infty$.
	\end{assumption}
	
	That $\phi$ is a special Bernstein function means that $\lambda\mapsto \lambda/\phi(\lambda)$ is a Bernstein function, too. Such functions have the potential measure decomposition $U(dt)=c \delta_0(dt)+u^\phi(t)dt$, where $u^\phi$ is non-increasing and integrable in $(0,1)$. Also, under $a=b=0$, $\nu(0,\infty)=\infty$, it holds that $c=0$. For details, see \cite[Chapter 13]{bernstein}.

	Under \ref{assphi}, we finally obtain
	\begin{align}\label{gen-over}
		\mathfrak{A}^+ f(x,t)=\int_0^\infty \left(P_sf\big(\cdot,(t-s)\vee0\big)(x)-f(x,t)\right)\nu(ds).
	\end{align}
	
	The evolution of the undershooted Feller process, i.e. $X(t)=A(S(L(t)-))^+=M(H(t-))^+$ is driven by the operator $\mathfrak{A}$, i.e. by the generator of the process $Z^-$ with the transitions as in \eqref{transitionunder-NO5}. {Here the property $A(L(t)-)^+=M(H(t-))^+=M(H(t-))=A(L(t)-)$ a.s. can be confirmed by a direct calculation. Namely, $\PP^{(x,0)}(H(t-)<t)=1$ for all $t>0$, see \cite[Proposition 2 \& Theorem 4, in Chapter III]{bertoin1996}, hence the probability that $H(t-)$ jumps at some fixed time $t>0$ is zero.} We can obtain an explicit expression for  $\mathfrak{A}$ by the following reasoning. The first term in \eqref{transitionunder-NO5}, in this case, corresponds to the subordinate Feller process $M(S^0(u))$ in the first coordinate, and ${t-S^0(u)}$ in the second coordinate, but only if ${t-S^0(u)}$ did not go below level 0 (or in other words if ${S^0(u)}$ did not go over level $t$). This evolution is then just the subordinated evolution $\overline P_s f(x,t)\coloneqq P_sf(x,t-s) \mathds{1}_{[s < t]}$. On the other hand, the second term in \eqref{transitionunder-NO5} corresponds to the event when the process ${t-S^0(u)}$ has crossed the level 0, and in that case we have stopped the second coordinate in 0, while the first coordinate is again the subordinate Feller process $M(S^0(u))$ but stopped in its position before the jump of ${t-S^0(u)}$ below level 0. Since $M(S^0)$ and $S$ have simultaneous jumps, this position is $M(S^0(L(t)-))$. This induces the following representation for $\mathfrak{A}^-$ (under the assumption \ref{assphi}), which is a modification of \eqref{gen-over} that takes into account this stopping:
	\begin{align}
		\mathfrak{A}f(x,t) \, = \,  \int_0^{+\infty} \left( P_sf(x,t-s) \mathds{1}_{[s < t]} + \underbrace{f(x, 0) \mathds{1}_{[s \geq t]}}_{\text{stopping event}} -f(x,t)\right) \nu(ds),
		\label{genunder}
	\end{align}
	and by rearranging \eqref{genunder} we get
	\begin{align}
		\mathfrak{A}f(x,t) \, = \,  \int_0^{+\infty} \left( P_sf(x,t-s) \mathds{1}_{[s < t]}  -f(x,t)\right) \nu(ds) + f(x,0) \bar{\nu}(t).
		\label{phiddt-g}
	\end{align}
	With this at hand, we then expect that the evolution equation governing the process $X^+$ is 
	\begin{align}
		\begin{cases}\label{eq_over}
			\begin{array}{rcll}
				\mathfrak{A}^+q^+(x,t) &=& 0,  &(x,t)\in E\times (0,+\infty),\\
				q^+(x,0) &=& u(x), &x\in E,
			\end{array}
		\end{cases} 
	\end{align}
	and that the evolution equation governing the process $X$ is
	\begin{align}
		\begin{cases}\label{eq_under} 
			\begin{array}{rcll}
				\mathfrak{A}q(x,t) &=& 0,  &(x,t)\in E\times (0,+\infty),\\
				q(x,0) &=& u(x), &x\in E,
			\end{array}
		\end{cases} 
	\end{align}
	
	Here is the formal statement establishing well-posedness in the pointwise sense, and it is the combination of \cite[Theorem 5.2]{ascione2024} and \cite[Theorem 4.2]{biocic2025}. 
	\begin{theorem}
		Assume that the function $\phi$  satisfies \ref{assphi}, and take $u\in \DD(G)$. Then a pointwise solution to \eqref{eq_over}
		is given by $q^+(x,t)=\mathds{E}^{(x,0)}[u(M({D_t}))]$, $t\ge0$ and a pointwise solution to \eqref{eq_under} is $q(x,t) = \mathds{E}^{(x,0)}u(M(H(t-)))$. 
		Furthermore, these solutions are the unique in the class of jointly continuous functions in $E\times [0,+\infty)$ with property $\lim_{t\to +\infty} q(x,t)=0$, for all $x\in E$.
	\end{theorem}
	\begin{proof}[Idea of the proof]
		The proof is conducted by using the Laplace transform technique and is in spirit similar to the proof of Theorem \ref{theorem_convolution}. The calculations, however, are cumbersome and include regularity claims proved for $q^+$ and $q$ (e.g. it is proved that $\partial_t q(t,x)$ exists a.e.), as well as the proofs of the positive maximum principle for these operators. Also, under the imposed assumptions, it is useful that the distribution of the undershooting $H(t-)$ and overshooting $D(t)$ is known and given in \cite{bertoin1996}.
		
		Nevertheless, full proofs with all details can be found in \cite{ascione2024} and \cite{biocic2025}.
	\end{proof}

	\begin{remark}[Courrège-type form of $\mathfrak{A}^+$]
		If the Courrège-type form of the generator of the Markov additive process $(A,S)$ is known (denoted by $\mathcal{A}$) and e.g. given by \eqref{generatorlimit}, then it is possible to heuristically deduce a Courrège-type form of the operator $\mathfrak{A}^+$. Indeed, since $\ex^{(x,t)}_+ [f(Z^+(u))]=\ex^{(x,0)}\left[ f\big(A(u),(t-S(u))\vee 0\big)\right]$, we get
		\begin{align}
			\begin{split}\label{gen:1026}
				\mathfrak{A}^+f(x,t)&=\lim_{u\searrow0}\frac{\ex^{(x,t)}_+f({Z^+(u)})-f(x,t)}{u}\\
				&=\lim_{u\searrow0}\frac{\ex^{(x,0)}[ f(A(u),(t-S(u))\vee 0)]-f(x,t)}{u}\eqqcolon\mathcal{A}F_t(x,0),
			\end{split}
		\end{align}
		where $F_t(x,s)=f(x,(t-s)\vee 0)$. Here, we emphasize that $F_t$ is certainly not in the $C_0(\R^{d+1})$-domain of $\mathcal{A}$ since it does not vanish at the infinity. However, we can look at \eqref{gen:1026} heuristically, or e.g. by considering the pointwise extension. In other words, if $\mathcal{A}$ has the form \eqref{generatorlimit}, by \eqref{gen:1026} the generator $\mathfrak{A}^+$ has the form
		\begin{align*}
			&\mathfrak{A}^+f(x,t)=\sum_{i=1}^d b_i(x) \partial_{x_i} f(x,t) - \gamma(x) \partial_t f(x,t) + \frac{1}{2} \sum_{1 \leq i, j \leq d} a_{ij}(x) \partial_{x_ix_j}^2f(x,t)  \\
			&+ \int \left( f(x+y, (t-w)\vee 0) - f(x,t)-\sum_{i=1}^d y_i \mathds{1}_{[(y,w) \in [-1,1]^{d+1}]} \partial_{x_i} f(x,t)\right) K(x;dy,dw).
		\end{align*}
	\end{remark}
	
	\subsection{Variable-order fractional kinetic equations and anomalous aggregation}
	One might want to consider the case where the waiting time of a CTRW depends on the walker's current position. This case is useful in applications because it models an anomalous aggregation of particles, as explained in the following lines. Take an uncoupled CTRW with waiting times having power law behavior with (current position dependent) exponent $\alpha(x)$, $x \in \mathbb{R}^d$, with $\alpha(x) \in (0,1)$. Then, under suitable assumptions, this process converges (in probability) to the region where $\alpha(x)$ is minimum as shown in \cite{fedotov2012subdiffusive}, see also the statements of the theorems below for a better understanding of this convergence. A heuristic explanation is that the tail is heavier when $\alpha(x)$ is smaller, so the process tends to be trapped in that region, more than in the other. This phenomenon also has experimental evidence and is connected to variable-order fractional diffusion equations, see \cite{fedotov2019asymptotic, fedotov2021variable}. A similar behavior can be observed for scaling limits of these space-dependent CTRWs, as shown in \cite{savtoa}, which we will discuss later.
	
	Scaling limits of CTRWs governed by fractional variable-order equations are considered, e.g., in \cite{Kolokoltsov2023variable, STRAKA2018451}. In such cases, the limit process $(A,S)$ (from Theorem \ref{maintheoremctrwlim}) is Markov additive, with generator
	\begin{align}
		\begin{split}
			\mathcal{A}f(x,t) \, = \, &Gf(x,t)\\
			&+b(x)  \partial_t f(x,t) + \int_0^{+\infty} \left( f(x,t+w) -f(x,t) \right) \nu(x,dw).
		\end{split} \label{59}
	\end{align}
	Here, the operator $(G, \mathcal{D}(G))$ acts only on the $x$-variable. Moreover, when $G$ has a Courr\`ege-type form, and under appropriate technical assumptions on the coefficients, one has that \eqref{59} generates a Markov additive process, see \cite[Section 4]{Meerschaert2014} or \cite[Section 2]{savtoa}. Thus, in order to obtain a governing equation for $X^+$ and $X$ using the method of Section \ref{sec_coupled_harmonic}, we can resort to the simplified transition probabilities given in Proposition \ref{p:harmonic} to determine the Markov processes $Z^\pm$ and the harmonic problems \eqref{evolutions-NO2}. In particular, if we assume that $(A,S)$ is generated by \eqref{59}, with $G$ having a Courr\`ege-type form as in \cite[Section 3]{savtoa}, we have that the processes $A$ and $S$ do not jump simultaneously a.s. (by \cite[Lemma 3.2]{savtoa}), and thus $X^+(t)=A(L(t))=A(L(t)-)^+=X(t)$ a.s. by Lemma \ref{lemma_ctrw=octrw}. Hence, we can look only at the transition probabilities \eqref{transitionover-NO2}:
	\begin{align}
		p^+_u(s,t;dy,dw) \, = \, \mathds{P}^{(x,0)} \left( A(u) \in dy, (t-S(u)) \vee 0 \in dw \right).
	\end{align}
	Therefore, the operator $\mathfrak{A}^+$ can be determined starting from \eqref{59} just by observing that the second coordinate is reversed and stopped when it crosses zero, and we get
	\begin{align}
		\begin{split}
			\mathfrak{A}^+f(x,t) \, = \, &Gf(x,t) -b(x) \partial_tf(x,t) \\
			&  -\int_0^{+\infty} (f(x,t-w) \mathds{1}_{[w<t]} +f(x,0) \mathds{1}_{[w \geq t]} -f(x,t) ) \nu(x, dw).
		\end{split}
		\label{1908}
	\end{align}
	By setting $b(x)=0$ and $\nu(x,dw) =dw \alpha(x) w^{-\alpha(x)-1}/\Gamma(1-\alpha(x))$ in \eqref{1908}, one obtains the variable-order equation studied in \cite{Kolokoltsov2023variable}. Instead, if we rearrange the second term of \eqref{1908} as follows
	\begin{align}
		&-\int_0^{+\infty} (f(x,t-w) \mathds{1}_{[w<t]} +f(x,0) \mathds{1}_{[w \geq t]} -f(x,t) ) \nu(x, dw) \notag \\
		= \, & -\frac{\partial}{\partial t} \int_0^t f(x,w) \nu(x,(t-w, +\infty)) \, dw \, + \,  f(x,0) \, \nu(x,(t, +\infty)),
	\end{align}
	we obtain the variable-order fractional-type equation considered, for example, in \cite{savtoa}.
	
	The variable-order fractional diffusion equation is obtained by setting $G=\Delta$:
	\begin{align}
		\partial_t^{\alpha(x)} q(x,t) \, = \, \Delta q(x,t) \qquad q(x,0) = u(x).
	\end{align}
	The analysis of this equation is quite delicate (see, e.g., \cite{KIAN20181146, kian2018time}). The underlying stochastic process is a time-changed Brownian motion $B(L(t))$, $t\ge0$, where $L(t)$ is the inverse of the additive component $S(t)$, $t\ge0$, of the Markov additive process $(B(t), S(t))$, $t\ge0$, such that
	\begin{align}
		\mathds{E}^{(x,0)} e^{-\lambda S(t)} \, = \, \mathds{E}^{(x,0)}e^{-\int_0^t \lambda^{\alpha(B(s))}ds}.
	\end{align}
	We note that this holds under technical assumptions on the function $\alpha(x)$ (these technicalities were developed in \cite[Section 3]{savtoa}).
	
	It is very interesting to see what happens to the anomalous aggregation phenomena mentioned above, previously considered for space-dependent CTRWs, in the above described continuous process. The theory is developed in \cite[Section 4]{savtoa} and we state the main results omitting the proofs. In the following, we consider the one-dimensional case, i.e., $B$ is a one-dimensional standard Brownian motion, and we denote by $\mathpzc{l}(\cdot)$ the Lebesgue measure on $\RR$.
	\begin{theorem}[{\cite[Theorem 4.12]{savtoa}}]\label{thm:occupation}
		Let  $\alpha: \mathbb{R}\mapsto (0,1)$, with $\alpha^*=\min_{x\in \mathbb{R}}\alpha(x)>0, \max_{x\in\mathbb{R}}\alpha(x)<1$ and \[1>\lim_{x \to +\infty}\alpha(x)=\alpha_I>\alpha^*,\,\,1>\lim_{x\to -\infty}\alpha(x)=\alpha_J>\alpha^*.\] Also, let there exist $\beta_0>0$ small enough such that for all $\beta_0\geq \beta$, the set $A_\beta=\{x\in\mathbb{R}:\alpha(x)<\alpha^*+\beta<1\}$ is bounded and satisfies $0<\mathpzc{l}({A_\beta})<\infty$ and $\mathpzc{l}(\partial A_\beta)=0$. Then,
		\begin{enumerate}
			\item if $2\alpha^*<\min\{\alpha_I,\alpha_J\}$, then for all $0\leq \beta\leq \beta_0$
			\begin{equation}\label{eq:occupationG}
				\lim_{t \to +\infty}\frac{\int_{0}^{t}\mathds{1}_{[B(L(s))\in A_\beta]}ds}{t}=1,\,\mathds{P}^{(x,0)}\text{-a.s.};
			\end{equation}
			\item  if $2\alpha^*>\min\{ \alpha_I,\alpha_J\}$, then for any $K>0$,
			\begin{equation}\label{eq:occupationG1}
				\lim_{t \to +\infty}\frac{\int_{0}^{t}\mathds{1}_{\big[B(L(s))\in A^c_\beta\cap [-K,K]^c\big]}ds}{t}=1,\,\mathds{P}^{(x,0)}\text{-a.s.}.
			\end{equation}
		\end{enumerate}
	\end{theorem}
	The previous theorem shows that, whenever the minimum of $\alpha (x)$ is small enough, the process is trapped in it, in the sense that the occupation measure behaves, asymptotically, linearly as $t$. If, instead, the minimum is not small enough, this trapping effect does not take place. This statement can be made stronger by assuming that the function $\alpha(x)$ jumps on the minimum, as follows.
	\begin{theorem}[{\cite[Theorem 4.13]{savtoa}}]
		\label{thm:hyp}
		Let  $\alpha:\mathbb{R}\mapsto (0,1)$ with $\alpha^*=\min_{x\in\mathbb{R}}\alpha(x)>0$, $\max_{x\in\mathbb{R}}\alpha(x)<1$ and \[1>\lim_{x\to +\infty}\alpha(x)=\alpha_I,\,\,1>\lim_{x\to -\infty}\alpha(x)=\alpha_J.\]
		
		Let $A_0=\{x\in\mathbb{R}:\alpha(x)=\alpha^*\}=\bigcup_i I_i,$ be bounded where $I_i$ are disjoint intervals. Assume also $\mathpzc{l}({A_0})\in({0,\infty})$, $\mathpzc{l}({\partial A_0})=0$, and that for all small enough $\beta>0$ we have  $A_0=A_\beta$, where $A_\beta=\{x\in\mathbb{R}:\alpha(x)<\alpha^*+\beta\}$.
		
		Then, if $2\alpha^*<\min\{\alpha_I,\alpha_J\}$, it holds true that 
		\begin{equation}\label{eq:hyp}
			\lim_{t \to +\infty}\mathds{P}^{(x,0)}({B(L(t))\in A_0})=1.
		\end{equation}
	\end{theorem}
	For clarity, let us consider a special case, which may be of the greatest interest.
	\begin{corollary}[{\cite[Corollary 4.14]{savtoa}}]\label{cor:conseq}
		Let   $\alpha:\mathbb{R}\mapsto (0,1)$ be piece-wise constant taking values $0<\alpha_1<\alpha_2<\dots<\alpha_n<1$. Assume that $A=\{x\in\mathbb{R}:\alpha(x)=\alpha_1\}$ is a finite union of disjoint intervals such that $\mathpzc{l}({A})\in(0, +\infty)$ and $\min\{\lim_{x\to +\infty}\alpha(x);\lim_{x\to -\infty}\alpha(x)\}=\alpha_j$, for some $j\in\{2,\dots,n\}$. If $2\alpha_1<\alpha_j$ then 
		\begin{equation}\label{eq:hyp1}
			\lim_{t\to +\infty}\mathds{P}^{(x,0)}({B(L(t))\in A})=1.
		\end{equation}
		Otherwise, if $2\alpha_1>\alpha_j$ and $A_j=\{x\in \mathbb{R}:\alpha(x)=\alpha_j\}$, then
		\begin{equation}\label{eq:hyp2}
			\lim_{t\to +\infty}\mathds{P}^{(x,0)}({B(L(t))\in A_j})=1.
		\end{equation}
	\end{corollary}
	Theorem \ref{thm:occupation} can be formulated also when the region where $\alpha(\cdot)$ is minimum is unbounded, subject to growth conditions on $\mathpzc{l}(A \cap [-x,x])$ as $x \to +\infty$, see \cite[Section 4.2]{savtoa}.

	\section{Killed time-changed Markov process}
	\label{section_killed}
	In this section, we consider an evolution equation for the process $X^+(t) = A(L(t))$, $t\ge0$, killed upon exiting an open set $D$ (a subset of the state space). 
	
	Here, we treat only the uncoupled case in which $A$ is a suitable Feller process, while $L$ is the inverse of a strictly increasing subordinator $S$, independent of $A$. Hence, $X^+(t)=X(t)$ a.s. for all $t\ge0$. {The evolution equation in the uncoupled case has been studied in various settings. The research started with the fractional derivatives in time, and with the Laplacian as the spatial operator. Further developments include more general non-local derivatives and more general spatial operators, as well as different types of problems that include adding potentials (the so-called Schr\"odinger equations) to treat the soft killing effect, adding forcing terms, treating semilinear problems, and nonhomogeneous boundary data, among others. This line of research has led to a broad and rich theory; we refer the reader to the notable works  \cite{BaeumerLuksMeersc-SpaceTime18,DuTon_SPA,CKKW_FM18,Ton19jmaa,CDX21,KW25,CMN12,CGCV24}, and the references therein. In this section, we introduce a new contribution to this theory. Namely, to treat the evolution equation pointwise, theory based on strong solutions with initial conditions in the  domain of the spatial generator is usually employed. However, as we show below, such theory, although rather general, does not cover some very reasonable initial conditions. On the other hand, weak-solutions theory of these evolution problems is also well-established, but from the standpoint of applications, a pointwise theory is often more desirable. Therefore, we show a novel approach on how to solve the evolution equation pointwise for initial conditions that do not have to be inside of the domain of the spatial operator.}
	
	First, we show a tight bond between the killed processes $A$ and $X$ (both killed upon exiting the open set $D$). Let $\tau_D\coloneq \inf\{t>0:X(t) \notin D\}$, be the exit time of the process $X$ from the open set $D$, and let $T_D \coloneqq \inf  \{ t > 0 : A(t) \notin D \}$ be the exit time of the process $A$ from $D$. Since $X$ is obtained by subordination of $A$ with the inverse subordinator $L$, there is a strong connection between the exit times $T_D$ and $\tau_D$. In particular, since $L$ has continuous trajectories (i.e. $S$ is strictly increasing), it can be seen that the process $X$ killed upon exiting $D$, i.e.
	\begin{align}
		X^D(t) = \begin{cases} X(t), \qquad & t <\tau_D, \\ \partial, & t \geq \tau_D ,\end{cases}
	\end{align}
	where $\partial$ is an added point to the state space usually called the cemetery (absorbing point), is almost surely for every $t \geq 0$ equal to the process $A^D(L(t))$, $t\ge0$, where $A^D$ is the process $A$ killed upon exiting $D$, i.e.
	\begin{align}
		A^D(t) = \begin{cases}
			A(t), \qquad &t < T_D, \\
			\partial, & t \geq T_D.
		\end{cases}
		\label{killed_process}
	\end{align}
	Indeed, since $S$ is strictly increasing, we have that $L(t) < h$ if and only if $S(h-)>t$, and we have $\tau_D = S(T_D-)$. From this, it easily follows that $X^D=A^D(L)$. Moreover, recall that $S$ is independent of $A$, and recall that at any fixed time $t$, the probability that {the subordinator $S$} jumps is zero, so we have $S(T_D-)=S(T_D)$ almost surely. Hence, $\tau_D=S(T_D)$ almost surely.

	One might now be tempted to apply directly Theorem \ref{theorem_convolution}, as follows. Since $A$ is a Markov process, $A^D$ remains Markovian and we can consider its semigroup $P^D_t$, $t \geq 0$,  generated by $(G^D, \mathcal{D}(G^D))$ on a suitable Banach space $\mathfrak{B}$. Under suitable assumptions, e.g. if the process $A$ is \textit{doubly Feller} and $D$ is \textit{regular} for $A$ (see \cite{Chung-DoublyFeller-Seminar85}), $P^D_t$, $t \geq 0$, is a semigroup on $C_0(D)$.
	
	Then, under the assumptions of Theorem \ref{theorem_convolution}, the mapping $t \mapsto \mathcal{P}_t^Du$, where 
	\begin{align}
		\mathcal{P}_t^Du \, \coloneqq \, \int_0^{+\infty} P_s^D u \, l_t(ds)=\EE^{(x,0)} u(A^D(L(t)))=\EE^{(x,0)} u(X^D(t)),
	\end{align}
	solves the equation
	\begin{align}\label{eq:Dirichlet-time-frac}
		\phi(\partial_t) (q(t)-q(0)) = G^Dq(t), \qquad q(0) = u \in \mathcal{D}(G^D),
	\end{align}
	in $\mathfrak{B}$.
	A major problem in applications is that in general the generator of the killed process $(G^D,\mathcal{D}(G^D))$ is a quite mystical object, in the sense that it is not clear which functions are in $\DD(G^D)$. If the generator of $A$ is a nice local operator, for example if $A$ is an It\^o diffusion process with smooth enough coefficients and if the boundary of $D$ is smooth enough, the corresponding semigroup $P_t^D$, $t\ge0$, is Feller semigroup, i.e., it is a strongly continuous sub-Markovian semigroup on $C_0(D) = \{ u \in C({D}): \forall \varepsilon>0, \, \exists K\subset D \text{ compact s.t. } |u(x)|<\varepsilon \text{ for all $x\in D\setminus K$}\}$ equipped with the supremum norm, and it moreover holds $C_c^2(D) \subset \mathcal{D}(G^D)$ (see e.g. \cite[Theorem 21.28]{SchillingBMSC} and \cite[Theorem 2.3]{BaeumerLuksMeersc-SpaceTime18}). However, the property $C_c^2(D) \subset \mathcal{D}(G^D)$ does not hold if the process $A$ is jumping, as we demonstrate in Proposition \ref{p:generator domain}. In other words, determining which functions $u$ indeed belong to $\DD(G^D)$ and thus can be taken as the initial condition in \eqref{eq:Dirichlet-time-frac} is not apparent. Moreover, we would indeed like to take $u\equiv \1$, so that the expected solution to \eqref{eq:Dirichlet-time-frac} is $q(x,t)=\mathcal{P}_t\1(x)=\PP^{(x,0)}(\tau_D>t)$. But $\1$ is certainly  not in $C_0(D)$-domain of the generator $G^D$ because it does not vanish at the boundary.\footnote{Moreover, the consideration of the semigroup $P_t^D$ on $C_b(D)$ or $L^2(D)$ (if that is possible) will not fix the problem that $\1_D\notin G^D$ in general. E.g. in the fractional Laplacian case $G=-(\Delta)^{\alpha/2}$ for $\alpha\in (0,2)$, and if $\partial D$ is smooth, it can be seen that $c^{-1}\dist(x,\partial D)^{-\alpha}\le -(\Delta)^{\alpha/2}\1_D(x)\le \dist(x,\partial D)^{-\alpha}$, $x\in D$.}
	
	We overcome this problem by looking at the problem \eqref{eq:Dirichlet-time-frac} in the pointwise sense rather than in the strong sense, and instead of a general Feller process $A$, we deal with an element of a subclass of L\'evy processes, namely $A$ will be a subordinate Brownian motion. We find that this simplified but still rather general setting allows us to deal with $\BB_b(D)\cap C^2_{loc}(D)$ functions as initial conditions. Moreover, note that such functions are certainly not contained in $\DD(G^D)$ (since vanishing on the boundary is not demanded), which suggests that looking at the problem \eqref{eq:Dirichlet-time-frac} in the strong sense may indeed be too restrictive, i.e. it shows the limitations of the abstract Cauchy problem. 
	
	Since we will not directly rely on the abstract Cauchy problem technique, and the corresponding implied regularity of solutions  (e.g. as in the proof of Theorem \ref{theorem_convolution}), we will need to accomplish the same kind of necessary regularity some other way around. To this end, we will first solve the local in time parabolic problem (corresponding to \eqref{eq:Dirichlet-time-frac}) in Subsection \ref{ss:parabolic-local}, and then building on those results, we will provide the evolution equation for $X^D$ by moving to non-local in time parabolic problems in Subsection \ref{ss:parabolic-non-local}.

	Now we construct the underlying process $(A(t),S(t))$, $t\ge0$, and record some preliminary results. Let $B=(B_t,\, t \geq 0)$ be a Brownian motion in $\mathbb{R}^d$, $d\ge 1$, with the characteristic exponent $\xi\mapsto |\xi|^2$. Let $S=(S(t),\,t \geq 0)$, and $\sigma=(\sigma(t),\,t \geq 0)$, be two independent subordinators, independent of $B$. Assume that $\sigma$ has the Laplace exponent 
	\begin{align}\label{eq:psi Laplace defn}
		\psi(\lambda)=b_\psi\lambda+\int_0^\infty(1-e^{-\lambda t})\mu(dt),
	\end{align}
	with $b_\psi>0$ or $\mu(0,+\infty)=+\infty$, so that  $\sigma$ is not a compound Poisson process but may have a drift. The Laplace exponent of  $S$, denoted by $\phi$, is assumed to be as in \eqref{bernstein-eq0} with $b>0$ or $\nu(0,+\infty)=+\infty$. Note that under these assumptions, both $\sigma$ and $S$ are strictly increasing.
	
	Set $A(t) =  B (\sigma(t))$, $t \geq 0$, which defines a process called a subordinate Brownian motion. The process $A$ is a L\'evy process with the characteristic exponent $\xi\mapsto\psi(|\xi|^2)$, see \cite{sato} for more details. For an open set $D\subset \mathbb{R}^d$, denote by $A^D$ the process $A$ killed upon exiting the set $D$, defined as in \eqref{killed_process}.
	
	In this section, it is irrelevant to consider $S(0) \neq 0$ and $\sigma(0) \neq 0$ and so to simplify the notation, we consider the Markov process $(A(t), S(t))$, $t\ge0$, only with respect to the family of probability measures $\mathds{P}^{x} \coloneqq \mathds{P}^{(x,0)}$, and we only consider the case when $\sigma(0)=0$. Under each such measure,
	\begin{align}
		& \mathds{E}^x e^{-\lambda \sigma (t)} = e^{-t \psi (\lambda)},\quad t,\,\lambda\ge0,\\
		&  \mathds{E}^x e^{-\lambda S (t)} = e^{-t \phi (\lambda)},\quad t,\,\lambda\ge0. 
	\end{align}
	Therefore, when applied to events in the sigma-algebra generated by $S(t), t \geq 0$, or $\sigma(t), t \geq 0$, we will avoid the superscript in the probability measure, i.e.,  $\mathds{P} \coloneqq \mathds{P}^x$, for $x$ arbitrary, because $x$ is irrelevant in this context.
	
	We will look at general open sets $D$, but sometimes we will need additional assumptions on $D$ such as regularity of $\partial D$ in the following sense: We say that $z\in \partial D$ is regular for $A$ if $\mathds{P}^z(T_D=0)=1$, i.e. if the process started at $z\in \partial D$ immediately exits the set $D$. Otherwise, we say $z$ is irregular. If all points on $\partial D$ are regular for $X$, we say $D$ is regular for $X$. At this point, we also note that it is well known that $T_D$ satisfies: $\mathds{P}^x(T_D=t)=0$, for all $x\in \mathbb{R}^d$ and $t>0$, see e.g. the proof of \cite[Proposition 1.20]{chung_zhao}.
	
	Hence, in the notation of Section \ref{sec3_non-local}, we have $(A,S)=(B(\sigma),S)$, and we are interested in the evolution equation for the process $$X(t) = A(L(t)) = B(\sigma(L(t))), \quad t \geq 0,$$
	where $L$ is the inverse of $S$, killed upon exiting the open set $D$, denoted by $X^D$.
	
	\subsection{Subordinate Brownian motion: transition kernel, semigroup and generator}
	The transition density of the Brownian motion $B$ is given by
	\begin{align}\label{eq:trans density BM}
		p_B(t,x,y)=p_B(t,x-y)\coloneqq \frac{1}{(4\pi t)^{d/2}}e^{-\frac{|x-y|^2}{4t}}, \quad t>0,\, x,y\in\mathbb{R}^d.
	\end{align}
	Since the process $A$ is obtained by subordinating the Brownian motion B, it is symmetric, and thus has a symmetric transition density, which is given by
	\begin{align}\label{eq:trans density SBM}
		p(t,x,y)=p(t,x-y)=\int_0^\infty p_B(s,x,y)\mathds{P}(\sigma(t)\in ds)=\mathds{E}[p_B(\sigma(t),x,y)].
	\end{align}
	Note that the density $p$ is strictly positive, jointly continuous off the diagonal, and can explode at the diagonal since $p(t,0)=(4\pi)^{-d/2}\mathds{E}[\sigma(t)^{-d/2}]\in(0,\infty]$, and we will return to this peculiarity a few paragraphs below.
	
	The transition density of the killed process $A^D$ is given by Hunt's formula
	\begin{align}\label{eq:trans density KSBM}
		p^D(t,x,y)=p(t,x,y)-\mathds{E}^x[p(t-T_D,A({T_D}),y);{t>T_D}], \qquad t>0,\, x,y\in\mathbb{R}^d,
	\end{align}
	which can be verified by using the strong Markov property. It is well known that $p^D$ is also symmetric on $\mathbb{R}^d\times \mathbb{R}^d$ and jointly continuous off the diagonal, e.g. we can repeat the calculations from \cite[Theorem 2.4]{chung_zhao}, see also \cite[Lemma 2]{Chung-DoublyFeller-Seminar85}. The following Lemma will also be useful later.
	\begin{lemma}\label{l:ap:Pt in C0}
		The transition density $p^D$ satisfies
		\begin{align}
			\lim_{D\ni x\to z}p^D(t,x,y)=0,
		\end{align}
		for all $y\in D$ and for all regular points $z\in \partial D$ for $X$.
		
		If $D$ is regular, then for all $u\in \BB_b(D)$ we have $ P^D_tu\in C_0(D)$, $t>0$.
	\end{lemma}
	\begin{proof}
		The first part of the proof follows by the same calculation as in the case of Brownian motion, see \cite[Theorem 2.4]{chung_zhao}.
		
		For the second part, note that $p^D\le p$, so the vanishing at the boundary can be obtained by using the first part of the lemma and the dominated convergence theorem. The continuity in the interior point $x\in D$ is obtained by separating $D=B(x,\de(x)/2)\cup (D\setminus B(x,\de(x)/2)$\footnote{Here $\de(x)\coloneqq \dist(x,D^c)$.}, using the joint continuity of $p^D$ off the diagonal and the dominated convergence theorem.
	\end{proof}
	
	Moreover, the next lemma states that $x\mapsto p(t,x,y)$ and $x\mapsto p^D(t,x,y)$ are differentiable away from $y$ which we will use many times throughout the section. Recall that the diagonal behavior of $p$ (as well as $p^D$) depends on the subordinator $\sigma$: $p(t,0)=(4\pi)^{-d/2}\mathds{E}[\sigma(t)^{-d/2}]\in(0,\infty]$. In particular, if $b_\psi>0$, then obviously $p(t,0)<\infty$, or if $\psi$ satisfies the so-called weak lower scaling condition at infinity, then  $p(t,0)<\infty$, too, as can be seen in Lemma \ref{l:density in origin}. In such cases, $p$ and $p_D$ are regular at fixed time $t>0$ even in the case when $x=y$. However, $p$ explodes on the diagonal in the case when $\psi(\lambda)=\log(1+\lambda^{\alpha/2})$, i.e. when $A$ is the so-called geometric stable process, see \cite[Theorem 2.6]{SSV06-pot_th_gsp}. 
	\begin{lemma}\label{l:differ of dens SBM SBKM}
		Let $r_0>0$. For every multi-index $\beta=(\beta_1,\dots,\beta_d)\ge0$, there exist constants $c_1,c_2,c_3>0$ such that for all $t>0$ and all $x,y\in \mathbb{R}^d$ such that $|x-y|>r_0$ it holds that
		\begin{align}
			|\partial^\beta_x p(t,x,y)|&\le c_1(\beta, r_0),\label{eq:differ of dens SBM}\\
			|\partial^\beta_x p^D(t,x,y)|&\le c_3(\beta, r_0,\de(x)),\label{eq:differ of dens SBKM}
		\end{align}
		and for all $t>0$ and $x,y\in D$ it holds that
		\begin{align}
			|\partial^\beta_x \mathds{E}^x\left[p(t-T_D,A({T_D}),y);t>T_D\right]|&\le c_2(\beta, \de(x)). \label{eq:differ of dens SBKM-2}
		\end{align}
		Moreover, both $x\mapsto p(t,x,y)$, $x\in \mathbb{R}^d$, and $x\mapsto p^D(t,x,y)$, $x\in D$, are infinitely differentiable in $x\ne y$, and for every multi-index $\beta$ it holds that $\lim_{|x|\to\infty}\partial^\beta_x p(t,x,y)=0$.
	\end{lemma}
	\begin{proof}
		Let us first deal with $p(t,x,y)$. By using the representation \eqref{eq:trans density BM} we first show
		\begin{align}\label{eq:differ SBM eq0}
			\partial^\beta_xp(t,x,y)=\int_0^\infty \partial^\beta_x p_B(s,x,y)\mathds{P}(\sigma(t)\in ds).
		\end{align}
		Indeed, by using \eqref{eq:trans density BM}, it is easy to see that $\partial^\beta_x p_B(s,x,y)$ can be absolutely bounded by a (finite) linear combination of the following terms 
		\begin{align*}
			\frac{|x-y|^{c_3(\beta)}}{s^{d/2+c_2(\beta)}}e^{-\frac{|x-y|^2}{4s}}, \quad s>0,\,x,y\in \mathbb{R}^d,
		\end{align*}
		for integer valued constants $c_2(\beta)$ and $c_3(\beta)$. Hence, if we take $|x-y|>r_0$, by the rapid decay of exponential function, we have $|\partial^\beta_x p_B(s,x,y)|\le C(\beta,r_0)$, $s>0$, and $\partial^\beta_x p_B(s,x,y)$ vanishes at infinity. In other words, we can use the dominated convergence theorem to get \eqref{eq:differ SBM eq0}, and consequently \eqref{eq:differ of dens SBM}, with $\partial^\beta_x p(t,x,y)$ vanishing at infinity.
		
		Next we show \eqref{eq:differ of dens SBKM-2}. Notice that from the symmetry of $p$ and $p^D$ we have
		\begin{align}\label{eq:symmetry of transition remainder}
			\mathds{E}^x\left[p(t-T_D,A({T_D}),y);t>T_D\right]=\mathds{E}^y\left[p(t-T_D,A({T_D}),x);t>T_D\right].
		\end{align}
		Hence, from the beginning of the proof, and the bound obtained for $p$ we have \begin{align*}
			|\partial^\beta_x\mathds{E}^y\left[p(t-T_D,A({T_D}),x);t>T_D\right])|&=|\mathds{E}^y\left[\partial^\beta_x p(t-T_D,A({T_D}),x);t>T_D\right]|\\&\le C(\beta,\de(x)),
		\end{align*}
		where we used that $|x-A({T_D})|\ge \de(x)$ and
		\begin{align*}
			\sup_{s>0,h\ge\delta}\frac{h^a}{s^{b}}e^{-\frac{h^2}{s}}\le C(\delta,a,b),\quad a,b\ge 0.
		\end{align*}
		The claim \eqref{eq:differ of dens SBKM} follows from already proved claims \eqref{eq:differ of dens SBM} and \eqref{eq:differ of dens SBKM-2}.
	\end{proof}
	
	The probability density $p(t,\cdot)$ falls into a general category of the so-called unimodal densities, i.e. $p(t,\cdot)$ radially decreases. Hence, by \cite[Corollary 7]{bogdan_density_and_tails_unimodal}, we have
	\begin{align}\label{eq:upper bound density p}
		p(t,x)\le C(d,\psi) \frac{t \psi(|x|^{-2})}{|x|^d},\quad t>0,\, x\in \mathbb{R}^d.
	\end{align}
	For large $t$, if $|x|>r_0$, the bound from Lemma \ref{l:differ of dens SBM SBKM} is better, so we get $p(t,x)\le C(d,\psi,r_0)\, (1\wedge t)$, $t>0$.
	As we have said, in general $x\mapsto p(t,x)$ can explode at the origin, but under an additional assumption this is impossible due to the next lemma (which serves as an addition to Lemma \ref{l:differ of dens SBM SBKM}).
	\begin{lemma}\label{l:density in origin}
		Assume that the Laplace exponent  $\psi$, in the case when $b_\psi=0$ in \eqref{eq:psi Laplace defn},  satisfies the weak lower scaling at infinity, i.e., there exist $a_\psi>0$ and $\delta_\psi\in (0,1)$ such that
		\begin{align}\label{eq:LWSC}
			\psi(\lambda t)\ge a_\psi \lambda^{\delta_\psi} \psi(t),\quad t\ge 1, \, \lambda\ge 1.
		\end{align}
		Then, for every multi-index $\beta\ge 0$, $t_0>0$, and $k\in \mathbb{N}\cup\{0\}$ we have
		\begin{align*}
			|\partial^{(k)}_tp(t,x)|&\le C(t_0,k),\quad t\ge t_0,\, x\in \mathbb{R}^d.\\
			|\partial^\beta_x p(t,x)|&\le C(t_0, \beta),\quad t\ge t_0,\, x\in \mathbb{R}^d.
		\end{align*}
	\end{lemma}
	\begin{proof}
		
		By the inverse Fourier transform, the density $p$ can be written as
		$$ p(t,x)=\frac{1}{(2\pi)^{d}}\int_{\mathbb{R}^d} e^{-i x\cdot \xi} e^{-t \psi (|\xi|^2)}d\xi,\quad t>0,\, x\in \mathbb{R}^d.$$
		Note that every Bernstein function satisfies
		$$1\wedge \lambda \le \frac{\psi(\lambda \, t)}{\psi(t)}\le 1\vee \lambda,\quad \lambda >0.$$
		In particular, $\psi(\lambda)\le\psi(1) \lambda$, $\lambda\ge 1$.
		
		Assume now that $b=0$. Then by the lower scaling, and the dominated convergence theorem, we get for $t\ge t_0$ and all $x\in \mathbb{R}^d$
		\begin{align}
			|\partial_tp(t,x)|&\le \int_{\mathbb{R}^d}\psi(|\xi|^2)e^{-t \psi (|\xi|^2)}d\xi\le \psi(1)+\int_{B(0,1)^c}|\xi|^2e^{-t\, a\,\psi(1) |\xi|^{2\delta}} \le C(t_0),\label{eq:density diff t eq1}
		\end{align}
		and for $t\ge t_0$, by passing the derivative inside the integral we have
		\begin{align}
			|\partial_{x_i}p(t,x)|&\le \int_{\mathbb{R}^d}|\xi|e^{-t \psi (|\xi|^2)}d\xi\le \psi(1)+\int_{B(0,1)^c}|\xi|e^{-t\, a\,\psi(1) |\xi|^{2\delta}} \le C(t_0).\label{eq:density diff x eq2}
		\end{align}
		In a similar way we proceed in the case $k\ge 2$ or for a general multi-index $\beta$.
		
		In the case when $b> 0$, we have $\psi(|\xi|^2)\ge b |\xi|^2$, so by using $e^{-t\psi(|\xi|^2)}\le e^{-tb |\xi|^2}$, we can bound the derivatives of $p$ similarly as in  \eqref{eq:density diff t eq1} and \eqref{eq:density diff x eq2}.
	\end{proof}
	By $P_t$, $t \geq 0$, we denote the semigroup of $A$, i.e.
	\begin{align}\label{eq:semigroup SBM}
		P_tu(x)=\mathds{E}^x[u(A(t))]=\int_{\mathbb{R}^d}p(t,x,y)u(y)dy.
	\end{align}
	We are little imprecise by not writing what is the underlying Banach space for $P_t$, $t\ge0$, but this is intentional. Namely, the process $A$ is a L\'evy process, so \eqref{eq:semigroup SBM} generates a strongly continuous semigroup on $C_0(\R^d)$, i.e. a Feller semigroup, see \cite[Section 2.1]{schillinglevy}. However, \eqref{eq:semigroup SBM} can be also uniquely extended to $L^2(\R^d)$ semigroup, see \cite[Section 1.5]{schillinglevy}.  We will actively use both semigroups and in each usage we will always indicate what is the underlying Banach space. Moreover, $P_t$ satisfies the so-called strong Feller property, i.e. $P_t (\BB_b(\R^d))\subset C_b(\R^d)$, since $A$ has an absolutely continuous transition densities which is a classical result due to Hawkes \cite{Hawkes-potential-levy79}.
	
	By $P^D_t$, $t\ge0$, we denote the corresponding semigroup of the killed process, i.e.
	\begin{align}\label{eq:semigroup KSBM}
		P^D_tu(x)=\mathds{E}^xu(A^D(t))=\mathds{E}^x[u(A(t));t<T_D]=\int_Dp^D(t,x,y)u(y)dy.
	\end{align}
	Here we extended all functions $u$ on $\partial$ by setting $u(\partial)=0$. It can be easily seen that the semigroup $P^D$ induces a $L^2(D)$-semigroup. Also, $P^D_t$ satisfies the strong Feller property. Indeed, for $u\in \mathcal{B}_b(D)$ we have
	\begin{align*}
		P^Du(x)=P_t\big(u\mathds{1}_D\big)(x)-\int_D\mathds{E}^y[p(t-T_D,A({T_D}),x)\mathds{1}_{[t>T_D]}]u(y)dy,
	\end{align*}
	where the first term is in $C_b(D)$ since $T$ is strongly Feller, and the second term is in $C_b(D)$ since $p$ is continuous off the diagonal. On the other hand, $P^D_t$, $t\ge0$, is $C_0(D)$ semigroup only if $D$ is regular for $A$, which easily follows from Lemma \ref{l:ap:Pt in C0}, see also \cite{Chung-DoublyFeller-Seminar85}.

	Recall that  $A$ is a L\'evy process so it has the L\'evy measure, denoted by $J(dx)$. Under our assumptions, see \cite[Theorem 30.1]{sato}, the measure $J$ is absolutely continuous: $J(dx)=j(|x|)dx$, where 
	\begin{align*}
		j(r)\coloneqq\int_0^\infty \frac{1}{(4\pi t)^{d/2}}e^{-\frac{r^2}{4t}}\mu(dt),\quad r>0.
	\end{align*}
	Note that $j$ is a decreasing function such that $\lim_{r\to \infty} j(r)=0$.
	
	The $C_0(\mathbb{R}^d)$-infinitesimal generator of $A$, denoted by $-\Lpsi$, acts on $C_0^2(\mathbb{R}^d)$ as
	\begin{align}
		-\Lpsi u(x)&=b_\psi\Delta u(x)+\int_{\mathbb{R}^d}\big(u(y)-u(x)-\nabla u(x)\cdot (y-x)\mathds{1}_{[|x-y|\le 1]}\big)j(|y|)dy\nonumber\\
		&=b_\psi\Delta u(x)+\textrm{P.V.}\int\big(u(y)-u(x)\big)j(|y|)dy, \quad u\in C_0^2(\mathbb{R}^d),\label{eq:Lpsi defn}
	\end{align}
	where for details we refer to \cite{sato}.
	We extend the definition of the operator $\Lpsi$ in the pointwise sense whenever the expression in \eqref{eq:Lpsi defn} makes sense, which is  true, e.g., for all $x\in\mathbb{R}^d$ if $u\in C^{2}(\mathbb{R}^d)\cap L^1(\mathbb{R}^d,1\wedge j(|x|)dx)$.
	
	Denote by $-\LpsiD$ the infinitesimal generator of $P^D_t$, $t\ge0$, both in the $L^2(D)$ sense and in $C_0(D)$ (where we recall that the latter is considered only if $D$ is regular for $A$). The domain of this operator can be tricky to find. We emphasize that in the case of the $C_0(D)$-semigroup it is possible that   $C_c^\infty(D)\not\subset \DD(-\LpsiD)$ which we prove in the next proposition. On the other hand, $C_c^\infty(D)$ is in the $L^2(D)$ domain of $-\LpsiD$. We refer to \cite{BaeumerLuksMeersc-SpaceTime18} for an important and detailed approach on the generators of killed $C_0(D)$-semigroups.
	
	\begin{proposition}\label{p:generator domain}
		Let $\mathpzc{l}(D)<+\infty$\footnote{We keep denoting by $\mathpzc{l}(\cdot)$ the Lebesgue measure in $\R^d$.}. The space $C_c^2(D)$ is contained in the $L^2(D)$-domain of $-\LpsiD$ and for $u\in C_c^2(D)$ it holds that
		\begin{align*}
			\LpsiD u=\Lpsi u,\quad \textit{in $D$}.
		\end{align*}
		On the other hand, if $\mu\not\equiv0$, and if $u\in C_c^2(D)$, $u\not\equiv 0$, with $u\ge0$ or $u\le 0$, then $u$ is not in the $C_0(D)$-domain of $-\LpsiD$ (where the regularity of $D$ is implicitly assumed).
	\end{proposition}
	\begin{proof}
		Take $u\in C_c^2(D)$. First we show that for all $x\in D$, it holds
		\begin{align}\label{eq:killed semigr conv}
			\lim_{t\to0}\frac{ P^D_tu(x)-u(x)}{t}=-\Lpsi u(x),
		\end{align}
		which implies that $-\Lpsi u$ is the only candidate for $-\LpsiD u$ both in $L^2(D)$ and in $C_0(D)$.
		
		To prove \eqref{eq:killed semigr conv}, write
		\begin{align}\label{eq:semi decomp}
			\frac{ P^D_tu(x)-u(x)}{t}=\frac{P_tu(x)-u(x)}{t}-\frac{\wt  P^D_tu(x)}{t},
		\end{align}
		where
		\begin{align}\label{eq:remainder of semigr}
			\wt  P^D_tu(x)\coloneqq \int_D \mathds{E}^x\left[p(t-T_D,A({T_D}),y);t>T_D\right]u(y)dy=\mathds{E}^x\left[P_{t-T_D}u(A({T_D}));t>T_D\right].
		\end{align}
		For the first term in \eqref{eq:semi decomp} we have $	\frac{P_tu(x)-u(x)}{t}\to -\Lpsi u(x)$ uniformly in $\mathbb{R}^d$, and 
		\begin{align}
			\left\|\frac{P_tu-u}{t}\right\|_{L^\infty(D)}\le 	\left\|\frac{P_tu-u}{t}\right\|_{L^\infty(\mathbb{R}^d)}\le C \|-\Lpsi u\|_{L^\infty(\mathbb{R}^d)}<\infty,\quad t<1,
		\end{align}
		where for the second inequality we used \cite[Eq. (15.3)]{bernstein}, and for the last $|\Lpsi u(x)|\le c_1\|u\|_{C^2(\mathbb{R}^d)}(1\wedge j(|x|))\le c_2$, $x\in \mathbb{R}^d$, see e.g. \cite[Eq. (6)]{Bio22}.
		
		For the second term in \eqref{eq:semi decomp}, since $u\equiv 0$ on $D^c$ and by using the same arguments as above, we have that $\frac{P_{t-T_D}u(A({T_D}))\mathds{1}_{[t>T_D]}}{t}$, for $t\in(0,1)$,
		is uniformly bounded $\mathds{P}^x$-a.s.,  with the same bound for all $x\in D$, so $\frac{ \widetilde{P}^D_tu}{t}$ is uniformly bounded in $D$, too. Notice that $\lim_{t\to0}\frac{P_{t-T_D}u(A({T_D}))\mathds{1}_{[t>T_D]}}{t}\to 0$  $\mathds{P}^x$-a.s. since the numerator becomes 0 for small $t$, so by the dominated convergence we have $\widetilde{P}^D_tu(x)/t\to0$ for all $x\in D$. Thus, we have proved \eqref{eq:killed semigr conv}.
		
		To prove that $u$ is in the $L^2(D)$ domain of $-\Lpsi_D$, we use \eqref{eq:semi decomp} to get
		\begin{align*}
			\left\|\frac{ P^D_tu-u}{t}+\Lpsi u\right\|_{L^2(D)}\le 	C\left\|\frac{P_tu-u}{t}+\Lpsi u\right\|^2_{L^\infty(D)}+\left\|\frac{\wt  P^D_t u}{t}\right\|_{L^2(D)}.
		\end{align*}
		The first term goes to 0 by the uniform convergence of $\frac{P_tu-u}{t}$ to $-\Lpsi u$, while the second goes to 0 by using the dominated convergence of $\wt  P^D_tu(x)/t$ to 0 as we explained in the previous paragraph, where in both convergences we used the assumption $\mathpzc{l}(D)<+\infty$.
		
		To prove that  $u$ is not in the $C_0(D)$ domain of $-\LpsiD$, for $u\not\equiv0$, $u\ge0$ or $u\le0$, first recall from the beginning of the proof that $-\Lpsi u$ is the only candidate for $-\LpsiD u$. However, $-\Lpsi u\not\in C_0(D)$ since for $z\in \partial D$ we have $u\equiv0$ in $D\cap B(z,\varepsilon)$ for some $\varepsilon>0$, so for $x\in D\cap B(z,\varepsilon)$, the relation  \eqref{eq:Lpsi defn} becomes
		\begin{align*}
			\Lpsi u(x)=\int_{D}u(y)j(|x-y|)dy.
		\end{align*}
		Since $\mu\not\equiv0$, we have $j>0$, so the previous relation shows that $-\Lpsi f\not\in C_0(D)$.	
	\end{proof}
	
	\subsection{Non-local in space parabolic problem}\label{ss:parabolic-local}
	With minimal assumptions on the open set $D$, we want to solve the problem
	\begin{align}
		\begin{cases}
			\partial_t u(x,t)=-\Lpsi u(x,t), & x\in D,\,t>0,\\
			u(x,t)=0,& x\in D^c,\, t>0,\\
			u(x,0)=1,& x\in D,
		\end{cases}\label{eq:parabolic problem}
	\end{align}
	in the pointwise sense, where we consider $\Lpsi$ in the extended pointwise sense as in \eqref{eq:Lpsi defn}. The obvious candidate for a solution to \eqref{eq:parabolic problem} is $q(x,t)= P^D_t\mathds{1}(x)=\mathds{P}^x(T_D>t)$, i.e. the tail of the exit time distribution, but we need to make this rigorous. The difficulty here lies in the fact that $\mathds{1}\not\in \DD(-\LpsiD)$ in $C_0(D)$ sense so the natural strategy involving the abstract Cauchy problem cannot be directly applied. 
	
	In this quite simple initial condition, we can see the importance of studying pointwise parabolic problems for initial data outside of the domain of the spatial operator. The proof that $\mathds{P}^x(T_D>t)$ solves \eqref{eq:parabolic problem} will also show how to deal with more general initial conditions. Theorem \ref{t:solution parabolic-f} gives a solution to \eqref{eq:parabolic problem} for the initial condition in $\BB_b(D)\cap C_{loc}^2(D)$.
	
	To successfully deal with the parabolic problem, it would be useful to have some nice properties of the operator $\LpsiD$ such as compactness. This was obtained by Chen and Song in \cite{ChenSong-spectralSBM-07} for all $D$ of finite Lebesgue measure. Interestingly, in \cite[Section 3]{ChenSong-spectralSBM-07} there are examples showing that $ P_t^D$ does not have to be a Hilbert-Schmidt operator (for a connection between these classes, see e.g. \cite[Chapter 5]{apple-semigr}).
	\begin{theorem}[{\cite[Theorem 2.1]{ChenSong-spectralSBM-07}}]\label{t:spectrum}
		For any open set $D\subset \mathbb{R}^d$ with $\mathpzc{l}(D)<+\infty$, and every $t > 0$, $ P^D_t$ is a compact operator in $L^2(D)$ and therefore it has a discrete spectrum.
	\end{theorem}
	Theorem \ref{t:spectrum} implies that there exists an orthonormal basis $(e_n)_n$ of $L^2(D)$ consisting of eigenfunctions of $-\Lpsi$, i.e. there exists $0\le \lambda_1\le\lambda_2\le \dots$ such that, $\lim_n\lambda_n=\infty$ and 
	\begin{align*}
		\LpsiD e_n=\lambda_n e_n,
	\end{align*}
	in $L^2(D)$ sense. In particular, for all $f\in \DD(-\LpsiD)$ we have
	\begin{align}\label{eq:LpsiD series}
		\LpsiD f=\sum_{n=1}^\infty \lambda_n \wh f_n e_n, \quad \text{in $L^2(D)$,}
	\end{align}
	where from now on we always write $\wh f_n$ for the $n$-th coefficient of the $L^2(D)$-representation $f=\sum_n \wh f_n e_n$. Also, for all $f\in L^2(D)$ we have
	\begin{align}\label{eq:SKBM group series}
		P^D_t f=\sum_{n=1}^\infty e^{-\lambda_nt} \wh f_n e_n.
	\end{align}
	
	We are ready to show that the tail of the exit distribution $\PP^x(T_D>t)$ solves local in time parabolic problem in the pointwise sense.
	\begin{theorem}\label{t:solution parabolic}
		Let $D$ be regular for $A$ and $\mathpzc{l}(D)<+\infty$. The  problem
		\begin{align}
			\begin{cases}
				\partial_t u(x,t)=-\Lpsi u(x,t), & x\in D,\,t>0,\\
				u(x,t)=0,& x\in D^c,\, t>0,\\
				u(x,0)=1,& x\in D,
			\end{cases}\label{eq:parabolic problem no2}
		\end{align}
		has a pointwise solution $u(x,t)\coloneqq \mathds{P}^x(T_D>t)$.
	\end{theorem}
	\begin{proof}
		Let  $u(x,t)= P^D_t\mathds{1}(x)=\mathds{P}^x(T_D>t)$. It is obvious that $u$ satisfies the boundary conditions of \eqref{eq:parabolic problem no2} since $D$ is regular for $A$ and by the right-continuity of the paths of $A$. 
		
		First we show that $u$ solves \eqref{eq:parabolic problem no2} weakly in the following way:
		\begin{align}\label{eq:weak solution 1}
			\int_D\big(u(x,t)-u(x,0))\varphi(x)dx+\int_0^t\int_D u(x,s)\Lpsi \varphi(x)dx\,ds=0, \quad \forall \varphi\in C_c^\infty(D).
		\end{align}
		Take $\varphi\in C_c^\infty(D)$ and recall that $\varphi \in \DD(-\LpsiD)$ in $L^2(D)$ sense by Proposition \ref{p:generator domain}. By \eqref{eq:SKBM group series} and \eqref{eq:LpsiD series} we have in $L^2(D)$ sense
		\begin{align*}
			P^D_t\mathds{1}&=\sum_{n=1}^\infty e^{-\lambda_n\,t}\wh{\mathds{1}}_n e_n,\\
			\varphi=\sum_{n=1}^\infty\wh\varphi_n e_n,&\quad \Lpsi\varphi=\sum_{n=1}^\infty\lambda_n\wh\varphi_n e_n,
		\end{align*}
		so by using these series representations it is easy to see that \eqref{eq:weak solution 1} holds.

		Now in \eqref{eq:weak solution 1} we want to transfer the operator $\Lpsi$ from acting on $\varphi$ to acting on $u$, i.e.
		\begin{align}\label{eq:change of operator}
			\int_D u(x,s)\Lpsi \varphi(x)dx=\int_D \Lpsi u(x,s) \varphi(x)dx, \quad s>0,
		\end{align}
		where $\Lpsi u(x,t)$ is taken as a pointwise operator. For this it is enough to show that $u(\cdot,s)= P^D_s\mathds{1}$ is $C^2$ in a neighborhood of $\supp \varphi$, for every $s>0$, since then we may repeat the classical calculation from \cite[Section 3]{bogdan1999potential} and use the integration by parts to obtain \eqref{eq:change of operator}. Let us then show that $u(\cdot,s)= P^D_s\mathds{1}$ is in $C_{loc}^2(D)$, $s>0$.
		
		Take $\eta\in C_c^\infty(D)$ such that $0 \le \eta\le 1$ and $\eta\equiv 1$ on an open set $U\subset D$ such that  $\supp\varphi \subset U$ and $\dist(\supp \varphi,D^c)\le 2\dist(U,D^c)$. Write
		\begin{align}\label{eq:decomp semig R1}
			P^D_t\mathds{1}=P_t\eta-\wt  P^D_t\eta+ P^D_t(1-\eta),
		\end{align}
		where $\wt  P^D_t\eta$ is defined as in \eqref{eq:remainder of semigr}. For the first term note that for every multi-index $\beta \ge 0$ we have
		\begin{align*}
			\partial^\beta_x P_t\eta(\cdot)&=\partial^\beta_x\int_{\mathbb{R}^d}p(t,\cdot,y)\eta(y)dy=\partial^\beta_x\int_{\mathbb{R}^d}p(t,h)\eta(\cdot-h)dh=\int_{\mathbb{R}^d}p(t,h)\partial^\beta_x\eta(\cdot-h)dh\\
			&=\int_{\mathbb{R}^d}p(t,\cdot,y)\partial^\beta_x\eta(y)dh=P_t\partial^\beta_x\eta(\cdot),
		\end{align*}
		where we used that $p$ is radial. Thus, $\|P_t\eta\|_{C^k(\R ^d)}\le \|P_t\|\|\eta\|_{C^k(\R ^d)}=\|\eta\|_{C^k(\mathbb{R} ^d)}$, for all $k\ge 0$.
		
		For the second term in \eqref{eq:decomp semig R1}, for any multi-index $\beta\ge0$ we have
		\begin{align*}
			\partial^\beta_x\wt  P^D_t\eta(x) &=	\partial^\beta_x\int_D \mathds{E}^y\left[p(t-T_D,A({T_D}),x);t>T_D\right]\eta(y)dy\\
			&=\int_D \partial^\beta_x\mathds{E}^y\left[p(t-T_D,A({T_D}),x);t>T_D\right]\eta(y)dy,
		\end{align*}
		where we used \eqref{eq:symmetry of transition remainder} and \eqref{eq:differ of dens SBKM-2}. Thus, by using \eqref{eq:differ of dens SBKM-2} again,  we have $\|\partial^\beta_x\wt  P^D_t\eta\|_{C^k(U)}\le C\|\eta\|_{L^\infty(D)}$, for $C=C(\dist(\supp \eta,D^c),k)>0$.
		
		For the third term, note that $1-\eta\equiv 0$ in $U$, so first choose an open $V\supset \supp \varphi$  such that $\overline V\subset U$. Now for $x\in V$ we have
		\begin{align*}
			\partial^\beta_x  P^D_t(1-\eta)(x)=\int_D \partial^\beta_x p^D(t,x,y)(1-\eta)(y)dy,
		\end{align*}
		since the only non-zero part of the integral above is in the case when $|x-y|\ge \dist(U,V^c)>0$ where $p^D$ is differentiable by Lemma \ref{l:differ of dens SBM SBKM}. Hence, $\|  P^D_t(1-\eta)\|_{C^k(V)}\le C\|1-\eta\|_{L^1(D)}$ for $C=C(\dist(U,V^c),k)>0$.
		
		Thus, we have just proved that $u(\cdot,t)= P_t^D\mathds{1}\in C^\infty(D)$, and $\|  P^D_t\mathds{1}\|_{C^k(U)}\le C$ uniformly for all $t>0$ (where $C$ depends on $U$ and $k$),  hence \eqref{eq:change of operator} holds. By implementing \eqref{eq:change of operator} into \eqref{eq:weak solution 1}, we get that for almost all $x\in D$ it holds that
		\begin{align}\label{eq:weak solution 2}
			u(x,t)-u(x,0)+\int_0^t \Lpsi u(x,s)ds=0.
		\end{align}
		Moreover, $x\mapsto\Lpsi u(x,s)$ is continuous and bounded. Indeed, since $D$ is regular, we have $u(\cdot,s)= P^D_s\mathds{1}\in C_0(D)$, see Lemma \ref{l:ap:Pt in C0}. Hence, by writing
		\begin{align}
			\begin{split}\label{eq:operator on PtD}
				\Lpsi u(x,s)&=b_\psi\Delta u(x,s)+\int_{B(0,\delta)^c}\big(u(x+h,s)-u(x,s)\big)j(|h|)dh\\
				&\hspace{3em}+\int_{B(0,\delta)}\big(u(x+h,s)-u(x,s)-\nabla u(x,s)\mathds{1}_{|h|\le \delta}\big)j(|h|)dh,
			\end{split}
		\end{align}
		we can see that the first and the third term in the above expression are continuous in $x$ for all $s>0$ since $u(x,s)= P^D_s\mathds{1}(x)$ is smooth in $D$. The second term is continuous in $D$: $u(\cdot,s)\in C_0(D)$, and hence $u(\cdot,s)\in C_b(\mathbb{R}^d)$ (since $u(\cdot,s)=0$ in $D^c$), so we can use the dominated convergence theorem. 
		
		In other words, \eqref{eq:weak solution 2} holds for all $x\in D$. To finish the proof, we want to take the time derivative in \eqref{eq:weak solution 2}. To do so, note that it is enough to show that $s\mapsto \Lpsi u(x,s)$ is continuous. To this end, look at \eqref{eq:operator on PtD} again. To prove that the first term is continuous in time, recall the decomposition \eqref{eq:decomp semig R1} of $ P_t^D\mathds{1}$. From the proof that $ P_t^D\mathds{1}\in C^\infty$ it is easy to see\footnote{This uses the continuity in time of $p$ (and $p^D$) easily observed, since the probability that $\sigma$ has jump at fixed time $t>0$ is 0.} that $t\mapsto \Delta  P_t^D\mathds{1}$ is continuous. The second term in \eqref{eq:operator on PtD} is continuous in time since $ P_t^D\mathds{1}$ is continuous in time and uniformly  bounded so we may use the dominated convergence theorem. The third term in \eqref{eq:operator on PtD} is arbitrarily small, uniformly for $t>0$, for small $\delta$ since $\| P_t^D\mathds{1}\|_{C^2(U)}\le C$ uniformly in $t>0$ as we proved earlier in the proof. Thus, $s\mapsto \Lpsi u(x,s)$ is continuous in $(0,+\infty)$, hence $t\mapsto \int_0^t \Lpsi u(x,s)ds$ is differentiable in $(0,+\infty)$, hence \eqref{eq:weak solution 2} implies
		\begin{align*}
			\partial_t u(x,t)+\Lpsi u(x,t)=0,\quad t>0.
		\end{align*}
	\end{proof}
	
	By using the same technique as in the previous theorem, we obtain a more general result.
	\begin{theorem}\label{t:solution parabolic-f}
		Let $D$ be regular for $A$ and $\mathpzc{l}(D)<+\infty$, and $f\in \BB_b(D)\cap C_{loc}^2(D)$. The  problem
		\begin{align}
			\begin{cases}
				\partial_t u(x,t)=-\Lpsi u(x,t), & x\in D,\,t>0,\\
				u(x,t)=0,& x\in D^c,\, t>0,\\
				u(x,0)=f(x),& x\in D,
			\end{cases}\label{eq:parabolic problem f}
		\end{align}
		has a pointwise solution $u(x,t)\coloneqq \mathds{E}^x[f(A^D(t))]=\mathds{E}^x[f(A(t));t<T_D]$, and $u(x,t)\to u(x,0)$ as $t\to 0$, for all $x\in D$.
	\end{theorem}
	\begin{proof}
		The only key difference with respect to the proof of Theorem \ref{t:solution parabolic} is to use the decomposition $f=f\eta + f(1-\eta)$ in place of \eqref{eq:decomp semig R1}.
	\end{proof}
	
	We note that in this quite general setting, the transition kernel $p_D$ lacks regularity on the diagonal which does not allow us (by using the same approach) to obtain a smoothing property and instant regularity for the parabolic problem in the case of only integrable data, i.e. some regularity of $f$ is still needed.
	
	\begin{remark}[Generalization in more regular settings]
		If $D$ and $\Lpsi$ posses more regularity, e.g. if $D$ is an open bounded Lipschitz set, and $\Lpsi=(-\Delta)^{\alpha/2}$, $\alpha\in (0,2)$, also known as the $\frac\alpha2$-fractional Laplacian, then it follows from  \cite[Lemma 2.4]{Rut24} that if $ P^D_{t_0}f(x_0)$ is finite for some $t_0>0$ and $x_0\in D$, then $P_tf\in C_b(D)\cap C^\infty(D)$ for all $t>0$.
		
		The proof relies on the decomposition of the heat kernel $p_D(t,x,y)$ in the form: $p^D(t,x,y)\asymp \mathds{P}^x(T_D>t)p(t,x,y)\mathds{P}^y(T_D>t)$, for $x,y\in D$, and $t\in(0,T)$, with a constant of comparability depending on $T$, together with the regularity of $p_D$ up to the boundary in the sense of the existence of the limit $\lim_{D\ni y\to \partial D}\frac{p_D(t,x,y)}{\mathds{P}^y(T_D>1)}$.
		
		The point of this considerations (in this regular setting) is to conclude: if there exists $x_0\in D$ and $t_0>0$ such that $ P^D_{t_0}f(x_0)$ is finite, it follows (by Theorem \ref{t:solution parabolic-f}) that $u(x,t)= P^D_{t+s}f$ solves $(\partial_t+\Lpsi)u(x,t)=0$ with an initial condition $u(x,0)= P^D_sf(x)\in C_b^2(D)$. But here $s$ is arbitrary, and $\partial_t$ is time-shift invariant, so we can see that $u(x,t)= P^D_tf(x)$ satisfies $(\partial_t+\Lpsi)u(x,t)=0$ pointwise in $t>0$. Moreover, it can be seen in \cite[Lemma 2.4]{Rut24} that a nice condition ensuring that $ P^D_{t_0}f(x_0)$ is finite is $$\int_D \mathds{P}^y(T_D>1)|f(y)|dy<\infty.$$
		
	\end{remark}

	\subsection{Non-local in time parabolic problem for non-local operator}\label{ss:parabolic-non-local}
	Now we move to solving an evolution problem for the killed CTRWs limits, in the uncoupled case, i.e. the evolution of $X^D$. We solve the following non-local in time parabolic problem in the pointwise sense:
	\begin{align}
		\begin{cases}
			\phi(\partial_t) \big(U(x,t)-U(x,0)\big)=-\Lpsi U(x,t), & x\in D,\,t>0,\\
			U(x,t)=0,& x\in D^c,\, t>0,\\
			U(x,0)=f(x),& x\in D,
		\end{cases}\label{eq:time-non-local parabolic problem}
	\end{align}
	for $f\in \BB_b(D)\cap C^2_{loc}(D)$, where we show that a solution is $U(x,t)=\EE^xf(X^D(t))$, thus obtaining the connection between $X^D$ and its evolution equation.
	
	Before we give the statement and the proof, we want to emphasize again that the considered class $\BB_b(D)\cap C^2_{loc}(D)$ is not a subset of $\DD(-\LpsiD)$ in $C_0(D)$, but we are able to provide a pointwise solution to the evolution problem for $X^D$. This is especially interesting in the special case $f\equiv 1$, because the solution is then the tail of the exit time distribution of $X^D$, implying the importance of studying pointwise problems for initial conditions outside $C_0$-domains of generators.
	
	The presented approach connects the problems \eqref{eq:parabolic problem} and \eqref{eq:time-non-local parabolic problem} by using the connection between the processes $X^D$ and $A^D$ which are at the heart of the mentioned problems, respectively, as discussed at the very beginning of the section.

	\begin{theorem}\label{t:solution time-non-local parabolic}
		Let $D$ be regular for $A$, $\mathpzc{l}(D)<+\infty$, and $f\in \BB_b(D)\cap C^2_{loc}(D)$. The problem
		\begin{align}
			\begin{cases}
				\phi(\partial_t) \big(U(x,t)-U(x,0)\big)=-\Lpsi U(x,t), & x\in D,\,t>0,\\
				U(x,t)=0,& x\in D^c,\, t>0,\\
				U(x,0)=f(x),& x\in D.
			\end{cases}\label{eq:time-non-local parabolic problem no2}
		\end{align}
		has a pointwise solution $U(x,t)\coloneqq \EE^xf(X^D(t))$. 
		
		In particular, if $f\equiv \1$, the solution is 
		$U(x,t)=\mathds{P}^x(\tau_D>t)$.
	\end{theorem}
	
	\begin{proof}
		Let $U(x,t)\coloneqq \EE^xf(X^D(t))$. Note that since $D$ is regular for $A$, it is also regular for $X$. Indeed, for $x\in \partial D$ we have
		$$\mathds{P}^x(\tau_D=0)=\mathds{P}^x(S(T_D)=0)=\mathds{P}^x(S(T_D)=0,T_D=0)=\mathds{P}^x(T_D=0)=1,$$
		where we used that $\tau_D=S(T_D)$ almost surely. Hence, $U$ satisfies the boundary conditions of \eqref{eq:time-non-local parabolic problem no2}.
		
		We use the technique of the Laplace transform to show that $U$ satisfies \eqref{eq:time-non-local parabolic problem no2}. First we show that $U(x,t)$ satisfies \eqref{eq:time-non-local parabolic problem no2} in the integral form, i.e. we will show that for all $t>0$
		\begin{align}\label{eq:time non-local integral form}
			b\big(U(x,t)-U(x,0)\big)+\int_0^t \big(U(x,s)-U(x,0)\big)\overline \nu(t-s)ds=\int_0^t -\Lpsi U(x,s)ds.
		\end{align}
		To this end, let $x\in D$ and recall that by the independence of $A$ and $S$, we have $A^D(L)=X^D$, hence
		\begin{align}
			\begin{split}
				U(x,t)&=\EE^xf(X^D(t))=\EE^xf(A^D(L(t)))=\int_0^\infty \EE^xf(A^D(s))\p(L(t)\in ds)\\&=\int_0^\infty u(x,s)\,\p(L(t)\in ds)=\mathds{E}[u(x,L(t))],
			\end{split}\label{eq:U relative to u}
		\end{align}
		where $u(x,t)=\EE^xf(A^D(t))$ is a solution  to the problem \eqref{eq:parabolic problem f} from Theorem \ref{t:solution parabolic-f}. Furthermore, $\partial_tu(x,t)=-\Lpsi u(x,t)$ in $D\times (0,+\infty)$, in the pointwise sense, so in particular $\partial_t u(x,t)$ exists for all $t>0$. Now, in the same way as in \eqref{eq:uncoupled-sol}, we can write 
		\begin{align}\label{eq:U intergral repr.}
			U(x,t)&=f(x)+\int_0^{+\infty} \partial_s u(x,s)\PP(L(t)\ge s)ds.
		\end{align}
		From \eqref{eq:U intergral repr.}, we easily calculate the Laplace transform in time of $U$ (as in \eqref{resolv})
		\begin{align}
			\begin{split}
				\LL U(x,\lambda)&=\frac{f(x)}{\lambda}+\int_0^\infty  \partial_s u(x,s)\frac{e^{-\phi(\lambda)s}}{\lambda}ds=\frac{\phi(\lambda)}{\lambda}\LL u(x,\phi(\lambda)).
			\end{split}\label{eq:U laplace to u}
		\end{align}
		Here, in the  last equality we used the formula for the Laplace transform of a derivative \cite[Proposition 1.6.6]{arendt2001cauchy}.
		
		Apply now the Laplace transform to the left hand side of \eqref{eq:time non-local integral form} in the variable $t$. By using the convolution property of the Laplace transform \cite[Proposition 1.6.4]{arendt2001cauchy}, we get
		\begin{align}
			b\left(\LL U(x,\lambda)-\frac{f(x)}{\lambda}\right)	&+\LL\left(\int_0^t \big(U(x,s)-U(x,0)\big)\overline \nu(t-s)ds\right)(\lambda)\notag\\
			&=b\left(\LL U(x,\lambda)-\frac{f(x)}{\lambda}\right)+\left(\LL U(x,\lambda)-\frac{f(x)}{\lambda}\right)\left(\frac{\phi(\lambda)}{\lambda}-b\right)\notag\\
			\overset{\eqref{eq:U laplace to u}}&{=}\left(\frac{\phi(\lambda)}{\lambda}\LL u(x,\phi(\lambda))-\frac{f(x)}{\lambda}\right)\frac{\phi(\lambda)}{\lambda}.\label{eq:U laplace LHS}
		\end{align}
		
		Now we take the Laplace transform of the right hand side of \eqref{eq:time non-local integral form} in the variable $t$. Before doing so, note that the Laplace transform and $-\Lpsi$ commute in the cases of the functions $U(x,t)$ and $u(x,t)$ since both functions are in  $C^2(B(x,\de(x)/2))$ and the bounds are uniform in $t>0$, so we may use Fubini's theorem to show this. For $u$, its regularity  was argued in Theorem \ref{t:solution parabolic-f} (i.e. Theorem \ref{t:solution parabolic}), while the claim for $U$ follows from \eqref{eq:U relative to u} and by the claim for $u$. It follows that the Laplace transform of the right hand side of \eqref{eq:time non-local integral form} is
		\begin{align*}
			\LL&\left(\int_0^t -\Lpsi U(x,s)ds\right)(\lambda)=\frac{1}{\lambda}\LL\big(-\Lpsi  U(x,\cdot)\big)(\lambda)=-\frac{1}{\lambda}\Lpsi  \LL U(x,\lambda)\\\overset{\eqref{eq:U laplace to u}}&{=}-\frac{\phi(\lambda)}{\lambda^2}\Lpsi \LL u(x,\phi(\lambda))=\frac{\phi(\lambda)}{\lambda^2}\LL\left(-\Lpsi u(x,\cdot)\right)\big(\phi(\lambda)\big)\\&=\frac{\phi(\lambda)}{\lambda^2}\left(\phi(\lambda)\LL u(x,\phi(\lambda))-f(x)\right),
		\end{align*} 
		which is the same as \eqref{eq:U laplace LHS}, where in the last equality we again used the formula for the Laplace transform of a derivative \cite[Proposition 1.6.6]{arendt2001cauchy}. By inverting the Laplace transform, we obtain that for almost all $t>0$, the relation \eqref{eq:time non-local integral form} holds. Moreover, by using the uniform boundedness in time of $U(x,t)$ in $C^2(B(x,\de(x)/2))$, and continuity of the paths of $L(t)$, $t\ge0$, it is easy to see that both sides of \eqref{eq:time non-local integral form} are continuous in $t$, hence \eqref{eq:time non-local integral form} holds for all $t>0$.

		To obtain $\phi(\partial_t)U(x,t)=-\Lpsi U(x,t)$, for $x\in D$ and $t>0$, it is enough to show that just one side of \eqref{eq:time non-local integral form} is differentiable. We will prove that this is the case for the right hand side by showing that the integrand $-\Lpsi U(x,s)$ is continuous in $s$. 
		
		This can be done similarly as in Theorem \ref{t:solution parabolic-f} (i.e. Theorem \ref{t:solution parabolic}) when we showed that $s\mapsto \Lpsi u(x,s)$ is continuous. First, recall that by \eqref{eq:U relative to u} we have $U(x,t)=\mathds{E}[ P^D_{L(t)}f(x)]$, so by decomposing $ P^D_{L(t)}f$ as in \eqref{eq:decomp semig R1}, (with $f=f\eta+f(1-\eta)$), and by recalling that $L(t)$ has continuous trajectories, we obtain the same regularity for $U$ as for $u$. In other words, by the same arguments as for $u$, by using the decomposition \eqref{eq:operator on PtD} on $U(x,s)$, it follows that $s\mapsto-\Lpsi U(x,s)$ is continuous,
		
		Thus, both sides of \eqref{eq:time non-local integral form} are differentiable, so by differentiating them we obtain 
		$$ \phi(\partial_t)\big(U(x,t)-U(x,0)\big)=-\Lpsi U(x,t),\quad x\in D,\,t>0.$$
	\end{proof}

	\begin{remark}[Nonuniqueness of pointwise solutions]
		One may be tempted to think that the problems \eqref{eq:parabolic problem f} and \eqref{t:solution time-non-local parabolic} for $f\equiv 0$ admit only the trivial solution $U\equiv 0$. However, a typical feature of non-local settings is that equations of the form  \eqref{eq:parabolic problem f} and \eqref{t:solution time-non-local parabolic} with $f\equiv 0$ have infinitely many (non-negative) pointwise solutions. 
		
		For instance, if $\Lpsi = (-\Delta)^{\alpha/2}$ for $\alpha\in (0,2)$, and if $D$ is a Lipschitz domain, then for each $Q\in \partial D$ there exists a smooth non-negative function, the so-called Martin kernel, $x\mapsto \eta_D(x,Q)$ which pointwise solves \eqref{eq:parabolic problem f} with $f\equiv 0$. For details, see \cite{ABR25}.
		
		To obtain uniqueness, one needs to supplement \eqref{eq:parabolic problem f} with an additional boundary condition of the \eqref{eq:parabolic problem} of the type $\lim_{x\to \partial D}U(x,t)/u^*(x)=\zeta$, where $u^*$ is certain reference function directly related to the problem, see \cite{CGCV22}. for analogous boundary behavior in non-local-in-time equations, see \cite{CGCV24}.
	\end{remark}
	
	\appendix
	
	\section{Subordinators and Bernstein functions}
	\label{appendix_bernstein}
	
	All the information given in this appendix essentially comes from \cite{bernstein}, see in particular Chapter 5 therein.
	\begin{definition}
		A function $\phi : (0,+\infty) \to [0,+\infty)$ is called Bernstein function if $\phi\in C^{\infty}(0,+\infty)$ and
		$$
		(-1)^{n-1} \phi^{(n)}(\lambda) \ge 0, \qquad n \in \mathbb{N} ,\, \lambda > 0.
		$$
	\end{definition}
	
	The Bernstein functions are characterized by the following theorem. 
	
	\begin{theorem}
		A function $\phi : (0,{+\infty}) \to [0,+\infty)$ is a Bernstein function if and only if it holds that
		\begin{align}
			\phi(\lambda) = a + b \lambda + \int_{(0,{+\infty})} (1 - e^{-\lambda t}) \,\nu(dt),
			\label{A1002}
		\end{align}
		where $a,b \ge 0$ and $\nu$ is a Borel measure on $(0,{+\infty})$ that satisfies $\int_{(0,{+\infty})} (1 \wedge t)\,\nu(dt) < {+\infty}$.
		
		Moreover, the triplet $(a,b,\nu)$ uniquely determines $ \phi$.
	\end{theorem}
	
	In the triplet $(a,b,\nu)$, $a$ is usually called the killing term, $b$ is called the drift, and $\nu$ the L\'evy measure of $\phi$. The reason for this nomenclature lies in the fact that Bernstein functions are in 1--1 correspondence with subordinators.

	\begin{definition}
		A càdlàg process $S = (S(t))_{t \ge 0}$ defined on a probability space $(\Omega,\mathcal{F},\mathbb{P})$ is called a subordinator if it is non-negative, has independent and stationary increments,  and $S_0 = 0$.
	\end{definition}
	It is easily observed that paths of a subordinator $S$ are non-decreasing functions. In other words, subordinators are non-decreasing L\'evy processes.
	
	Let $E_a$ be an exponentially distributed random variable with rate parameter $a \ge 0$, i.e. $
	\mathbb{P}(E_a \le  t) = 1- e^{-a t}$, $t\ge0$. In the case $a = 0$, we set $E_a = +\infty$. Assume that $E_a$ and the subordinator $(S(t))_{t\ge0}$ are independent, and define a process $({S}^\dagger(t))_{t \ge 0}$ by
	$$
	{S}^\dagger(t) :=
	\begin{cases}
		S(t), & t < E_a, \\
		+\infty, & t \ge E_a.
	\end{cases}
	$$
	
	The process $S^\dagger$ is the subordinator $S$ killed at an independent exponential time $E_a$ and such a process is called a  {killed subordinator}.
	
	The following theorem gives the 1--1 correspondence between killed subordinators and Bernstein functions.
	
	\begin{theorem}
		If $S^\dagger$ is a killed subordinator, then there exists a Bernstein function $\phi$ such that
		\begin{align}\label{eq:1741}
			\EE e^{-\lambda S^\dagger(t)}=e^{-t \phi(\lambda)}.
		\end{align}
		Conversely, if $\phi$ is a Bernstein function, then there exists a killed subordinator $S^\dagger$ such that \eqref{eq:1741} holds.
	\end{theorem}
	
	We also mention that each non-killed subordinator $S$ with corresponding triplet $(0,b,\nu)$ can be written as $S(t)=bt+\wt S(t)$, where $\wt S$ is a subordinator with  triplet $(0,0,\nu)$ (also called a pure-jump subordinator with the L\'evy measure $\nu$). In particular, if the subordinator $S$ has  triplet $(0,0,\nu)$ and $\nu(0,+\infty)\eqqcolon\lambda<+\infty$, then $S$ is called a compound Poisson process since after each of independent exponential time with rate $\lambda$, $S$ jumps by following the distribution $\nu(ds)/\nu(0,+\infty)$.
	
	Note also that $S$ is strictly increasing if and only if $b>0$ or $\nu(0,+\infty)=+\infty$. If $S$ is a proper killed subordinator, i.e. $a>0$ and let $E_a$ be its killing exponential, then its inverse subordinator $L(t)=\inf\{s>0:S(s)>t\}$ remains trapped in the position $E_a$ for all $t\ge S(E_a)$, since we treat the killing of a subordinator as sending it to the value $+\infty$. Some authors, on the other hand, sometimes also consider that the inverse $L$ is also killed after the time $S(E_a)$.

	The best known subordinator is the so-called stable subordinator. Its Bernstein function is $\phi(\lambda)=\lambda^{\alpha}$, for some $\alpha\in (0,1)$, with the corresponding triplet $(0,0,\nu)$ with $\nu(dt)=\frac{\alpha}{\Gamma(1-\alpha)}t^{-1-\alpha}dt$.
	
	One particularly interesting subclass of Bernstein functions are called special Bernstein functions. The related subordinators are called special subordinators. Here is the definition and its characterization, for details see \cite[Chapter 13]{bernstein}.
	
	\begin{definition}
		A Bernstein function $\phi$ is called a special Bernstein function if $\phi^*(\lambda)\coloneqq \frac{\lambda}{\phi(\lambda)}$, $\lambda>0$, is also a Bernstein function.
		
		A subordinator $S$ is special subordinator if its corresponding Bernstein function is special.
	\end{definition}
	Obviously, stable subordinators are special. To state the  characterization theorem, note that by $U$ we denote the potential measure of a subordinator $S$, given by the relation $\EE\int_0^{+\infty} f(S(t))dt=\int_0^{+\infty} f(t) U(dt)$, $f\ge0$.
	\begin{theorem}
		Let $S$ be a subordinator with potential measure $U$. Then $S$ is special if, and only if,
		$$
		U(dt) = c \delta_0(dt) + u(t)\,dt 
		$$
		for some $c \ge 0$ and some non-increasing function $u : (0,{+\infty}) \to (0,+\infty)$ satisfying $
		\int_0^1 u(t)\,dt < +\infty.$
	\end{theorem}
	
	Lastly, we bring a generalization of \cite[Lemma A.3]{daniel24} and \cite[Lemma A.1]{Jorge2023a} to include also the case when the subordinator has a non-zero killing term.
	\begin{lemma}\label{daniel}
		Let $\phi$ be a Bernstein function with triplet $(a,b,\nu)$, $S$ its corresponding (killed) subordinator, and $L$ the inverse of $S$. Then for all $\lambda >0$ there exists $C_1,C_2>0$ such that
		\begin{align}
			\EE[e^{\lambda L(t)}]\le C_1e^{C_2 t},\quad t>0. 
		\end{align}
	\end{lemma}
	\begin{proof}
		If $a=0$, the claim is known, see \cite[Lemma A.3]{daniel24}.
		
		Let now $a>0$, and write the subordinator $S$ as $S(t)=\wt S(t)\1_{E_a<t}+\infty\1_{E_a\ge t}$, where $\wt S$ is the subordinator with triplet $(0,b,\nu)$ and $E_a$ is an independent exponential random variable with rate $a>0$. Denote by $\wt L$ the inverse of $\wt S$. Then it holds that $L(t)=\wt L(t)\wedge E_a$ and $\wt L$ and $E_a$ are independent. Hence,
		\begin{align}
			\EE[e^{\lambda L(t)}]&=\EE[e^{\lambda (\wt L(t)\wedge E_a)}]\le\EE[e^{\lambda (\wt L(t)}],
		\end{align}
		and now the claim follows from \cite[Lemma A.3]{daniel24}.
	\end{proof}

	\section{Feller semigroups and Bochner subordination}
	\label{appendix_semigroups}
	We direct the reader to classical monographs, such as \cite{hille1996functional}, for a full (very general) overview of the theory of semigroups and to \cite{schillinglevy} for a more modern and accessible treatment with a probabilistic approach. 
	
	A strongly continuous semigroup of operators on a Banach space $(\mathfrak{B}, \left\| \cdot \right\|)$ is a family of bounded linear operators $P_t$, $t \geq 0$, such that $P_t:\mathfrak{B} \mapsto \mathfrak{B}$, which is strongly continuous at $t=0$, i.e., $\left\| P_tu-u \right\| \to 0$ as $t \to 0$ for all $u\in\mathfrak{B}$, and has the semigroup property $P_sP_t=P_tP_s=P_{t+s}$, for all $s,t \geq 0$. If, furthermore, $\left\| P_tu \right\| \leq \left\|u \right\|$, the family is called a \emph{strongly continuous contraction semigroup}. The (infinitesimal) generator of the semigroup $P_t$, $t\ge0$, is the operator $(G, \mathcal{D}(G))$
	\begin{align}
		&  Gu \coloneqq \text{strong-}\lim_{t \to 0} \frac{P_tu-u}{t},\quad u\in \DD(G), \\
		& \mathcal{D}(G) \coloneqq \left\lbrace u \in \mathfrak{B}: \lim_{t \to 0} \frac{P_tu-u}{t} \text{ exists as the strong limit} \right\rbrace.
	\end{align}
	The generator of a strongly continuous semigroup is a closed
	and densely defined linear operator that determines the semigroup uniquely. 
	
	For every strongly continuous semigroup there exist constants $M >0$ and $\lambda_0\geq0$ such that
	\begin{align}
		\left\|  P_t u \right\| \, \leq \, M e^{\lambda_0 t} \left\| u \right\|,
	\end{align}
	for all $u\in \mathfrak{B}$.
	Therefore, the $\lambda$-potential operator on $\mathfrak{B}$
	\begin{align}
		U_\lambda u \, = \, \int_0^{+\infty} e^{-\lambda t} P_tu \, dt
	\end{align}
	is well defined for $\mathfrak{Re}(\lambda) > \lambda_0$. Furthermore, every $\lambda$ such that $\mathfrak{Re} (\lambda) > \lambda_0$ belongs to the resolvent set of $G$ denoted by $\rho(G)$. In other words, for $\lambda \in \rho (G)$ it holds that $U_\lambda = (\lambda - G)^{-1}$, i.e., the operator $U_\lambda$ is the resolvent of $G$.
	
	Bochner's subordination provides a method to construct a strongly continuous contraction semigroup given another one (see \cite[Chapter 15]{bernstein} for a full overview of the theory). In particular, let $P_t$, $t \geq 0$, be a strongly continuous contraction semigroup, let $S(t)$, $t \geq 0$, be a (killed) subordinator with the Laplace exponent $\phi$ as in \eqref{A1002} and let $\mu_t(\cdot)$ be subordinator's law. Then, the family of operators $P_t^\phi$, $t \geq 0$, given by
	\begin{align}
		P_t^\phi u \, = \, \int_0^{+\infty} P_su \, \mu_t(ds)
		\label{b0952}
	\end{align}
	forms (another) strongly continuous contraction semigroup on the same Banach space, where the integral \eqref{b0952} is understood in the Bochner sense. Furthermore, if $(G, \mathcal{D}(G))$ is the generator of $P_t$, $t \geq 0$, then $P_t^\phi$, $t \geq 0$, is generated by $(G^\phi, \mathcal{D}(G^\phi))$ such that
	\begin{align}
		G^\phi u \mid_{\mathcal{D}(G)} \, = \, au + bG u + \int_0^{+\infty} \left( P_s u - u \right) \nu (ds),
	\end{align}
	where $\mathcal{D}(G)$ is an operator core for $(G^\phi, \mathcal{D}(G^\phi))$. It turns our that, in suitable sense, $G^\phi u = - \phi (-G)$ (this is due to \cite{Schilling1998SubordinationBochner}).
	
	Feller semigroups are strongly continuous contraction semigroup on the Banach space $(C_0(E), \left\| \cdot \right\|_\infty)$ of continuous functions vanishing on $E$ at infinity, where $E$ is a locally compact and separable space, with the additional property of being positivity preserving, i.e., $u \geq 0 \implies P_u \geq 0$. We recall that vanishing at infinity, in this general setting, means that
	\begin{align}
		\forall \epsilon >0, \, \exists K \subset E \text{ compact s.t. } |u(x)| \leq \epsilon, \, \forall x \in E\setminus K.
	\end{align}
	In this case, one can use the Riesz representation theorem (see, e.g., \cite[Theorem 1.5]{schillinglevy}) to say that
	\begin{align}
		P_tu (x) \, = \, \int_E u(y) \, p_t(x, dy)
	\end{align}
	where $p_t (\cdot, \cdot)$ is a kernel of (sub-)probability measures satisfying the Chapman-Kolmogorov equations. With this at hand, it is clear that following the Kolmogorov procedure to construct the canonical process, it is always possible to associate the semigroup $P_t$, $t \geq 0$, on $C_0(E)$ with a Markov process $M(t)$, $t \geq0$, on $E$ as
	\begin{align}
		\mathds{E}^x u(M(t)) \, = \, P_tu(x)
	\end{align}
	where $\mathds{E}^x$ is the (unique) probability measure such that $\mathds{P}^x (M(0)=x)=1$. Such a process $M(t)$, $t \geq0$, is called a Feller process, and it always admits a càdlàg modification. Moreover, it is in fact a strong Markov process that is also quasi-left continuous, for details we refer the reader to \cite[Chapter III, Sections 2 and 3]{revuz2013continuous}.
	
	\section*{Acknowledgments}
	The authors acknowledge financial support under the National Recovery and Resilience Plan (NRRP), Mission 4, Component 2, Investment 1.1, Call for tender No. 104 published on 2.2.2022 by the Italian Ministry of University and Research (MUR), funded by the European Union – NextGenerationEU– Project Title “Non–Markovian Dynamics and Non-local Equations” – 202277N5H9 - CUP: D53D23005670006 - Grant Assignment Decree No. 973 adopted on June 30, 2023, by the Italian Ministry of University and Research (MUR).
	
	The author B. Toaldo would like to thank the Isaac Newton Institute for Mathematical Sciences, Cambridge, for support and hospitality during the programme Stochastic systems for anomalous diffusion, where work on this paper was undertaken. This work was supported by EPSRC grant EP/Z000580/1.
	
	The author I. Biočić acknowledges financial support by the European Union – NextGenerationEU through the National Recovery and Resilience Plan 2021-2026 Institutional grant of University of Zagreb Faculty of Science (IK IA 1.1.3. Impact4Math), as well as the support by Croatian Science Foundation through the project IP-2025-02-8793.
	
	The authors are very grateful to Giacomo Ascione, Enrico Scalas, and Francesco Virgili for their valuable comments, which have considerably improved a previous version of the manuscript.
	
	\bibliographystyle{abbrv}
	
	\bibliography{References}

	\newpage
	
	\noindent{\bf Ivan Bio\v{c}i\'c}
	
	\noindent Department of Mathematics, Faculty of Science, University of Zagreb, Zagreb, Croatia,
	
	\noindent Department of Mathematics “Giuseppe Peano”, University of Turin, Turin, Italy,
	
	\noindent Email:  \texttt{ivan.biocic@math.hr},
	\texttt{ivan.biocic@unito.it}
	
	\bigskip

	\noindent{\bf Bruno Toaldo}
	
	\noindent Department of Mathematics “Giuseppe Peano”, University of Turin, Turin, Italy,
	
	\noindent Email: \texttt{bruno.toaldo@unito.it}

\end{document}